\documentclass[moor]{informs1} 

\usepackage[numbers,sort]{natbib}
 \NatBibNumeric
 \def\bibfont{\small}%
 \bibpunct[, ]{[}{]}{,}{n}{}{,}%
 
\usepackage[colorlinks=true,breaklinks=true,bookmarks=true,urlcolor=blue, citecolor=blue,linkcolor=blue,bookmarksopen=false,draft=false]{hyperref}

\def\EMAIL#1{\href{mailto:#1}{#1}} 

\let\theoremstyle\relax 

\usepackage[T1]{fontenc}
\usepackage{geometry}
\usepackage{amsmath}
\usepackage{amsthm}
\usepackage{amssymb}
\usepackage{array}
\usepackage{float}
\usepackage{booktabs}
\usepackage{tablefootnote}

\makeatletter
\renewenvironment{proof}[1][\proofname]{\par
\pushQED{\qed}%
\normalfont \topsep6\p@\@plus6\p@\relax
\trivlist
\item\relax
{\itshape
#1\@addpunct{.}\enspace}\hspace\labelsep\ignorespaces
}{%
\popQED\endtrivlist\@endpefalse
}
\makeatother

\usepackage{amsmath}
\usepackage{amsfonts}
\usepackage{latexsym}
\usepackage{amssymb}

\newtheorem{thm}{Theorem}[section]
\newtheorem{theorem}[thm]{Theorem}
\newtheorem{lem}[thm]{Lemma}
\newtheorem{lemma}[thm]{Lemma}
\newtheorem{prop}[thm]{Proposition}
\newtheorem{proposition}[thm]{Proposition}

\newtheorem{definition}[thm]{Definition}

\newtheorem{remark}[thm]{Remark}

\newcommand{\beq}{\begin{equation}}
\newcommand{\eeq}{\end{equation}}
\newcommand{\beqa}{\begin{eqnarray}}
\newcommand{\eeqa}{\end{eqnarray}}
\newcommand{\beqas}{\begin{eqnarray*}}
\newcommand{\eeqas}{\end{eqnarray*}}
\newcommand{\ei}{\end{itemize}}

\newcommand{\R}{\mathbb{R}}

\newcommand{\lam}{{\lambda}}

\newcommand{\cK}{{\cal K}}

\newcommand{\inner}[2]{\langle #1,#2\rangle}

\newcommand{\inte}{\mathrm{int}\,}

\newcommand{\dom}{\mathrm{dom}\,}

\usepackage{blindtext}
\usepackage{url} 
\usepackage{enumerate} 
\usepackage{setspace}
\usepackage{algorithm,algpseudocode}
\usepackage{xcolor}
\usepackage[title]{appendix}

\usepackage{lmodern}

\usepackage{multicol}
\usepackage{hyperref}
\hypersetup{
    colorlinks,
    citecolor=blue,
    filecolor=blue,
    linkcolor=green,
}
\definecolor{green}{rgb}{0,0.6,0.0} 

\global\long\def\cK{{\mathcal K}}
\global\long\def\cH{{\mathcal H}}

\global\long\def\rn{\Re^{n}}%
\global\long\def\R{\Re}%
\global\long\def\r{\Re}%
\global\long\def\pt{\mathbb{\partial}}%
\global\long\def\lam{\lambda}%
\global\long\def\argmin{\operatorname*{argmin}}%
\global\long\def\dom{\operatorname*{dom}}%
\global\long\def\inner#1#2{\langle#1,#2\rangle}%
\global\long\def\cConv{\overline{{\rm Conv}}\ }%
\global\long\def\intr{\operatorname*{int}}%
\begin{document}


\RUNAUTHOR{Kong, Melo, and Monteiro}

\RUNTITLE{An NL-IAPIAL Method for Cone-Constrained Composite Optimization}

\TITLE{Iteration complexity of a proximal augmented Lagrangian method for solving nonconvex composite optimization problems with nonlinear convex constraints}


\ARTICLEAUTHORS{%
\AUTHOR{Weiwei Kong}
\AFF{Computer Science and Mathematics Division, Oak Ridge National Laboratory, Oak Ridge, TN, 37830. \EMAIL{wwkong92@gmail.com}}
\AUTHOR{Jefferson G. Melo}
\AFF{Instituto de Matem\'atica e Estat\'istica, Universidade Federal de Goi\'as, Campus II- Caixa Postal 131, CEP 74001-970, Goi\^ania-GO, Brazil. \EMAIL{jefferson@ufg.br}}
\AUTHOR{Renato D.C. Monteiro}
\AFF{School of Industrial and Systems Engineering, Georgia Institute of
Technology, Atlanta, GA, 30332-0205. \EMAIL{monteiro@isye.gatech.edu}}
} 
	
\ABSTRACT{%
This paper proposes and analyzes  a proximal augmented Lagrangian  (NL-IAPIAL)  method for solving  smooth nonconvex composite optimization problems with nonlinear $\cK$-convex constraints, i.e., the constraints are convex with respect to the order given by a  closed   convex cone $\cK$. Each NL-IAPIAL iteration consists of inexactly solving a proximal augmented Lagrangian subproblem  by an accelerated composite gradient (ACG) method  followed by a  Lagrange multiplier update. Under some mild assumptions, it is shown that NL-IAPIAL generates an approximate  stationary solution of the constrained problem in ${\cal O}(\log(1/\rho)/\rho^{3})$ inner iterations,  where $\rho>0$ is a given tolerance. Numerical experiments are also given to illustrate the computational efficiency of the proposed method.
}%


\MSCCLASS{%
49M05, 49M37, 90C26, 90C30, 90C60, 65K05, 65K10, 68Q25, 65Y20.
}%

\ORMSCLASS{Primary: programming: nonlinear: theory, algorithms; Secondary: programming: nonlinear: nondifferentiable; mathematics: convexity}


\makeatletter
\maketitle

\section{Introduction} \label{sec:intro}

This paper presents a
nonlinear inner-accelerated proximal inexact
augmented Lagrangian (NL-IAPIAL) method for solving the
cone convex constrained nonconvex composite optimization (CCC-NCO) problem
\begin{equation}
\phi^{*}:=\inf_{z\in\rn}\left\{ \phi(z):=f(z)+h(z):g(z)\preceq_{\cK}0\right\} ,\label{eq:main_prb}
\end{equation}
where ${\cal K}$ is a closed convex cone such
that $\emptyset\neq{\cal K}\neq \r^\ell$,  $g:\rn\mapsto\r^\ell$
is a differentiable $\cK$-convex function with
a Lipschitz continuous
gradient,
$h$ is a proper closed convex function with compact domain,
 $f$ is a nonconvex differentiable function on the domain of $h$ with a
Lipschitz continuous gradient, and the relation $g(z)\preceq_{\cK} 0$ means $g(z)\in-\cK$.

More specifically, the NL-IAPIAL method
is based on the 
 augmented Lagrangian (AL) (see \cite{zhaosongAugLag18} and \cite[Section 11.K]{VariaAna})
\begin{align}
{\cal L}_{\beta}(z,p) 
& := (f+h)(z)+\frac{1}{2\beta}\left[{\rm dist}^{2}(p+\beta g(z),-{\cal K})-\|p\|^{2}\right] \quad \forall \beta > 0, \label{eq:aug_lagr_def}
\end{align}
where ${\rm dist}(y,S)$ denotes the Euclidean distance
between a point $y\in\Re^\ell$ and a set $S\subseteq\Re^\ell$.
It performs the following proximal point-type update to generate its $k$-th iterate: given $(z_{k-1},p_{k-1})$ and $(\lam_k,\beta_k)$, compute
\begin{align}
z_{k} & \approx\argmin_{u}\left\{ \lam_k {\cal L}_{\beta_k}(u;p_{k-1})+\frac{1}{2}\|u-z_{k-1}\|^{2}\right\} ,\label{eq:approx_primal_update}\\
p_{k} & =\Pi_{\cK^{*}}(p_{k-1}+{\beta_k}g(z_{k})),\label{eq:dual_update}
\end{align}
where $\cK^{*}$ denotes the dual cone of $\cK$,  the function $\Pi_{\cK^{*}}$
denotes the projection onto ${\cal K}^{*}$, and $z_{k}$ 
is a suitable approximate solution of the composite problem underlying
\eqref{eq:approx_primal_update}. Even though there are different approaches for obtaining $z_k$ as in \eqref{eq:approx_primal_update}, NL-IAPIAL employs an accelerated composite gradient (ACG) algorithm to obtain it, and hence the ``inner-accelerated'' qualifier in its name. 
Moreover, at the end of the $k$-th iteration above, it performs a key test to decide whether $\beta_k$ is left unchanged or doubled.

Under a Slater-like assumption\footnote{See Proposition~\ref{prop:weak_slater} in view of assumption (A4) in Subsection~\ref{subsec:prb_of_interest}.} and a suitable choice of the inputs $(\lam,\beta)$, it is shown that for any $(\hat{\rho},\hat{\eta})\in\r_{++}^{2}$, the NL-IAPIAL method obtains a near stationary solution, i.e.,
a quadruple $(\hat{z},\hat{p},\hat{w}, \hat{q})$
satisfying 
\begin{align}
\hat{w}\in\nabla f(\hat{z})+\pt h(\hat{z})+\nabla g(\hat{z})\hat{p},& \quad \inner{g(\hat z) + \hat q}{\hat p} = 0, \quad  g(\hat z) + \hat q \preceq_{\cK} 0, \quad\hat{p}\succeq_{\cK^{*}} 0 \label{eq:rho_eta_approx_soln}
\\
&\|\hat{w}\|\leq\hat{\rho}, \quad \|\hat{q}\|\leq \hat{\eta},\label{ineq:rho_eta_approx_soln}
\end{align}
in ${\cal O}((\hat{\eta}^{-1/2}\hat{\rho}^{-2}+\hat{\rho}^{-3})\log(\hat{\rho}^{-1} + \hat\eta^{-1}))$
ACG iterations. 
If \eqref{eq:main_prb} satisfies a certain regularity condition, then it is well-known that a necessary condition for a point $\hat z$ to be a local minimum of \eqref{eq:main_prb} is that there exists a  multiplier $\hat p\in\cK^*$ such that $(\hat z,\hat p,\hat q,\hat w) = (\hat z,\hat p,0,0)$  satisfies \eqref{eq:rho_eta_approx_soln}.
Moreover, the aforementioned complexity bound
is derived without assuming that the initial point
$z_0 \in \dom h$ is feasible, i.e., it
also satisfies
$g(z_0) \preceq_{\cK} 0$.
A key fact derived in this work is that the
sequence of Lagrange multipliers generated by NL-IAPIAL
is bounded,
and its proof strongly uses the fact that
its constraint function $g$ is $\cK$-convex
(although \eqref{eq:main_prb} is nonconvex due to
the nonconvexity assumption on $f$).
\\

\textit{Overview of AL methods}. The discussion below separates the
AL methods into two classes:
\begin{itemize}

\item[(i)] \textit{Proximal AL (PAL) methods} whose $k$-th iteration is: 
given a pair $(z_{k-1},p_{k-1})$ and a penalty parameter $\beta_k$, choose a prox parameter $\lam_k$ such that the objective function of \eqref{eq:approx_primal_update} is strongly convex,
    compute an approximate solution $z_k$ of
   \eqref{eq:approx_primal_update}, set
   \begin{gather}
p_{k}= (1-\theta) \Pi_{{\cal K}^{*}}(p_{k-1}+\chi_{k}\beta_{k}g(z_{k})) \label{eq:gen_dual_update2}
\end{gather}
for some $\chi_k\in (0,1]$ and fixed $\theta\in[0,1)$, and choose the next penalty parameter $\beta_{k+1}$ from $[\beta_k,\infty)$.
A classical PAL method for the case where $f$ is convex has been studied by Rockafellar \cite{rockafellar1976augmented} under the assumption that $\theta = 0$, $\chi_k = 1$, and $\lam_k=\beta_k$ for every $k$. It is worth noting that
when $f$ is convex,
his method, as well as the aforementioned PAL method, can be
viewed as a primal-dual, variable stepsize,
inexact proximal point method, i.e, one which inexactly solves
\begin{equation}
\partial_z {\cal L}_0(z;p) + \frac1{\lam_k} (z-z_{k-1}) \ni 0,  \quad - \partial_p {\cal L}_0(z;p) + \frac{1}{\chi_k\beta_k} (p-p_{k-1}) \ni 0, \label{eq:pal_update}
\end{equation}
for $(z,p)=(z_k,p_k)$
where
${\cal L}_0(z;p) := (f+h)(z) + \inner{p}{g(z)} - \delta_{{\cal K}^*}(p)$, for every $(z,p) \in \Re^n \times \Re^\ell$ with
the convention that $+\infty-\infty=+\infty$, and $\delta_{{\cal K}^*}(p)$ takes value $0$ if $p\succeq_{{\cal K}^*} 0$ and $+\infty$ otherwise.
Note that system \eqref{eq:pal_update} is equivalent to
\[
\nabla f(z) + \partial h(z) + \nabla g(z) p + \frac1{\lam_k} (z-z_{k-1}) \ni 0,  \quad - g(z) + \partial \delta_{\cal K^*}(p) + \frac{1}{\chi_k\beta_k} (p-p_{k-1}) \ni 0.
\]

\item[(ii)] \textit{Non-proximal AL (n-PAL) methods} whose $k$-th iteration is: given a pair $(z_{k-1},p_{k-1})$ and a penalty parameter $\beta_{k}$,
    compute an approximate stationary point $z_k$ of ${\cal L}_{\beta_k}(\cdot;p_{k-1})$, set
    \begin{gather}
    p_{k}=\Pi_{{\cal K}^{*}}(p_{k-1}+\chi_{k}\beta_{k}g(z_{k})) \label{eq:gen_dual_update}
    \end{gather}
    for some $\chi_k\in(0,1]$, and choose the next penalty parameter $\beta_{k+1}$ from $[\beta_{k},\infty)$.
Detailed discussion of dual-only methods can be found, for example,
in \cite{bertsekas2016nonlinear} where the conditions 
$\beta_{k} > \beta_{k-1} > 0$ for all $k \ge 1$
and $\beta_k \uparrow \infty$ are assumed,
and in \cite{fletcher2013practical,nocedal2006numerical} where
$\beta_{k} = \beta_{k-1}$ is allowed at iterations
for which the feasibility gap decreases sufficiently.
It is worth noting that when $f$ is convex, these methods can be viewed as a dual-only,
variable stepsize, inexact proximal point method for the same
operator above, i.e., one which
inexactly solves
\begin{equation}
\partial_z {\cal L}_0(z;p) \ni 0,  \quad - \partial_p {\cal L}_0(z;p) + \frac{1}{\chi_k\beta_k} (p-p_{k-1}) \ni 0, \label{eq:n-pal_update}
\end{equation}
for $(z,p)=(z_k,p_k)$ and ${\cal L}_0(\cdot;\cdot)$ is as in (i).
\end{itemize}

Notice how both kinds of AL methods include a prox term in the $p$ block,
which leads to the multiplier update \eqref{eq:gen_dual_update}.
However, while the first one adds a proximal term to the $z$-block (hence
the qualifier PAL), the other ones do not (hence the qualifier n-PAL). For a more detailed comparison of the above classes, see the first paragraph in Section~\ref{sec:conclusion}.\\

\emph{Related works}.
The literature of AL-based methods is quite vast, so we focus our attention on those dealing with iteration complexities.
Since AL-based methods for the convex case have been extensively studied in the literature (see, for example, \cite{LanRen2013PenMet, Aybatpenalty, IterComplConicprog, AybatAugLag, LanMonteiroAugLag, ShiqiaMaAugLag16, zhaosongAugLag18, Patrascu2017, YangyangAugLag17}),
we focus on papers that deal with nonconvex problems with nontrivial composite functions Methods for the nonconvex problems where the composite $h$ is the zero function have already been studied in \cite{ProxAugLag_Ming,xie2019complexity}.

Papers \cite{HongPertAugLag,RenWilmelo2020iteration,RJWIPAAL2020} as well as this one propose and study the complexity of PAL methods
for solving the CCC-NCO problem or its linearly constrained version
in which ${\cal K}=\{0\}$.
More specifically, both papers \cite{HongPertAugLag,RJWIPAAL2020} consider
PAL methods applied to the linearly constrained CCC-NCO problem where $\theta \in (0,1]$ and
$\chi_k=1$ for every $k$. However, as $\theta$ approaches zero, the
prox stepsizes $\lam_k$ of both methods converge to zero
which causes the following issues:
1) their derived complexity bounds 
diverge to infinity (see
the second column in Table~\ref{tab:summary_tbl} below), which makes their analyses invalid for
the case where $\theta=0$; and 2) deteriorating computational performance.
Using a different approach, i.e., one that does not rely on a merit function, paper \cite{RenWilmelo2020iteration} establishes the iteration complexity of a PAL method, with $\theta=0$ and $\chi_k=1$ for every $k$,
for solving the
linearly constrained CCC-NCO problem under
the condition that $p_k$ is reset to zero whenever $\beta_k$ is increased.

Papers \cite{ImprovedShrinkingALM20, inexactAugLag19}
propose and study the iteration complexity of n-PAL methods
for solving nonlinearly constrained NCO problems.
More specifically, \cite{inexactAugLag19} uses the AG method of \cite{nonconv_lan16} to obtain the approximate
stationary point $z_k$ of ${\cal L}_{\beta_k}(\cdot;p_{k-1})$. On the other hand, \cite{ImprovedShrinkingALM20} obtains
such $z_k$ by applying 
an inner accelerated prox method as in
\cite{Aaronetal2017,WJRproxmet1} whose generated subproblems are convex
and similar to the ones generated by the PAL methods.
It is worth mentioning that both of these papers make a strong assumption about how the feasibility of an iterate is related to its stationarity (see condition ${\cal F}$ in Table~\ref{tab:conditions}).

We now describe other papers that have motivated this work or are tangentially related to it. Papers  \cite{WJRproxmet1, WJRComputQPAIPP, MinMax-RenWilliam, PPmetNonconvex2019} establish the complexity of quadratic penalty-based methods for solving \eqref{eq:main_prb}.
Paper \cite{Lan-ConstrainedStocasticProxMetNonconvex2019} considers a primal-dual proximal point scheme and analyzes its complexity under strong conditions on the initial point. Papers \cite{ErrorBoundJzhang-ZQLuo2020,ADMMJzhang-ZQLuo2020}  present a  primal-dual first-order algorithm for solving \eqref{eq:main_prb} when $h$ is the indicator function of a box (in \cite{ADMMJzhang-ZQLuo2020}) or more generally a polyhedron (in \cite{ErrorBoundJzhang-ZQLuo2020}).
Paper \cite{SZhang-Pen-admm} considers a penalty-ADMM method that solves an equivalent reformulation of \eqref{eq:main_prb}. 
Paper \cite{HybridPenaltyAugLag19} presents an inexact proximal point method applied to the function defined as $\phi(z)$ if $z$ is feasible and $+\infty$ otherwise. 
It can be viewed as an extension to the nonconvex setting
of the proximal point method (PPM) applied to
\eqref{eq:main_prb} (see, for example, \cite{rockafellar1976augmented} for the analysis of inexact versions of  PPMs for solving \eqref{eq:main_prb} in the convex setting).

Before closing this literature review, we list the assumptions of the above PAL and n-PAL methods in Table~\ref{tab:conditions} and give a summary of these methods in
Table~\ref{tab:summary_tbl}, which compares some of the more recent methods in terms of iteration complexity, type of constraints, necessary
conditions, and ranges of $\theta$ and $\chi_k$. 

\begin{table}[!tbh]
\begin{centering}
\begin{tabular}{c>{\raggedright}m{0.65\columnwidth}}
\toprule 
{\scriptsize{}${\cal B}$} & {\scriptsize{}Either (i) the quantity $\sup_{x\in\dom h}|\phi(x)|$
is finite, (ii) $\dom h$ is bounded, and/or (iii) the feasible set is bounded.}\tabularnewline
\midrule
{\scriptsize{}${\cal A}$} & {\scriptsize{}If the constraints have an affine component of the form $Ax=b$ then $A$ has full row rank.}\tabularnewline
\midrule
{\scriptsize{}${\cal F}$} & {\scriptsize{}There exists some $\nu>0$ such that $\nu\|g(x_{k})\|\leq{\rm dist}(0,\nabla g(x_{k})g(x_{k})+\beta_{k}^{-1}\pt h(x_{k}))$
for algorithmically generated sequences $\{x_{k}\}_{k\geq1}$ and
$\{\beta_{k}\}_{k\geq1}$.}\tabularnewline
\midrule
{\scriptsize{}${\cal N}$} & {\scriptsize{}
The function $h$ restricted 
to its domain is $r$-Lipschitz continuous.}\tabularnewline
\midrule
{\scriptsize{}${\cal SP}$} & {\scriptsize{}If $g(x)\preceq_{{\cal K}}0$ can be divided into $g_{e}(x)=0$
and $g_{\iota}(x)\preceq_{{\cal J}}0$ for some closed convex cone
${\cal J}$, then there exists $\bar{x}\in{\rm int}(\dom h)$ such
that $g_{e}(x)=0$ and $g_{\iota}(x)\prec_{{\cal J}}0$ .}\tabularnewline
\bottomrule
\end{tabular}
\par\end{centering}
\caption{Abbreviations for common boundedness and regularity conditions. A
discussion of the relationship between ${\cal SP}$ and ${\cal SP}^{\circ}$
is given in Subsection~\ref{subsec:prb_of_interest}. It is known (see, for example,
\cite{RenWilmelo2020iteration}) that ${\cal N}$ is equivalent to requiring that, for every $x\in\dom h$, there exists $r>0$ such that $\pt h(x)\subseteq{\cal N}_{\dom h}(x)+{\cal B}_{r}(0)$ where
${\cal B}_{r}(0)=\{x:\|x\| \leq r\}$. 
\label{tab:conditions}}
\end{table}

\begin{table}[!tbh]
\begin{centering}
\begin{tabular}{ccccccc}
\toprule 
\textbf{\scriptsize{}Name} & \textbf{\scriptsize{}Complexity} & \textbf{\scriptsize{}Constraints} & \textbf{\scriptsize{}$\theta$} & \textbf{\scriptsize{}$\chi_k$} & \textbf{\scriptsize{}Key Conditions} & \textbf{\scriptsize{}AL group}\tabularnewline
\midrule
{\scriptsize{}PProx-PDA\tablefootnote{This method generates prox subproblems of the form $\argmin_{x\in X}\{\lam h(x) + c\|Ax-b\|^2 / 2 + \|x-x_0\|^2 / 2 \}$ and the analysis of \cite{HongPertAugLag}
makes the strong assumption that they can be solved exactly for any $x_0$, $c$, and $\lam$.} \cite{HongPertAugLag}} & {\scriptsize{}${\cal O}(\theta^{-2}\varepsilon^{-4})$} & {\scriptsize{}Linear} & {\scriptsize{}$(0,1)$} & {\scriptsize{}$1$} & {\scriptsize{}${\cal B}$, ${\cal A}$} & {\scriptsize{}PAL} \tabularnewline
{\scriptsize{}$\theta$-IPAAL\tablefootnote{It is also shown that conditions $\cal N$ and $\cal SP$ can be removed to yield an iteration complexity of $\tilde{{\cal O}}(\theta^{-4}\varepsilon^{-3})$.} \scriptsize{} \cite{RJWIPAAL2020}} & {\scriptsize{}$\tilde{{\cal O}}(\theta^{-15/4}\varepsilon^{-2.5})$} & {\scriptsize{}Linear} & {\scriptsize{}$(0,1)$} & {\scriptsize{}$1$} & {\scriptsize{}${\cal N},{\cal SP}$} & {\scriptsize{}PAL} \tabularnewline
{\scriptsize{}IAIPAL \cite{RenWilmelo2020iteration}} & \textbf{\scriptsize{}$\tilde{{\cal O}}(\varepsilon^{-3})$} & {\scriptsize{}Linear} & {\scriptsize{}$0$} & {\scriptsize{}$1$} & \textbf{\scriptsize{}${\cal B},{\cal N},{\cal SP}$} & {\scriptsize{}PAL} \tabularnewline
\midrule
{\scriptsize{}iALM (2019) \cite{inexactAugLag19}} & {\scriptsize{}$\tilde{{\cal O}}(\varepsilon^{-3})$} & {\scriptsize{}Nonlinear} & {\scriptsize{}-} & {\scriptsize{}${\cal O}(\beta_k^{-1})$} & {\scriptsize{}${\cal B},{\cal F}$} & {\scriptsize{}n-PAL}  \tabularnewline
{\scriptsize{}iALM (2020)\tablefootnote{An $\tilde{{\cal O}}(\varepsilon^{-2.5})$ iteration complexity bound is established for the case where the constraints are linear. Moreover, the method considered in this table is Algorithm 3 of \cite{ImprovedShrinkingALM20} where it is shown
that the associated sequence of multipliers is bounded under assumption ${\cal F}$.} \cite{ImprovedShrinkingALM20}} & {\scriptsize{}$\tilde{{\cal O}}(\varepsilon^{-3})$} & {\scriptsize{}Nonlinear} & {\scriptsize{}-} & {\scriptsize{}${\cal O}(\beta_k^{-1})$} & {\scriptsize{}${\cal B},{\cal F}$} & {\scriptsize{}n-PAL} \tabularnewline
\textbf{\scriptsize{}NL-IAPIAL} & \textbf{\scriptsize{}$\tilde{{\cal O}}(\varepsilon^{-3})$} & \textbf{\scriptsize{}${\cal K}$-Convex} & {\scriptsize{}$0$} & {\scriptsize{}$1$} & \textbf{\scriptsize{}${\cal B},{\cal N},{\cal SP}$} & {\scriptsize{}PAL} \tabularnewline
\bottomrule
\end{tabular}
\par\end{centering}
\caption{Comparison of relevant PAL and n-PAL methods with NL-IAPIAL where the first three methods assume that $g$ is an affine function of the form $g(x)=Ax-b$. For simplicity, we let $\varepsilon = \min\{\hat\rho,  \hat\eta\}$,  and let ${\tilde{\cal O}}(\cdot)$ be the same as ${{\cal O}}(\cdot)$ with all logarithmic dependencies on $\varepsilon$ removed.}
\label{tab:summary_tbl}
\end{table}

\textit{Contributions}.
We start by highlighting the differences and novelties of
the NL-IAPIAL
compared to the ones in
\cite{HongPertAugLag,RJWIPAAL2020,RenWilmelo2020iteration}.
In contrast to
the PAL methods of \cite{HongPertAugLag,RJWIPAAL2020},
whose iteration-complexities in terms of $\theta$ only (see the second column in Table~\ref{tab:summary_tbl})
are ${\cal O}(\theta^{-2})$ and ${\cal O}(\theta^{-15/4})$, respectively,
this work presents a PAL method and its corresponding
iteration-complexity, both of which
do not depend on $\theta$.
Moreover,
its analysis only assumes the existence of a Slater point and
its multiplier update uses $\theta=0$ and $\chi_k=1$ for every $k$, as prescribed in the classical versions
of both PAL and n-PAL methods.
In contrast to \cite{RenWilmelo2020iteration}
(see the end of the second paragraph of \textit{Related Works}),
our proposed PAL method has the following extra features:
1) it always updates $p_k$ as in \eqref{eq:dual_update}, regardless of whether $\beta_k$ increases or not; and
2) it solves
the more general nonlinear CCC-NCO problem.

Even though NL-IAPIAL is not an n-PAL method, it is still worth discussing 
some of its features relative to
the n-PAL methods of \cite{ImprovedShrinkingALM20,inexactAugLag19}.
First, in contrast to \cite{ImprovedShrinkingALM20,inexactAugLag19}, this work  does
not assume
the strong condition ${\cal F}$
of Table~\ref{tab:conditions} on the iterates generated by their methods (see the fifth column of Table~\ref{tab:summary_tbl}).
Second, in contrast to the methods in \cite{ImprovedShrinkingALM20,inexactAugLag19} whose choices of $\chi_k$ in \eqref{eq:gen_dual_update2} converge to zero as $\beta_k$ tends to infinity\footnote{Methods with this feature tend to become more like penalty-type methods as more iterations are performed.}, NL-IAPIAL chooses $\chi_k=1$ for every $k$ (see the sixth columns of Table~\ref{tab:summary_tbl}). 

Additional discussion of how NL-IAPIAL compares with other related first-order methods that are neither PAL nor n-PAL methods (i.e., \cite{HybridPenaltyAugLag19,ADMMJzhang-ZQLuo2020, ErrorBoundJzhang-ZQLuo2020}) is given in Section~\ref{sec:conclusion}. \\

\emph{Organization of the Paper.}
Subsection~\ref{subsec:notation}  provides some basic definitions and notation. Section~\ref{sec:aug_lagr} contains three subsections. The first one describes the main problem of interest and the assumptions made on it. The second one motivates and  states the NL-IAPIAL method, whereas the third one presents  its main  complexity results. Section~\ref{sec:proofs of propositions} is divided into two subsections. The first one proves Proposition~\ref{lem:StaticIPAAL}(b)--(c) which presents  iteration-complexity bounds for  NL-IAPIAL. The second one proves Proposition~\ref{th:pkbounded} which gives a bound on the  multipliers sequence generated by NL-IAPIAL.  Section~\ref{sec:numerical} is devoted to numerical experiments that illustrate the numerical efficiency of NL-IAPIAL. Section~\ref{sec:conclusion} gives several concluding remarks. The Appendix section contains three subsections.  Appendix~\ref{sec:acg}  reviews an ACG variant,
Appendix~\ref{app:cvx} describes some basic convex analysis
results, and Appendix~\ref{app:refine} is devoted to  the proof
of a basic result considered in the main part of the paper.

\subsection{Basic Definitions and Notations} \label{subsec:notation}

This subsection presents   notation and basic definitions used in this paper.
	
Let $\Re_{+}$  and $\Re_{++}$ denote the
set of nonnegative and positive
real numbers, respectively, and
let $\Re^2_{++}:=\Re_{++}\times\Re_{++}$. We denote by  $R^n$  an $n$-dimensional inner product space with inner product and associated norm denoted by $\inner{\cdot}{\cdot}$ and $\|\cdot\|$, respectively. 
For a given closed convex set $Z \subset \Re^n$, its boundary is denoted by $\partial Z$ and the distance of a point $z \in \Re^n$  to $Z$ is denoted by ${\rm dist}(z,Z)$. The indicator function of $Z$, denoted by $\delta_Z$, is defined by $\delta_Z(z)=0$ if $z\in Z$, and $\delta_Z(z)=+\infty$ otherwise.
 For any $t>0$, we let $\log_1^+(t):=\max\{\log t, 1\}$,
and we  define ${\cal O}_1(\cdot) = {\cal O}(1 + \cdot)$.

The domain of a function $h :\Re^n\to (-\infty,\infty]$ is the set   $\dom h := \{x\in \Re^n : h(x) < +\infty\}$.
Moreover, $h$ is said to be proper if
$\dom h \ne \emptyset$. The set of all lower semi-continuous proper convex functions defined in $\Re^n$ is denoted by $\cConv \rn$. The $\varepsilon$-subdifferential of a proper function $h :\Re^n\to (-\infty,\infty]$ is defined by 
\begin{equation}\label{def:epsSubdiff}
\partial_\varepsilon h(z):=\{u\in \Re^n: h(z')\geq h(z)+\inner{u}{z'-z}-\varepsilon, \quad \forall z' \in \Re^n\}
\end{equation}
for every $z\in \Re^n$.	The classical subdifferential, denoted by $\partial h(\cdot)$,  corresponds to $\partial_0 h(\cdot)$.  
Recall that, for a given $\varepsilon\geq 0$, the $\varepsilon$-normal cone of a closed convex set $C$ at $z\in C$, denoted by  $N^{\varepsilon}_C(z)$, is 
$$N^{\varepsilon}_C(z):=\{\xi \in \Re^n: \inner{\xi}{u-z}\leq \varepsilon,\quad \forall u\in C\}.$$
The normal cone of a closed convex set $C$ at $z\in C$ is denoted by  $N_C(z)=N^0_C(z)$.
If $\psi$ is a real-valued function which
is differentiable at $\bar z \in \Re^n$, then its affine   approximation $\ell_\psi(\cdot,\bar z)$ at $\bar z$ is given by
\begin{equation}\label{eq:defell}
\ell_\psi(z;\bar z) :=  \psi(\bar z) + \inner{\nabla \psi(\bar z)}{z-\bar z} \quad \forall  z \in \Re^n.
\end{equation}
For a closed convex cone $\cK\subset \Re^l$, the dual cone $\cK^{*}$ is defined as
\begin{align}
\cK^{*}:=\left\{ y\in\Re^l:\left\langle y,x\right\rangle \geq0,\, x\in\cK\right\}.
\end{align}
For given $u,v\in \Re^l$,
the notation  $u\preceq_{\cK} v$ (or $v\succeq_{\cK} u$)
means that $v-u\in \cK$. 
Moreover,  the notation  $u\prec_{\cK} v$ means that  $v-u\in \inte \cK$. 
  A function $g:\Re^n\to\Re^\ell$ is  said to be $\cK$-convex if
\begin{equation}\label{def:Kconvexity}
g(t z' + [1-t] z) - t g(z') - [1-t] g(z) \preceq_{\cK} 0 \quad \forall z, z'\in\rn, \, \forall t \in [0,1].
\end{equation}
Under the assumption that $g$ is differentiable, it is well-known that $g$ is $\cK$-convex if and only if
\begin{equation}\label{ineq:gradineq-Kconvexity}
\inner {p}{g'(z)(z'-z)} \leq \inner{p}{g(z')-g(z)}
\quad \forall z, z'\in\rn, \, \forall p \in {\cK}^*.
\end{equation}

\section{The NL-IAPIAL Method} \label{sec:aug_lagr}

This section consists of three subsections. The first one precisely describes the problem of interest and its assumptions. The second one motivates and states the NL-IAPIAL method. The third one presents the main complexity results for NL-IAPIAL.

\subsection{Problem of Interest} \label{subsec:prb_of_interest}

This subsection presents the main problem of interest and discusses the assumptions underlying it.

Consider problem \eqref{eq:main_prb}
where ${\cal K}$ is a closed convex cone such that
$\emptyset \ne {\cal K} \ne \R^l$, and
functions $f,g$ and $h$ satisfy the following
assumptions:
\begin{itemize}
\item[(A1)] $h\in\cConv\rn$ and its domain
${\cal H}:=\dom h$ is a compact set; moreover,
for some scalar $K_{h} \ge  0$, function $h$ is
$K_{h}$-Lipschitz continuous on ${\cal H}$, i.e.,
it satisfies
\[
|h(z')- h(z)| \le K_h \|z'-z\| \qquad \forall z, z' \in {\cal H};
\]
\item[(A2)] $f$ is a nonconvex function which is differentiable on ${\cal H}$, and there exist $0<m_{f} \le L_{f}$
such that $f$ is $m_f$-weakly convex on $\mathcal{H}$
(i.e., $f+m_f\|\cdot\|^2/2$ is convex on $\mathcal{H}$)  and
\begin{gather}
\|\nabla f(z') - \nabla f(z) \| \leq L_f \|z'-z\|
\qquad \forall z', z \in \mathcal{H}; \label{gradLips}
\end{gather}
\item[(A3)] $g:\r^n \mapsto \r^{\ell}$ is $\cK$-convex and differentiable, and there
exists $L_{g} > 0$ such that 
\[
\|\nabla g(z') - \nabla g(z)\|\leq L_{g}\|z'-z\|\qquad\forall z',z\in\rn;
\]
\item[(A4)] there exist $\bar{z}\in \inte{\cal H}$ and
$\tau_g >0$ such that $g(\bar z)\preceq_\cK 0$ and 
\begin{equation}
\max\left\{ \|\nabla g(z)p\|,\left|\left\langle p,g(\bar{z})\right\rangle \right|\right\} \geq \tau_g\|p\|\qquad\forall z\in{\cal H},\enskip \forall p \succeq_{{\cal K}^{*}} 0.\label{eq:gen_slater}
\end{equation}
\end{itemize}

We now make some comments about the above assumptions.
First,
any function $h$ of the form $h=\tilde h + \delta_Z$
where $\tilde h$ is a finite everywhere Lipschitz continuous
convex function
and $Z$ is a compact convex set clearly satisfies condition (A1).
Second, it is easy to see that (A2) implies that
\begin{equation}
-\frac{m_{f}}{2}\|z'-z\|^{2}\leq f(z')-
\ell_f(z'; z) \quad \forall z',z \in \mathcal{H},\label{lowerCurvature-m}
\end{equation}
where $\ell_f(\cdot; \cdot)$ is as in \eqref{eq:defell}.
Moreover, 
it is well-known that \eqref{gradLips} implies that 
$ |f(z') -\ell_f(z';z) | \le  L_f\|z'-z \|^2/2$ for every $z,z' \in {\mathcal H}$, and hence that \eqref{lowerCurvature-m} holds with $m_f=L_f$. However, we will show that 
better
iteration-complexity bounds for our method
can be derived when a scalar
$m_f< L_f$ 
satisfying \eqref{lowerCurvature-m}
is available. Third,  since $f$ is nonconvex on ${\mathcal H}$, (A2) implies the smallest $m_f$ satisfying  \eqref{lowerCurvature-m} is positive. Fourth, the assumption that ${\cal K} \ne \r^l$
 implies that ${\cal K}^* \ne \{0\}$. Finally, the cone $\cal K$ is not assumed to have a nonempty interior.

The result below, whose proof is given in Appendix~\ref{app:slater}, shows that if ${\cal K} = {\cal J}\times\{0\}$ where
$\cal J$ is a closed convex cone such that $\intr {\cal J}\neq \emptyset$, then (A4) is equivalent to a Slater-like assumption with respect to $g$. Hence, (A4) is a mild assumption on \eqref{eq:main_prb}.

\begin{prop} 
\label{prop:weak_slater}
Suppose ${\cal J}\subseteq\r^{s}$ is a closed
convex cone with nonempty interior, $g_{\iota}:\rn\mapsto\r^{s}$
is a (possibly nonconvex) continuously differentiable function, and
$g_{e}:\rn\mapsto\r^{t}$ is an onto affine map. Moreover, suppose
$\nabla g_{\iota}(\cdot)$ is $L_{g_{\iota}}$-Lipschitz continuous on
the set ${\cal H}$ defined in (A1), and let $g:=(g_{\iota},g_{e})$ and ${\cal K}:= {\cal J}\times\{0\}$.
Then, the following statements are equivalent:
\begin{itemize}
\item[(a)] there exists $\tau_{g}>0$ and $\bar{z}\in\intr{\cal H}$ such that
$g(\bar{z})\preceq_{{\cal K}}0$ and \eqref{eq:gen_slater} holds;
\item[(b)] there exists $\tilde{\tau}_{g}>0$ and $\bar{z}\in\intr{\cal H}$
such that $g(\bar{z})\preceq_{{\cal K}}0$ and
\begin{equation}
\max\left\{ \|\nabla g(\bar{z})p\|,|\left\langle p,g(\bar{z})\right\rangle |\right\} \geq\tilde{\tau}_{g}\|p\|\quad\forall p\succeq_{{\cal K}^{*}}0;\label{eq:sp0}
\end{equation}
\item[(c)] there exists $\bar{z}\in \intr{\cal H}$ such that $g_{\iota}(\bar{z})\prec_{{\cal J}}0$
and $g_{e}(\bar{z})=0$;
\end{itemize}
\end{prop}

Some comments about Proposition~\ref{prop:weak_slater} are in order.
First, if $g_\iota$ is ${\cal J}$-convex and $g_e$ is affine, then
 $g$ is $\cK$-convex. 
Second, the Slater condition is in regard to a single point $\bar z\in\cH$, as opposed to condition \eqref{eq:gen_slater} which involves inequality \eqref{eq:gen_slater} at all pairs $(z,p)\in\cH\times\cK^*$. Third, (A4) can be replaced by
the Slater-like assumption
of Proposition~\ref{prop:weak_slater} 
when ${\cal K}={\cal J}\times\{0\}$ since
the former is equivalent to the latter in this case.
Actually, a slightly more involved analysis
can be done to show that the assumption that $g_e$ is onto
(which is part of the assumption of Proposition~\ref{prop:weak_slater})
can be removed 
at the expense of obtaining a weaker version of (A4), namely:
inequality \eqref{eq:gen_slater}  holds for every pair $(z,p) \in \cH \times ({\cal J}^* \times {\rm Im}\,\nabla g_e)$, instead of 
$(z,p) \in \cH \times ({\cal J}^* \times \r^t) = \cH \times \cK^*$.
Finally, since the analysis of this paper can be easily adapted
to this slightly weaker version of (A4), 
the Slater-like condition of Proposition~\ref{prop:weak_slater}
without $g_e$ assumed to be onto (or equivalently, $\nabla g_e$
to have full column rank)
can be used in place of (A4) in order to guarantee that all of
the results derived in this paper for NL-IAPIAL hold.

Under assumptions (A1)--(A4), it can be shown that: (i) a necessary condition  for a point $z^{*}$ to be a local minimum of \eqref{eq:main_prb} is that there exists a  multiplier $p^*\in\cK^*$ satisfying
\begin{equation}
\begin{gathered}0 \in \nabla f(z^{*})+\pt h(z^{*})+\nabla g(z^{*})p^{*},
\quad \inner{g(z^*)}{p^*}=0, \quad g(z^*)\preceq_{\cK} 0, \quad p^{*} \succeq_{{\cal K}^{*}} 0;
\end{gathered}
\label{eq:stationary_soln}
\end{equation}
and (ii) the last three conditions in \eqref{eq:stationary_soln}  are equivalent\footnote{See Lemma~\ref{lem:dist_props}(c).}
to the inclusion $g(z^*) \in N_{\cK^*}(p^*)$.
The following definition describes the type of  approximate solution of \eqref{eq:main_prb} that is sought after by the NL-IAPIAL method.

\begin{definition}\label{def:stationarypoint}
	Given a tolerance pair  $(\hat\rho,\hat\eta)\in \Re_{++}\times\Re_{++}  $, a quadruple $(\hat z, \hat p,\hat w,\hat q)\in {\mathcal H}\times\Re^l\times\Re^n\times \Re^l$ is said to be a $(\hat\rho,\hat\eta)$-approximate stationary  quadruple of \eqref{eq:main_prb} if it satisfies \eqref{eq:rho_eta_approx_soln} and \eqref{ineq:rho_eta_approx_soln}.
\end{definition}
	
We now make some observations about Definition~\ref{def:stationarypoint}.
Another notion of approximate stationarity for \eqref{eq:main_prb} is as
follows:
a pair $(\hat z,\hat p)\in {\mathcal H}\times \Re^l$ is a $(\hat\rho,\hat\eta)$-approximate
stationary solution of \eqref{eq:main_prb} if it
satisfies
the inequalities
\begin{equation}\label{ineq:alternative-stationary-solution}
{\rm dist}\,(0,\nabla f(\hat{z})+\pt h(\hat{z})+\nabla g(\hat{z})\hat{p})\le \hat \rho, \qquad {\rm{dist}}(g(\hat z), N_{\cK^*}(\hat{p})) \leq \hat \eta.
\end{equation}
It turns out that
$(\hat z,\hat p)$
is a \textit{$(\hat\rho,\hat\eta)$-approximate
stationary solution} in the above sense if and only if
there exists a residual pair $(\hat w, \hat q)\in \Re^n\times \Re^l$
such that $(\hat z, \hat p,\hat w,\hat q)$ is a 
$(\hat\rho,\hat\eta)$-approximate stationary quadruple
of \eqref{eq:main_prb}.
In this regard, the
residual pair $(\hat w, \hat q)$ in Definition~\ref{def:stationarypoint} can be
viewed as 
a certificate that the pair $(\hat z,\hat p)$
in the same
definition is
a $(\hat\rho,\hat\eta)$-approximate stationary
solution of \eqref{eq:main_prb}.
Finally, our analysis is entirely based on
the notion of Definition~\ref{def:stationarypoint}
even though it could also have been  carried out
using the notion of
a $(\hat\rho,\hat\eta)$-approximate stationary
solution instead. The main reason for this choice
is that
the NL-IAPIAL method presented in
Subsection~\ref{sec:NLIAPIAL} naturally generates
residual pairs which always satisfy \eqref{eq:rho_eta_approx_soln}, and eventually
\eqref{ineq:rho_eta_approx_soln} after a
sufficient number of iterations.
Moreover, as opposed to
the residual pairs which ``realize'' the two distances in \eqref{ineq:alternative-stationary-solution},
the computation of these residual pairs do
not require projections onto
$\partial h(\hat z)$ or $N_{\cK^*}(\hat{p})$.

We end this subsection by stating a technical result
which describes some properties about
the smooth part of the Lagrangian in \eqref{eq:aug_lagr_def}.

\begin{lem}
\label{lem:dist_smoothness}
Assume that conditions (A2) and (A3) hold, and define the function
\begin{gather}
\widetilde{{\cal L}}_{\beta}(z,p) := f(z) + \frac{1}{2\beta}\left[{\rm dist}^{2}(p+\beta g(z),-{\cal K}) - \|p\|^2 \right] \quad \forall (z,p,\beta)\in \rn \times \r^{\ell}\times\r_{++} \label{eq:L_tilde_def}
\end{gather}
and the quantities
\begin{equation}
B_g^{(0)}:= \sup_{z\in{\cal H}} \|g (z)\|, \quad 
B_g^{(1)}:= \sup_{z\in{\cal H}} \|\nabla g (z)\|. \label{eq:bd_Psi_val}
\end{equation}
Then, for every $\beta >0$ and $p\in\r^{\ell}$, the following properties hold:
\begin{itemize}
\item[(a)] $\widetilde{{\cal L}}_{\beta}(\cdot, p)$ is $m_f$-weakly convex on $\mathcal{H}$, where $m_f$ is as in (A2);

\item[(b)] $\widetilde{{\cal L}}_{\beta}(\cdot, p)$ is  a  differentiable function whose gradient is given by
\[
\nabla_z\widetilde{{\cal L}}_{\beta}(z,p)=\nabla f(z) + \nabla g(z)\Pi_{{\cal K}^{*}}(p+\beta g(z)) \quad \forall z\in\rn;
\]

\item[(c)] $\nabla_z \widetilde{{\cal L}}_{\beta}(\cdot, p)$ is $\widetilde{\cal M}$-Lipschitz continuous where
\begin{equation}
    \widetilde{\cal M}  =  \widetilde{\cal M}(\beta,p) := L_f + L_{g}\|p\|+\beta M_g,
    \quad {M_g} := B_g^{(0)} L_{g} + [B_g^{(1)}]^2,
    \label{eq:Lipsc_tilde_def}
\end{equation}
and the quantities $L_f$ and  $L_g$ are as in (A2) and (A3), respectively.
\end{itemize}
\end{lem}

\begin{proof}
The statements of the lemma with $f\equiv 0$ (and hence $m_f=L_f=0$)
immediately follow from
\cite[Proposition 5]{zhaosongAugLag18}.
Hence,  the general case of the lemma easily follows  from  assumption (A2) and the definition of $\tilde {\cal L}_\beta$ in \eqref{eq:L_tilde_def}.
\end{proof}

\subsection{The NL-IAPIAL Method}\label{sec:NLIAPIAL}

This subsection  motivates and states the NL-IAPIAL method.  

Before presenting the method, we give a short but precise outline of its key steps, as well as a description of how its iterates are generated.
Recall from the introduction that the NL-IAPIAL method,
whose goal is to find a $(\hat \rho, \hat \eta)$-approximate stationary quadruple as in \eqref{eq:rho_eta_approx_soln} and \eqref{ineq:rho_eta_approx_soln}, is an iterative method which, at its $k$-th step, computes its next iterate $(z_k, p_k)$ according to \eqref{eq:approx_primal_update} and \eqref{eq:dual_update}.

We now describe the conditions which are required on the approximate solution $z_k$ of \eqref{eq:approx_primal_update}.
For a given scalar  $\sigma \in (0,1/\sqrt{2}]$, NL-IAPIAL requires that 
$z_k$,
together with a  residual
pair  $(v_k,\varepsilon_k)\in \R^n\times \R_{++}$,
satisfy
\begin{equation}
v_{k}\in\partial_{\varepsilon_{k}}\left(\lambda{\cal L}_{\beta_{k}}(\cdot,p_{k-1})+\frac{1}{2}\|\cdot-z_{k-1}\|^{2}\right)(z_{k}),\quad\|v_{k}\|^{2}+2\varepsilon_{k}\leq\sigma_k^{2}\|v_k + z_{k-1} - z_{k}\|^{2}. \label{eq:prox_incl}
\end{equation}
where
\begin{equation}\label{eq:acg_input_aux_defs}
\begin{gathered}
\sigma_k :=\frac{\sigma}{\sqrt{\widetilde{\cal M}_k}}, \qquad \widetilde{\cal M}_k :=\lam \widetilde{\cal M}(\beta_k, p_{k-1}) + 1,
\end{gathered}
\end{equation}
and $\widetilde{\cal M}(\cdot,\cdot)$ is as in \eqref{eq:Lipsc_tilde_def}.
Note that if $\sigma=0$ then the inequality in \eqref{eq:prox_incl} implies that
$(v_k,\varepsilon_k)=(0,0)$, and hence that
$z_k$ is a global solution of \eqref{eq:approx_primal_update} in view of the inclusion in \eqref{eq:prox_incl} and the definition of $\varepsilon$-subdifferential given in \eqref{def:epsSubdiff}. By relaxing
$\sigma$ to be positive, we are then allowing $z_k$ to be an
inexact (global) solution of \eqref{eq:approx_primal_update}.

The following result now describes a way of computing the
approximate triple $(z_k,v_k,\varepsilon_k)$ as in the above paragraph. Its proof strongly relies on the fact that
$z_{k-1}$ is chosen to be the initial point for the ACG variant
(see the fifth identity in \eqref{eq:acg_inputs}) and 
Proposition~\ref{prop:nest_complex} of Appendix~\ref{sec:acg}.

\begin{lemma}\label{lem:kthACGbound}
Let  $\lambda=1/(2m_f)$ where $m_f$ is  as in (A2), and  define 
\begin{align}
\begin{gathered}
\psi_{s}=\protect\lam \widetilde{{\cal L}}_{\beta_k}(\cdot, p_{k-1}) +\frac{1}{2}\|\cdot-z_{k-1}\|^{2},\quad\psi_{n}=\protect\lam h,\\[1mm] 
\widetilde M = \widetilde{\cal M}_k, \quad \widetilde\mu = \frac{1}{2}, \quad x_{0}=z_{k-1},\quad\tilde{\sigma}=\sigma_k,
\end{gathered} \label{eq:acg_inputs}
\end{align}
where $\widetilde{\cal M}_k$ is as in \eqref{eq:acg_input_aux_defs}.  Then, 
the ACG algorithm of Appendix~\ref{sec:acg}, with inputs given by \eqref{eq:acg_inputs}, computes a triple $(z_k,v_k,\varepsilon_k):=(y,u,\eta)$ satisfying    \eqref{eq:prox_incl} in a number of ACG iterations bounded by
\begin{equation}\label{eq:lemACGbound}
\left\lceil5\sqrt{\widetilde{\cal M}_k}\log_{1}^{+}\left(\frac{4\widetilde{\cal M}_k}{\sigma}\right) \right\rceil.
\end{equation}

\end{lemma}
\begin{proof}
We first show that the inputs in \eqref{eq:acg_inputs} satisfy conditions (B1)--(B2) in Appendix~\ref{sec:acg}. Indeed, using assumption (A1) and Lemma~\ref{lem:dist_smoothness}(a),
it is easy to see that both 
$\psi_s + (\lam m_f-1)\|\cdot\|^2/2$ and $\psi_n$ are convex.
Since $\lambda=1/(2m_f)$, it then follows that
$\psi_s$ is a $1/2$-strongly convex and hence that $\tilde \mu$ satisfies the first inequality in \eqref{ineqs in Assump B2}.
Now, in view of Lemma~\ref{lem:dist_smoothness}(c) and the definition of $\psi_s$ in \eqref{eq:acg_inputs}, it follows that $\widetilde M$ satisfies the second inequality in \eqref{ineqs in Assump B2}. Hence, we conclude that the inputs in \eqref{eq:acg_inputs} satisfy  the conditions (B1)--(B2) in Appendix~\ref{sec:acg}. 

We now derive the desired complexity bound. It follows from Proposition~\ref{prop:nest_complex} and the above result that the ACG algorithm of Appendix~\ref{sec:acg}  with inputs  given by \eqref{eq:acg_inputs} generates a triple $(z_k,v_k,\varepsilon_k):=(y,u,\eta)$ satisfying \eqref{eq:prox_incl} in at most 
\begin{equation}\label{eq:auxACGcomplexity00}
\left\lceil 1 + \left(\frac{1}{2}+\sqrt{2\widetilde{\cal M}_k-1}\right)\log_{1}^{+}\widetilde{\cal A}\right\rceil
\end{equation}
iterations, where 
$\widetilde{\cal A}=4(1+\widetilde\sigma)^{2}(\widetilde{\cal M}_k-1/2)\widetilde\sigma^{-2}.
$
Now, note that  the definitions of $\sigma_k$ and $\widetilde\sigma$ in \eqref{eq:acg_input_aux_defs} and  \eqref{eq:acg_inputs}, respectively,  yield 
$\widetilde{\cal A} \leq 16(\widetilde{\cal M}_k)^2\sigma^{-2}$. Hence,   \eqref{eq:lemACGbound} follows from \eqref{eq:auxACGcomplexity00},  the latter inequality, and the fact that   $\log^+_1(\cdot)\geq 1$ and $\widetilde{\cal M}_k\geq 1$.
\end{proof}

It is worth mentioning that
the main effort of an ACG iteration
consists of: (i) the computation of
$\nabla \psi_s(\tilde x_j) $ where
$\tilde x_j$ is one of the iterates
obtained in the $j$-th iteration
of ACG
(see \eqref{eq:prox-acg}); and,
(ii)
the solution of
the prox subproblem in \eqref{eq:prox-acg}.
Its description given in Appendix~\ref{sec:acg} assumes
that both (i) and (ii) can be carried out exactly with the aid
of given oracles.
Moreover, for the case where the functions
$\psi_s$ and $\psi_n$ are chosen as in \eqref{eq:acg_inputs}, it follows
from Lemma~\ref{lem:dist_smoothness}(b) that 
$$ \nabla \psi_s(z)=\lam \left[ \nabla f(z) + \nabla g(z)\Pi_{{\cal K}^{*}}(p_{k-1}  +\beta_k g(z)) \right] +z-z_{k-1}.$$
Finally, since we make the blanket assumption that an oracle for exactly evaluating $\Pi_{{\cal K}^{*}}(\cdot)$ at any given point is available, it follows that $\nabla \psi_s(x)$
can be obtained exactly by means of
the above formula. 

We are now ready to provide a complete description
of the NL-IAPIAL method.

\noindent\begin{minipage}[t]{1\columnwidth}%
\rule[0.5ex]{1\columnwidth}{1pt}

\noindent \textbf{NL-IAPIAL Method}

\noindent \rule[0.5ex]{1\columnwidth}{1pt}%
\end{minipage}

\noindent \textbf{Input}: a function triple $(f,g,h)$ and
a quadruple of parameters $(K_h,m_f,L_f,L_g)$  satisfying assumptions (A1)--(A4), a scalar
$\sigma\in(0,1/\sqrt{2}]$, a penalty parameter $\beta_1>0$, an initial
pair  $(z_{0},p_0)\in {\cal H}\times \Re^l$, and a tolerance pair $(\hat{\rho},\hat{\eta})\in\r_{++}^{2}$;

\noindent \textbf{Output}: a triple $(\hat{z},\hat{p},\hat{w},\hat{q})$ satisfying \eqref{eq:rho_eta_approx_soln}-\eqref{ineq:rho_eta_approx_soln};

\begin{itemize}
\item[{\bf 0.}] set  $\hat k=0$,  $k=1$ and
\begin{equation}\label{def:lamb-C1}
\lam=\frac{1}{2m_{f}}, \quad \beta =\beta_1,\quad C_\sigma= \frac{2(1+2\sigma)^2}{1-\sigma^2};
\end{equation}

\item[{\bf 1.}]  use the ACG described in 
Appendix~\ref{sec:acg}
with inputs $( \widetilde M, \widetilde \mu, \psi_s, \psi_n)$,
$x_0$ and $\tilde{\sigma}$ given by \eqref{eq:acg_inputs} to obtain a triple $(z_{k},v_{k},\varepsilon_{k}):=(y,u,\eta)$ satisfying \eqref{eq:prox_incl}, and compute
\begin{equation}\label{eq:dual_update2}  
p_{k}  :=\Pi_{\cK^{*}}(p_{k-1}+{\beta_k}g(z_{k})), \qquad r_{k}:=v_{k}+z_{k-1}-z_{k};
\end{equation}
\item[{\bf 2.}] compute the point $\hat z_k$ as
\begin{equation}
\hat{z}_{k}:=\argmin_{u}\left\{ \lam\left[\left\langle \nabla_z \widetilde{{\cal L}}_{\beta_k} (z_{k}, p_{k-1}),u-z_{k}\right\rangle +h(u)\right]-\left\langle r_{k},u-z_{k}\right\rangle +\frac{\widetilde{\cal M}_k}{2}\|u-z_{k}\|^{2}\right\}, \label{eq:z_k_def}
\end{equation}
and the triple $(\hat{p}_{k},\hat{w}_{k},\hat{q}_k)$ as
\begin{gather}
\begin{aligned}
\hat{p}_{k} & :=\Pi_{{\cal K}^{*}}\left(p_{k-1}+\beta_{k}g(\hat{z}_{k})\right),  \\
\hat{w}_{k} & :=w_{k}+\nabla_z\widetilde{{\cal L}}_{\beta_k} (\hat{z}_{k}, p_{k-1})-\nabla_z\widetilde{{\cal L}}_{\beta_k} (z_{k}, p_{k-1}),\\
\hat{q}_k &:= \frac{1}{\beta_{k}}(p_{k-1} - \hat{p}_k),
\end{aligned}
\label{eq:refined_points}
\end{gather}
where  $\widetilde{\cal M}_k$ and $\widetilde{{\cal L}}_{\beta_k}$  are as in  \eqref{eq:L_tilde_def} and \eqref{eq:acg_input_aux_defs}, respectively, and
\begin{equation}
w_{k}:=\frac{1}{\lam}\left[r_{k}+\widetilde{\cal M}_k(z_{k}-\hat{z}_{k})\right];  
\label{eq:refine_aux_defs}
\end{equation}
if $(\hat{w},\hat{q}):=(\hat{w}_{k},\hat{q}_k)$ satisfies \eqref{ineq:rho_eta_approx_soln} 
then stop and output $(\hat{z},\hat{p},\hat{w},\hat{q})=(\hat{z}_{k},\hat{p}_{k},\hat{w}_{k},\hat{q}_k)$;
\item[{\bf 3.}]
if $k > \hat k+1 $ and 
\begin{equation}\label{eq:stopCOndition-deltak}
\Delta_k:=\frac{1}{k-\hat k-1}\left[\mathcal{L}_{\beta_k}(z_{\hat k+1},p_{\hat k})-\mathcal{L}_{\beta_k}(z_k,p_k) - \frac{\|p_k\|^2}{2\beta_k} \right] \le
\frac{\lambda\hat \rho^2}{2C_\sigma},
    \end{equation}
then set $\beta_{k+1}=2 \beta_k$ and $\hat k=k$;
otherwise, set $\beta_{k+1}=\beta_k$;
\item[{\bf 4.}] 
update $k\gets k+1$, and go to step~1. 
\end{itemize}
\noindent \rule[0.5ex]{1\columnwidth}{1pt}

Some remarks about NL-IAPIAL  are in order.
 First, it performs two kinds of iterations, namely, the ones indexed by $k$ and the ones performed by the ACG algorithm every time it is called in step~1. We refer to the former as ``outer'' iterations and the latter as ``inner'' (or ACG) iterations.  Second, its input  $z_0$ can be any element in the domain of $h$ and does not necessarily need to be a point satisfying the constraint $g(z_0)\preceq_{\cK} 0$. Third, the ACG described in Appendix~\ref{sec:acg} is invoked in step~1 to compute  a triple $(z_k,v_k,\varepsilon_k)$ satisfying \eqref{eq:prox_incl}, which can be seen as an approximate stationary solution for the prox-subproblem \eqref{eq:approx_primal_update}. Fourth,   it will be shown in Lemma~\ref{prop:refinement} that the refined quadruple $(\hat{z},\hat{p},\hat{w},\hat{q}):=(\hat z_k,\hat p_k, \hat w_k,\hat q_k)$ computed in step~2 satisfies all the relations in \eqref{eq:rho_eta_approx_soln} at any outer iteration. As a consequence,  the  NL-IAPIAL output  $(\hat{z},\hat{p},\hat{w},\hat{q})$
is a $(\hat \rho, \hat \eta)$-approximate stationary quadruple of \eqref{eq:main_prb} in the sense of Definition~\ref{def:stationarypoint}.
Finally, it follows from Lemma~\ref{lem:dist_smoothness}(b),
 and the first identities in
\eqref{eq:dual_update2} and \eqref{eq:refined_points},
that the gradients of the function
$\widetilde{{\cal L}}_{\beta_k} (\cdot, p_{k-1})$ which
appear
in \eqref{eq:refined_points} can be computed as $\nabla_z\widetilde{{\cal L}}_{\beta_k} (z_k, p_{k-1})=\nabla f(z_k)+\nabla g(z_k) p_k$ and $\nabla_z\widetilde{{\cal L}}_{\beta_k} (\hat{z}_{k}, p_{k-1})=\nabla f(\hat z_k)+\nabla g(\hat z_k) \hat p_k$.

In the remaining part of
this subsection, we give
some intuition about
step 3 of NL-IAPIAL.
Define
the $l$-th cycle ${\cal C}_l$
as the $l$-th set of consecutive indices
$k$ for which $\beta_k$ remains constant, i.e.,
\begin{equation}\label{def:ctilde-l}
{\cal C}_l := \{ k : \beta_k=\tilde \beta_l:= 2^{l-1} \beta_1 \}.
\end{equation}
For every $l \ge 1$, we let
$k_l$ denote the largest index in ${\cal C}_l$.
Hence,
\[
{\cal C}_l = \{ k_{l-1} + 1, \ldots, k_l\} \quad \forall l \ge 1\]
where $k_0:=0$. Clearly, the different values of $\hat k$ that
arise in step 3 are exactly the indices in the index set $\{k_l : l \ge 0\}$.
Moreover, in view of the test performed
in step 3, we have that
$k_l-k_{l-1} \ge 2$ for every $l \ge 1$, or equivalently,
every cycle contains at least two indices. While generating the indices
in the $l$-th cycle,
if an index $k \ge k_{l-1}+2$
satisfying
\eqref{eq:stopCOndition-deltak} is found,
$k$ becomes the last index $k_l$ in the
$l$-th cycle and the $(l+1)$-th cycle
is started at iteration $k_l+1$ with
the penalty parameter set to
$\tilde \beta_{l+1}=2 \tilde \beta_l$, where $\tilde \beta_l$ is as in \eqref{def:ctilde-l}.

Finally,
the role played by criterion
\eqref{eq:stopCOndition-deltak}
is as follows. It is shown in Lemma~\ref{lem:what-Deltak} that for every $k \in {\cal C}_\ell$, there exists $j \in {\cal C}_\ell$, $j \le k$ such that
\begin{equation}\label{eq:aux wj-qj}
\|\hat w_j\|^2= \frac{C_\sigma \Delta_k}{\lam}+ {\cal O} \left(  \frac{1}{\tilde \beta_{l}} \right), \quad \|\hat q_j\|=\mathcal{O}\left( \frac{1}{\tilde \beta_{l}} \right).
\end{equation}
 Hence, if criterion \eqref{eq:stopCOndition-deltak} holds,
 then 
 \eqref{eq:aux wj-qj} implies that   $\|\hat w_j\|^2=\hat \rho^2/2+{\mathcal O}(1/\tilde \beta_l)$ and $\|\hat q_j\|={\mathcal O}(1/\tilde \beta_l)$. 
 On the other hand,
 since $\tilde \beta_l$ is doubled from one cycle to another, 
 these residual estimates imply
 that the stopping criterion in
 step~2  will eventually be
 satisfied.

 \subsection{Complexity results for  NL-IAPIAL }
This subsection contains the main complexity results for NL-IAPIAL.

We start by  considering a proposition, whose proof is presented in Section~\ref{subsec:boundMultplier}, that   shows that the sequence of Lagrange multipliers $\{p_k\}$ is bounded. Before presenting the result, we first introduce  the following quantities:
\begin{equation}\label{def:DhBf1}
\bar{d}:=\mbox{\rm dist}(\bar z, {\partial {\mathcal H}}),\qquad D_{h}:=\sup_{z',z\in{\cal H}}\|z'-z\|, \qquad \theta_h:=\frac{D_h}{\min\{1,\bar{d}\}}\qquad B_f^{(1)}:=\sup_{z\in{\cal H}} \|\nabla f (z)\|,
\end{equation}
\begin{equation}\label{def:kappa00}
\kappa_0:= 2\left[K_h+B_f^{(1)}\right] +\left[\frac{\sigma^2}{(1-\sigma)^2} + 4\left(\frac{1+\sigma}{1-\sigma}\right)\right]m_fD_h,
\end{equation}
where $\sigma \in (0,1/\sqrt{2}]$ is an input of NL-IAPIAL, $K_h$ and $m_f$ are as in (A1) and (A2), respectively,  and $\partial {\cal H}$ denotes the boundary of ${\cal H}$.
Observe that $\bar d >0$ in view of the fact that,
by (A4), $\bar z \in \intr {\cal H}$.
Moreover, using the fact that
${\cal H}$ is compact and $\nabla f$ is continuous
on ${\cal H}$ due to  (A1) and (A2), respectively,
it follows that
$D_h$ and $B_f^{(1)}$ are finite.
These two observations then imply that $\theta_h$ and $\kappa_0$ are also finite.

\begin{proposition}\label{th:pkbounded} The sequence  $\{p_k\}$ generated by  NL-IAPIAL satisfies 
\begin{equation} \label{ineq:pkbounded}
\|p_k\|\leq \kappa_p:=\max\left\{\|p_0\|,\frac{\theta_h\kappa_0}{\tau_g}\right\},\qquad \forall k\geq 0,
\end{equation}
where $\theta_h$, $\kappa_0$, and $\tau_g$, are as in \eqref{def:DhBf1}, \eqref{def:kappa00}, and (A4), respectively.
\end{proposition}

The following quantities will be used in the subsequent results:
\begin{equation}\label{phi*-Deltaphi*}
 \Delta \phi^*:=\phi^*-\phi_*, \quad \phi_*:=\inf_{z\in \Re^n} \phi(z),
\end{equation}
\begin{equation}\label{def:newkappa1}
\kappa_1 := \left(\frac{3L_{f}+L_g\kappa_p}{2m_f}\right)^{1/2},
\quad  \kappa_2:=6\kappa_p\sqrt{M_gC_\sigma}, \quad
\kappa_3:=\left[\left(\tau_g+4\sqrt{M_g} \right)\frac{\kappa_p \sqrt{M_g}}{2m_f}
\right]^{1/2},
\end{equation}
\begin{equation} \label{def:c_bar_phi_star}
\bar{\beta}=\bar{\beta}(\hat \rho, \hat \eta)  := \frac{m_f}{M_g}\left(\frac{\kappa_2^2}{\hat\rho^2} + \frac{\kappa_3^2}{\hat\eta}\right),
\end{equation}
where  the quantities     $(m_f,L_f)$, $L_g$, $\phi^*$, $M_g$,  $C_\sigma$,  $D_h$, and $\kappa_p$  are as in  (A2), (A3), \eqref{eq:main_prb}, \eqref{eq:Lipsc_tilde_def},  \eqref{def:lamb-C1}, \eqref{def:DhBf1}, and \eqref{ineq:pkbounded},  respectively.

The following result, whose proof is given in Subsection~\ref{sec:technical}, establishes
bounds on the number of ACG and outer iterations performed during an NL-IAPIAL cycle, and shows that NL-IAPIAL outputs a $(\hat\rho,\hat\eta)$-approximate stationary quadruple of \eqref{eq:main_prb}
within a logarithmic number of cycles.

\begin{proposition}\label{lem:StaticIPAAL}
The following statements about NL-IAPIAL hold: 
\begin{itemize}
	\item[(a)] every outer iteration within
	the $l$-th cycle performs at most
\begin{equation*}
\left\lceil 5\left(\kappa_1 + \sqrt{\frac{\tilde \beta_l M_g}{2m_f}}\right)\log_{1}^{+}\left(\frac{4\kappa_1^2}{\sigma} + \frac{2\tilde \beta_l M_g}{\sigma m_f}\right) \right\rceil
\end{equation*}
ACG iterations, where  $m_f$,  $M_g$,  $\tilde \beta_l$, and $\kappa_1$   are as in  (A2),  \eqref{eq:Lipsc_tilde_def},  \eqref{def:ctilde-l}, and \eqref{def:newkappa1}, respectively;
\item[(b)] every cycle  performs 
at most 
$$
\left\lceil\frac{4m_fC_\sigma\left(\Delta \phi^* + 2m_fD_h\right)}{\hat \rho^2}\right\rceil
$$
outer iterations,
where  $C_\sigma$,  $D_h$, and $\Delta \phi^*$ are as in \eqref{def:lamb-C1},   \eqref{def:DhBf1},  and  \eqref{phi*-Deltaphi*}, respectively; 
 \item[(c)] the last cycle $\bar l$ outputs a $(\hat \rho,\hat\eta)$-approximate stationary quadruple of \eqref{eq:main_prb} and satisfies
 \[
\bar{l}\le  \log_1^+\left(\frac{4 \bar\beta}{\beta_1}\right), \quad \tilde \beta_{\bar l} \le \max\{\beta_1,2\bar{\beta}\}
\]
where $\bar{\beta}$ is as in  \eqref{def:c_bar_phi_star}.
\end{itemize}
\end{proposition}

Notice that if $\beta_1 > 4\bar{\beta}$, then Proposition~\ref{lem:StaticIPAAL}(c) implies the number of ACG iterations
of NL-IAPIAL is bounded above by the product of the quantities in Proposition~\ref{lem:StaticIPAAL}(a)--(b).
The next result bounds the number of ACG iterations of NL-IAPIAL when $\beta_1 \leq 4\bar{\beta}$. 

\begin{theorem}\label{theor:StaticIPAAL}
Suppose $\beta_1 \leq 4\bar{\beta}$. Then  NL-IAPIAL outputs   a $(\hat \rho,\hat\eta)$-approximate stationary quadruple of \eqref{eq:main_prb} in
\begin{equation}\label{eq:main complexity-bound}
{O}\left(\left[1 + \frac{m_fC_\sigma\left(\Delta \phi^* + m_fD_h\right)}{\hat \rho^2}\right]\left[\kappa_1+\frac{\kappa_2}{\hat\rho} + \frac{\kappa_3}{\sqrt{\hat\eta}}\right] (\log_1^+)^2 \left[\frac{\bar{\beta}}{\beta_1} + \frac{\kappa_1^2}{\sigma} + \frac{\bar \beta M_g}{\sigma m_f}\right]\right)
\end{equation}
ACG  iterations,
where $m_f$, $C_\sigma$,  $D_h$, $\Delta\phi^*$, $(\kappa_1,\kappa_2,\kappa_3)$, and  $\bar{\beta}$ are as in  (A2),   \eqref{def:lamb-C1},     \eqref{def:DhBf1}, \eqref{phi*-Deltaphi*}, \eqref{def:newkappa1}, and  \eqref{def:c_bar_phi_star}, respectively.
\end{theorem}

\begin{proof}
First recall  that in the  $l$-th cycle of NL-IAPIAL,  we have $\beta_k=\tilde \beta_l=2^{l-1}\beta_1$, for every $l \ge 1$ (see  \eqref{def:ctilde-l}).  
 Also, Proposition~\ref{lem:StaticIPAAL}(c) implies that   NL-IAPIAL  outputs  a $(\hat \rho,\hat\eta)$-approximate stationary quadruple of \eqref{eq:main_prb} in at most $
\bar{l}:= \lfloor\log_1^+(4 \bar\beta/\beta_1)\rfloor
$
 cycles. Hence, since  $\beta_1\leq 4\bar{\beta}$, we have  
\begin{equation*}
		\tilde\beta_l=2^{l-1}\beta_1 \leq 4\bar{\beta}, \qquad \forall l=1,\ldots, {\bar l}.
\end{equation*}
It now follows from the above inequality and the definition of $\bar{\beta}$ in \eqref{def:c_bar_phi_star} that  the number of ACG iterations performed by NL-IAPIAL at every outer iteration (see Proposition~\ref{lem:StaticIPAAL}~(a)) is
$$
{O}\left(\left[\kappa_1+ 
\frac{\kappa_2}{\hat\rho} +\frac{\kappa_3}{\sqrt{\hat\eta}}\right]\log_{1}^{+}\left[\frac{\kappa_1^2}{\sigma} + \frac{\bar{\beta} M_g}{\sigma m_f}\right]\right).
$$
The conclusion now follows from the above fact and Proposition~\ref{lem:StaticIPAAL}~(b)--(c). 
\end{proof}

It is worth mentioning that the iteration complexity bound in  Theorem~\ref{theor:StaticIPAAL}, in terms of the tolerance pair $(\hat\rho, \hat \eta)$, is
\[
{\cal O}_1\left(\left[\frac{1}{\sqrt{\hat \eta}\cdot\hat\rho^2} + \frac{1}{\hat\rho^3}\right] (\log_1^+)^2 \left(\frac{1}{\hat\eta} + \frac{1}{\hat\rho^2}\right)\right),
\]
as previously claimed in Section~\ref{sec:intro}.

\section{Proofs of Proposition~\ref{th:pkbounded} and Proposition~\ref{lem:StaticIPAAL}}\label{sec:proofs of propositions}

This section contains two subsections,
the first of which proves Proposition~\ref{lem:StaticIPAAL} and
the second one proves Proposition~\ref{th:pkbounded}.
It is worth noting that
the proof of Proposition~\ref{lem:StaticIPAAL} uses Proposition~\ref{th:pkbounded}, but the proof of Proposition~\ref{th:pkbounded} is 
self-contained.
Moreover, we opted to postpone the proof
of Proposition~\ref{th:pkbounded} due to its 
technicalities.

\subsection{Proof of Proposition~\ref{lem:StaticIPAAL}}  \label{sec:technical}

The first  result below presents some  relations about the iterates generated by NL-IAPIAL.

\begin{lem}\label{basicLemma:sk}  Let  $\{(z_k,p_k,\beta_k)\}$ be generated by NL-IAPIAL and define, for every $k\geq 1$,
\begin{equation} \label{eq:sk_def}
    s_k := \Pi_{-\cK}(p_{k-1} + {\beta_k}g(z_k)).
\end{equation}
Then, the following relations hold for every $k \ge 1$:
\begin{gather}
p_{k-1} + {\beta_k}g(z_k) = p_k + s_k, \quad \inner{p_k}{s_k} = 0, \quad (p_k,s_k) \in \cK^* \times (-\cK), \label{eq:sk_moreau} \\
{\cal L}_{\beta_k}(z_k,p_{k-1}) = \phi(z_k) + \frac1{2{\beta_k}} \left( \|p_{k}\|^2 - \|p_{k-1}\|^2 \right). \label{auxeq:declemma1}
\end{gather}
\end{lem}
\begin{proof}
The relations in \eqref{eq:sk_moreau} follow from the definitions of
$p_k$ and $s_k$ in \eqref{eq:dual_update2} and \eqref{eq:sk_def}, respectively,  and  Theorem~III.3.2.5  of  \cite{Hiriart1}.
Now, in view of the definitions of ${\mathcal L}_\beta$ in \eqref{eq:aug_lagr_def} and
$s_k$ in \eqref{eq:sk_def}, respectively,
we have
\[
{\cal L}_{\beta_k}(z_k,p_{k-1})
= \phi(z_k) + \frac1{2{\beta_k}} \left[ \|p_{k-1} + {\beta_k}g(z_k)  - s_k\|^2 - \|p_{k-1}\|^2 \right]
\]
which, in view of  the first identity in \eqref{eq:sk_moreau},
immediately implies \eqref{auxeq:declemma1}.
\end{proof}

The next technical result characterizes the change in the augmented Lagrangian between consecutive iterations of the NL-IAPIAL method.

\begin{lemma}\label{lem:declemma5}
The sequence $\{(z_k,p_k)\}$  generated by NL-IAPIAL satisfies, for every $k\geq 1$, the relations   
\begin{align}
{\cal L}_{\beta_k}(z_k,p_k)&\leq {\cal L}_{\beta_k}(z_k,p_{k-1})+\frac{1}{\beta_k}\|p_k-p_{k-1}\|^2, \label{auxineq:declemma2}\\
{\cal L}_{\beta_k}(z_k,p_k)&\leq {\cal L}_{\beta_k}(z_{k-1},p_{k-1})- \left(\frac{1-\sigma^2}{2\lambda} \right)\|r_k\|^2+ \frac{1}{\beta_k}\|p_k-p_{k-1}\|^2,\label{auxineq:declemma3}
\end{align}
 where $(\sigma,\lam)$ is given by the input of NL-IAPIAL and $\{r_k\}$ is as in \eqref{eq:dual_update2}.
 \end{lemma}
\begin{proof}
Let $s_k$ be as in \eqref{eq:sk_def}. Using \eqref{auxeq:declemma1}, the definition of ${\mathcal L}_\beta$ in \eqref{eq:aug_lagr_def}, the fact that $s_k \in {-\cK}$
and $p_{k-1}+\beta_kg(z_k) = p_k+s_k$ in view of \eqref{eq:sk_moreau},
we have that
\begin{align*}
{\cal L}_{\beta_k}(z_k,p_k) - {\cal L}_{\beta_k}(z_{k},p_{k-1}) &=
{\cal L}_{\beta_k}(z_k,p_k) - \phi(z_k) - \frac1{2{\beta_k}} \left( \|p_{k}\|^2 - \|p_{k-1}\|^2 \right) \\
&= \frac{1}{2{\beta_k}}\left({\rm dist}^{2}(p_k+{\beta_k}g(z_k),-{\cal K})- \|p_{k}\|^2 \right)  - \frac1{2\beta_k} \left( \|p_{k}\|^2 - \|p_{k-1}\|^2 \right) \\
&\le \frac{1}{2\beta_k} \left( \|p_k+\beta_kg(z_k)-s_k\|^2 - \|p_{k}\|^2 \right) - \frac1{2\beta_k} \left( \|p_{k}\|^2 - \|p_{k-1}\|^2 \right)  \\
&= \frac{1}{2\beta_k} \left( \|2p_k-p_{k-1}\|^2 - 2\|p_{k}\|^2 + \|p_{k-1}\|^2 \right),
\end{align*}
which immediately implies  \eqref{auxineq:declemma2}.
Now, in view of the definition of the $\varepsilon$-subdifferential given in \eqref{def:epsSubdiff} and  the fact that
 $(z_k,v_k,\varepsilon_k)$ satisfies both the inclusion and the inequality in \eqref{eq:prox_incl}, we conclude that
	\begin{align}
		& \lambda{\cal L}_{\beta_k}(z_k,p_{k-1}) - \lambda{\cal L}_{\beta_k}(z_{k-1},p_{k-1})
		\leq -\frac{1}{2} \|z_k-z_{k-1}\|^2 +
		\inner{v_{k}}{z_k-z_{k-1}} +\varepsilon_k \nonumber \\
		&= -\frac{1}{2} \|v_{k}+z_k - z_{k-1} \|^2 + \frac{1}2\|v_{k}\|^2 + \varepsilon_k  \leq -\left(\frac{1-\sigma_k^2}{2}\right) \|r_k\|^2
		\le -\left(\frac{1-\sigma^2}{2}\right) \|r_k\|^2, \label{eq:aux_DeltaLagr_bd1}
		\end{align}
		where the last inequality follows from
		the  fact that $\sigma_k\leq \sigma$ in view of  \eqref{eq:acg_input_aux_defs}.
		Inequality \eqref{auxineq:declemma3} now follows by combining \eqref{auxineq:declemma2} with \eqref{eq:aux_DeltaLagr_bd1}.
 \end{proof}
 
Recall that the
$l$-th cycle ${\cal C}_l$ of NL-IAPIAL is  defined
in \eqref{def:ctilde-l}. 
The next results present some properties of the iterates generated during an NL-IAPIAL cycle. 
The first one shows that  the sequence $\{\|r_k\|\}_{k\in {\cal C}_l}$ is bounded and can be controlled  by   $\{\Delta_k\}_{k\in {\cal C}_l}$ plus a term which is of $\mathcal{O}(1/{\tilde \beta_l})$.

\begin{lemma}\label{lem:minrk-Deltak}
Consider the sequences
$\{(z_k,v_k,\varepsilon_k)\}$ and $\{\Delta_k\}$ generated by NL-IAPIAL and the sequence
$\{r_k\}$   as  in \eqref{eq:dual_update2}. Then, the following statements hold:

\begin{itemize}
    \item[(a)] for every $k\geq 1$, we have
\begin{equation}\label{ineq:rkbound}
 \|r_k\| \leq \frac{D_{h}}{1-\sigma};
   \end{equation}
   \item [(b)]  $k\in {\cal C}_l$ and $k\geq k_{l-1}+2$, there exists an index $j\in \{k_{l-1}+2,\ldots, k\}$ such that 
  
\begin{equation}\label{main-ineq-dreasingLag}
\|r_j\|^2\leq \frac{2\lambda}{1-\sigma^2}\left(\Delta_k + \frac{9\kappa_p^2}{\tilde \beta_l}\right),
\end{equation}
where  $\sigma$, $\kappa_p$,   and $D_h$ are  as in   \eqref{eq:acg_input_aux_defs},  \eqref{ineq:pkbounded},   and \eqref{def:DhBf1}, respectively.
\end{itemize}
\end{lemma}
\begin{proof}
(a)The definition of $\sigma_k$ in \eqref{eq:acg_input_aux_defs}, the inequality in \eqref{eq:prox_incl}, the triangle inequality for norms, and the fact that $z_k,z_{k-1}\in{\cal H}$ imply that 
\[
\|r_k \| = \|v_k + z_{k-1} - z_k\| \leq \|v_k\| + D_h \leq \sigma_k\|r_k\| + D_h \leq \sigma \|r_k\| + D_h,
\]
which, after a simple re-arrangement, proves \eqref{ineq:rkbound}.

\noindent (b) Now, to simplify notation, let $\bar k = k_{l-1}+1$.
Now, using  \eqref{ineq:pkbounded} and the fact that   $\|p_j-p_{j-1}\|^2\leq 2\|p_j\|^2+2\|p_{j-1}\|^2$, it follows that for any $k\geq \bar{k}+1$,
\begin{equation}\label{eq:auxboundpk2}
\frac{\|p_k\|^2}{2}+ \sum_{j={\bar k}}^k\|p_j-p_{j-1}\|^2
\le \frac{\kappa_p^2}{2}+ 4(k-\bar k+1)\kappa_p^2= \frac{(1+ 8(k-\bar k+1))\kappa_p^2}{2}\leq 9(k-\bar k)\kappa_p^2.
\end{equation}
Hence, 
\eqref{auxineq:declemma2} with $k=\bar{k}$, \eqref{auxineq:declemma3},  \eqref{eq:auxboundpk2}, and the fact that $\beta_{k}=\tilde \beta_l$ for every $k\in {\cal C}_l$,  imply that for any $k\in {\cal C}_l$ such that $k\geq \bar k+1$,
\begin{align*}
	\frac{(1-\sigma^2)}{2\lambda}\sum_{j=\bar k+1}^k\|r_j\|^2 
	& \overset{\eqref{auxineq:declemma3}}{\leq} \sum_{j=\bar k+1}^k\left [
	{\cal L}_{\beta_j}(z_{j-1},p_{j-1})-{\cal L}_{\beta_j}(z_j,p_j)+ \frac{1}{\beta_j}\|p_j-p_{j-1}\|^2 \right]
	\\
    & \overset{j \in {\cal C}_l}{=} \sum_{j=\bar k+1}^k\left [
	{\cal L}_{\tilde \beta_l}(z_{j-1},p_{j-1})-{\cal L}_{\tilde \beta_l}(z_j,p_j)+ \frac{1}{\tilde \beta_l}\|p_j-p_{j-1}\|^2 \right] \\
	& \leq {\cal L}_{\tilde \beta_l}(z_{\bar k},p_{\bar k})-{\cal L}_{\tilde \beta_l}(z_k,p_k)+ \frac{1}{\tilde \beta_l}\sum_{j={\bar k+1}}^k\|p_j-p_{j-1}\|^2 \\
	& \overset{\eqref{auxineq:declemma2}}{\leq} {\cal L}_{\tilde \beta_l}(z_{\bar k},p_{\bar k-1})-{\cal L}_{\tilde \beta_l}(z_{k},p_{k})+ \frac{1}{\tilde \beta_l}\sum_{j={\bar k}}^k\|p_j-p_{j-1}\|^2 \\
	& \overset{\eqref{eq:auxboundpk2}}{\leq} {\cal L}_{\tilde \beta_l}(z_{\bar k},p_{\bar k-1})-{\cal L}_{\tilde \beta_l}(z_{k},p_{k}) - \frac{\|p_k\|^2}{2 \tilde \beta_l} + \frac{9(k-\bar k)\kappa_p^2}{\tilde \beta_l} \\
	& =  (k-\bar k) \left[ \Delta_k + \frac{9\kappa_p^2}{\tilde \beta_l} \right],
	\end{align*}
	where the last equality follows from the definition of $\Delta_k$  in \eqref{eq:stopCOndition-deltak} and the fact that $\hat{k}=\bar{k}-1$.
The proof of \eqref{main-ineq-dreasingLag}  now follows by dividing the above inequality by $(k-\bar k)(1-\sigma^2)/(2\lam)$ and by taking $j$ such that $\|r_j\|=\min_{\bar k+1\leq j\leq k}\|r_j\|$.
\end{proof}

The next result, whose proof can be found in Appendix~\ref{app:refine}, contains  some useful relations about the sequence $\{(\hat z_k,\hat p_k, \hat w_k,\hat q_k)\}$ generated by NL-IAPIAL.
\begin{lemma}
\label{prop:refinement}
Consider the sequences  $\{(\hat z_k,\hat p_k, \hat w_k,\hat q_k)\}$, $\{p_k\}$, and $\{r_k\}$   generated by NL-IAPIAL.
Then, for every $k\geq 1$, we have:
\begin{equation}\label{eq:Relations(a)-PropRef}
\hat{w}_k\in\nabla f(\hat{z}_k)+\pt h(\hat{z}_k)+\nabla g(\hat{z}_k)\hat{p}_k, \quad \inner{g(\hat z_k) + \hat q_k}{\hat p_k} = 0, \quad  g(\hat z_k) + \hat q_k \preceq_{\cK} 0, \quad\hat{p}_k\succeq_{\cK^{*}} 0, 
\end{equation}
\begin{equation}
\begin{gathered}
\|\hat{w}_{k}\|\leq \frac{1}{\lam} \left(1+2\sigma\right)\|r_{k}\|, \quad \|\hat{q}_k\|\leq\frac{B_g^{(1)}\sigma}{\widetilde{\cal M}_k}\|r_{k}\| + \frac{1}{\beta_k}\|p_k - p_{k-1}\|,
\end{gathered}\label{eq:strong_refine}
\end{equation}
where  $B_g^{(1)}$  is as in \eqref{eq:bd_Psi_val} and $({\widetilde{\cal M}_k},\sigma)$ is given  in \eqref{eq:acg_input_aux_defs}.

\end{lemma}

Some comments about Lemma~\ref{prop:refinement} are in order.
First, in view of the fact that \eqref{eq:Relations(a)-PropRef} implies that the quadruple
$(\hat{z},\hat{p},\hat{w},\hat q)=(\hat{z}_{k},\hat{p}_{k},\hat{w}_{k},\hat q_k)$
satisfies all the relations in \eqref{eq:rho_eta_approx_soln}, it follows
that such a quadruple becomes a
$(\hat \rho,\hat\eta)$-approximate stationary quadruple of \eqref{eq:main_prb}
whenever $\|\hat w_k\| \le \hat \rho$ and $\|\hat q_k\|\leq \hat\eta$.
The inequalities in \eqref{eq:strong_refine} provide useful bounds
for these residual pair  in terms of
$\|r_k\|$ and $ \|p_k-p_{k-1}\|/\beta_k$ which  are used to prove that  $\{(\hat w_k,\hat q_k)\}$  eventually approaches zero. Hence, the latter two inequalities will eventually be satisfied, which implies that  NL-IAPIAL  computes a $(\hat \rho,\hat\eta)$-approximate stationary quadruple of \eqref{eq:main_prb}  after a finite number of iterations.

The next result shows that during an $l$-th cycle of  NL-IAPIAL, the residual sequence $\{(\hat w_k,\hat q_k)\}$ can be controlled  by  $\tilde \beta_l$ and  $\{\Delta_k\}$ defined in \eqref{eq:stopCOndition-deltak}. 

\begin{lemma}\label{lem:what-Deltak} Consider the sequence  $\{(\hat w_k,\hat{q}_k)\}_{k\in {\cal C}_l}$  generated during the $l$-th cycle of  NL-IAPIAL.  Then, 
 for every $k\in {\cal C}_l$ and $k\geq k_{l-1}+2$, there exists an index $j\in \{k_{l-1}+2,\ldots, k\}$ such that 
 \begin{equation}\label{ineq:what5}
 \|\hat w_j\|^2 \le  2m_fC_\sigma\Delta_k + \frac{m_f\kappa_2^2}{ 2M_g \tilde \beta_l}, \qquad \|\hat{q}_j\| \le \frac{m_f\kappa_3^2}{M_g\tilde \beta_l},
\end{equation}
where $C_\sigma$, $\Delta_k$, and $(\kappa_2,\kappa_3)$  are as in \eqref{def:lamb-C1}, \eqref{eq:stopCOndition-deltak}, and     \eqref{def:newkappa1}, respectively. 
 \end{lemma}
\begin{proof}
First, recall that for any $k\in {\cal C}_l$, we have $\beta_k=\tilde \beta_l$ in view of \eqref{def:ctilde-l}. Hence, the proof of the first inequality in \eqref{ineq:what5} for some $j\in\{k_{l-1}+2,\ldots,k\}$ follows immediately from Lemma~\ref{lem:minrk-Deltak}(b), the first inequality in \eqref{eq:strong_refine}, and the definitions of $(C_\sigma,\lambda)$ and $\kappa_2$ in \eqref{def:lamb-C1} and \eqref{def:newkappa1}, respectively. 
Now,  from the second inequality in  \eqref{eq:strong_refine}, the definition of $\lambda$ in \eqref{def:lamb-C1}, the triangle inequality for norms,  Proposition~\ref{th:pkbounded}, \eqref{ineq:rkbound}, and the fact that   
$\widetilde {\mathcal M}_k \geq \lam \tilde \beta_l M_g$ (see  \eqref{eq:Lipsc_tilde_def} and \eqref{eq:acg_input_aux_defs}), we have 
\begin{align*} 
    \|\hat q_j\|
     &\leq \frac{B_g^{(1)} \sigma}{\widetilde M_k}\|r_j\| + \frac{1}{\tilde \beta_l} \left(\|p_j\| + \|p_{j-1}\| \right) \leq \frac{\sigma B_g^{(1)} D_h}{\lam (1-\sigma) M_g\tilde \beta_l} + \frac{2\kappa_p}{\tilde \beta_l}\\
     &=\left(\frac{\sigma B_g^{(1)} D_h}{1-\sigma} + \frac{ M_g\kappa_p}{m_f}\right)\frac{2m_f}{ M_g\tilde \beta_l}. 
\end{align*}
On the other hand, it follows from  the fact that $ B_g^{(1)}\leq \sqrt{M_g} $ (see  \eqref{eq:Lipsc_tilde_def}) and the definitions of $\theta_h$, $\kappa_0$, and $\kappa_p$ in \eqref{def:DhBf1}, \eqref{def:kappa00}, and \eqref{ineq:pkbounded}, respectively, that
\begin{align*}
\frac{\sigma B_g^{(1)} D_h}{1-\sigma}\leq \frac{\sigma  \min\{1,\bar{d}\} \theta_h\sqrt{M_g}}{1-\sigma} 
     \leq \frac{\sigma D_h \theta_h\sqrt{M_g}}{1-\sigma}\leq \frac{ \kappa_0 \theta_h\sqrt{M_g}}{4 m_f}\leq \frac{\tau_g \kappa_p\sqrt{M_g}}{4m_f}. 
\end{align*}
Hence, we conclude that 
$$
\|\hat q_j\|\leq \left(\tau_g\sqrt{M_g}+4M_g \right)\frac{\kappa_p}{2M_g\tilde \beta_l} \quad \forall j \in \{k_{l-1}+2,\ldots, k\},
$$
which, together with the previous conclusion about $\|\hat{w}_j\|$ and the definition of $\kappa_3$ in \eqref{def:newkappa1}, implies the existence of an index $j\in \{k_{l-1}+2,\ldots, k\}$ satisfying \eqref{ineq:what5}.
\end{proof}
 
The next result establishes the rate in which the sequence $\{\Delta_k\}$ defined in \eqref{eq:stopCOndition-deltak} converges to zero 

\begin{lemma}\label{lem:upperboundDeltak} Consider the sequence $\{(z_{k},p_{k})\}_{k\in {\cal C}_l}$ generated during the $l$-th cycle of  NL-IAPIAL and let $\Delta_k$ be as in \eqref{eq:stopCOndition-deltak}. Then, for every $k\in {\cal C}_l$ and $k\geq k_{l-1}+2$, we have
		\begin{equation*}
\Delta_k \le \frac{\Delta \phi^* + 2m_fD_h}{k-k_{l-1}-1},
		\end{equation*}
where  $D_h$, $\Delta \phi^*$, and  $m_f$  are as in  \eqref{def:DhBf1},   \eqref{phi*-Deltaphi*}, and (A2), respectively.
	\end{lemma}
	\begin{proof} 
	From  step~1 of NL-IAPIAL we have that $(\lambda,z_k,v_k,\varepsilon_k,\sigma_k)$ satisfies  \eqref{eq:prox_incl}. Moreover,  we also have $1-2\sigma_k^2\geq 0$ due to $\sigma_k \le \sigma \in (0,1/\sqrt{2}]$ (see  NL-IAPIAL input and \eqref{eq:acg_input_aux_defs}).
Hence,  it follows from Lemma~\ref{lem:auxNewNest2} with $\tilde{\phi} =\lambda {\cal L}_{\beta_k}(\cdot,p_{k-1})$,  $(\tilde \sigma,s)=(\sigma_k,1)$, and  $(x_0,x)=(z_{k-1},z_k)$  that   
		\begin{equation}\label{ineq:Lz1p0}
		\lambda{\cal L}_{\beta_k}(z_k,p_{k-1}) \leq   \lambda{\cal L}_{\beta_k}(z,p_{k-1}) +\|z-z_{k-1}\|^2, \qquad \forall  z\in {\mathcal H}.
  		\end{equation}
Since the definition of ${\cal L}_\beta$ in \eqref{eq:aug_lagr_def} implies that  ${\cal L}_{\beta_k}(z,p_{k-1})\leq\phi(z)$ for every  $z\in { \mathcal F}:=\{z\in {\mathcal H}:g(z)\preceq_\mathcal{K} 0\}$, it follows from \eqref{ineq:Lz1p0} and the definitions of  $\phi^*$ and $D_h$ in \eqref{eq:main_prb} and \eqref{def:DhBf1}, respectively,  that  
		\begin{equation}\label{eq:aux:2345}
{\cal L}_{\beta_k}(z_k,p_{k-1}) \leq \phi^*  + \frac{D_h^2}{\lambda}.
		\end{equation}
Now, in view of  the definitions of $\mathcal{L}_\beta$ and $\phi_*$ given in \eqref{eq:aug_lagr_def} and \eqref{def:c_bar_phi_star}, respectively, we  have
\[
{\cal L}_{\beta_k}(z_k,p_k)
+ \frac{\|p_k\|^{2}}{2\tilde \beta_l} =\phi(z_k)+\frac{1}{2\tilde \beta_l}{\rm dist}^{2}(p_k+\tilde \beta_lg(z_k),-{\cal K})\geq \phi_*.
\]
Since the $l$-th cycle ${\mathcal C}_l$ starts at iteration $k_{l-1}+1$ and  $\beta_k=\tilde \beta_l$ for any ${k\in \cal C}_l$, it follows from the definition of   $\Delta_k$ given in \eqref{eq:stopCOndition-deltak},  \eqref{eq:aux:2345} with $k=k_{l-1}+1$, and the above inequality that 
\begin{align*}
\Delta_k = \frac{1}{k-k_{l-1}-1}\left({\cal L}_{\tilde \beta_l}(z_{k_{l-1}+1},p_{k_{l-1}})- {\cal L}_{\tilde \beta_l}(z_k,p_k)-\frac{\|p_k\|^{2}}{2\tilde \beta_l}\right)\leq \frac{1}{k-k_{l-1}-1}\left(\phi^*+\frac{D_h^2}{\lambda}-\phi_*\right),
\end{align*}
which proves the lemma in view of the definitions of $\lambda$ and $\Delta \phi^*$ in \eqref{def:lamb-C1} and   \eqref{def:c_bar_phi_star},  respectively.
\end{proof}

Now we are ready to present the proof of Proposition~\ref{lem:StaticIPAAL}.

\begin{proof}[Proof of Proposition~\ref{lem:StaticIPAAL}]

(a) First note that NL-IAPIAL calls in its step~1 the ACG algorithm of Appendix~\ref{sec:acg} with inputs given by \eqref{eq:acg_inputs}. Note also that  within the $l$-th cycle, we have $\beta_k=\tilde \beta_l$ in view of  \eqref{def:ctilde-l}. Hence,  since $m_f\leq L_f$ (see (A2)),  we conclude that (a) follows from Lemma~\ref{lem:kthACGbound} and the fact that \eqref{ineq:pkbounded} and the definitions of  $\widetilde{\cal M}_k$, $\lambda$, and  $\kappa_1$ given in   
   \eqref{eq:acg_input_aux_defs}, \eqref{def:lamb-C1},  and \eqref{def:newkappa1}, respectively, imply that 
\begin{align*}
\widetilde{\cal M}_k&= \lam\left(L_{f}+L_{g}\|p_{k-1}\|+\beta_k M_g\right)+1 \nonumber \\
&\leq \lam\left(L_{f}+L_g\kappa_p+\tilde \beta_l M_g\right)+1\leq \kappa_1^2+\frac{\tilde \beta_l M_g}{2m_f}.
\label{auxestimatesLpsi-tc}
\end{align*}
(b) Fix a cycle $l$ and note that $\hat k$ in step~3 corresponds to  $\hat k = k_{l-1}$. It follows from  Lemma~\ref{lem:upperboundDeltak} that,  for every $ k\in {\mathcal C}_l$ and  $k\geq \hat{k}+2$, 
$$
\Delta_k \le \frac{\Delta \phi^* + 2m_fD_h}{k-\hat{k}-1}.
$$
 Hence,   we have that if some $ k\in {\mathcal C}_l$ is such that 
\begin{equation}
k> \hat{k}+1 + \frac{2C_\sigma\left(\Delta \phi^* +2m_fD_h\right)}{\lambda\hat \rho^2} \label{eq:k_prf_bd}
\end{equation}
then $\Delta_k$ satisfies inequality \eqref{eq:stopCOndition-deltak}, ending the $l$-th cycle. Hence,   (b) follows immediately from this conclusion,  the definition of $\lambda$ in \eqref{def:lamb-C1}, and the fact that the $l$-th cycle starts at $\hat{k}+1$.

(c) 
First, recall that  in the  $l$-th cycle of NL-IAPIAL,  we have $\beta_k=\tilde \beta_l=2^{l-1}\beta_1$, for every $l \ge 1$ (see \eqref{def:ctilde-l}). 
If NL-IAPIAL  performs just one cycle then $\bar l=1$ and then the result immediately follows from \eqref{eq:Relations(a)-PropRef}, the stopping criterion in step~2 and Definition~\ref{def:stationarypoint}.
Assume then that NL-IAPIAL  performs more than one cycle. We argue that NL-IAPIAL stops before or at the first cycle $\bar{l}$ where $\tilde \beta_{\bar l} \geq \bar{\beta}(\hat \rho, \hat \eta)$ and 
 $\bar{\beta}(\hat \rho, \hat \eta)$ is as in \eqref{def:c_bar_phi_star}. Suppose that the algorithm has not stopped before a cycle $\bar{l}$, and note that the definition of $\bar{\beta}(\hat \rho, \hat \eta)$  in \eqref{def:c_bar_phi_star} implies
\begin{equation}
\tilde \beta_{\bar l} \geq \frac{m_f}{M_g} \left( \frac{\kappa_2^2}{\hat\rho^2} + \frac{\kappa_3^2}{\hat\eta} \right), \label{eq:lower_c_bds}
\end{equation}
where $\kappa_2$ and $\kappa_3$ are as in \eqref{def:newkappa1}. 
Now, if at the $\bar{l}$-th cycle, NL-IAPIAL performs at least $\bar{k}\geq k_{\bar{l}-1}+2$ outer iterations, where $\bar{k}$ is  the smallest index such that
\begin{equation}
\frac{2m_fC_\sigma\left(\Delta \phi^* + 2m_fD_h\right)}{\bar{k}-k_{\bar{l}-1}-1}\leq \frac{\hat\rho^2}{2}, \label{eq:bar_k_bd}
\end{equation}
then, in view of  \eqref{ineq:what5}, Lemma~\ref{lem:upperboundDeltak}, \eqref{eq:lower_c_bds}, and \eqref{eq:bar_k_bd},   there  exists an index $j\in \{k_{\bar{l}-1}+2,\ldots, \bar k\}$ such that
\[
\|\hat w_j\|^2 \leq 2 m_f C_\sigma \Delta_{\bar k} + \frac{m_f\kappa_2^2}{2M_g \tilde{\beta}_{\bar l}} 
\leq \frac{2m_fC_\sigma\left(\Delta \phi^* + 2m_fD_h\right)}{\bar{k}-k_{\bar{l}-1}-1} + 
\frac{\kappa_2^2}{2}\left( \frac{\kappa_2^2}{\hat\rho^2} + \frac{\kappa_3^2}{\hat\eta} \right)^{-1}
\le \frac{\hat{\rho}^2}{2} + \frac{\hat{\rho}^2}{2} = \hat{\rho}^2,
\]
and also
\[
\|\hat q_j\|  \le \frac{m_f \kappa_3^2}{M_g \tilde{\beta}_{\bar l}}
\leq \kappa_3^2\left( \frac{\kappa_2^2}{\hat\rho^2} + \frac{\kappa_3^2}{\hat\eta} \right)^{-1}
\leq \hat\eta.
\]
More specifically, since we assumed that at least $\bar{k}$ iterations are performed,
we have $j=\bar{k}$. Hence, NL-IAPIAL must stop before or on iteration $\bar{k}$ within the $\bar{l}$ cycle, in view of the stopping criterion in  step~2. In view of step~3 of NL-IAPIAL, we then have that
$$
\beta_k=\tilde \beta_l=2^{l-1}\beta_1 \leq 2\bar{\beta}, \quad \forall l \leq  \bar l.
$$
The conclusion now follows from the above bound,
step~2 of NL-IAPIAL, \eqref{eq:Relations(a)-PropRef}, and Definition~\ref{def:stationarypoint}.
\end{proof}

\subsection{Proof of Proposition~\ref{th:pkbounded}}\label{subsec:boundMultplier} 

The first lemma
describes some basic facts  about the sequence   $\{(z_{k},p_k,w_{k},r_k,\varepsilon_{k})\}$  generated by NL-IAPIAL.

\begin{lemma}
\label{prop:refinement1}
Consider the sequence $\{(z_{k},p_k,w_{k},r_k,\varepsilon_{k})\}$  generated by NL-IAPIAL.
Then,  the following statements hold for every $k\geq 1$:
\begin{itemize}
    \item[(a)] the quintuple  $(z_{k},p_k,w_{k},r_k,\varepsilon_{k})$ satisfies
\begin{equation}
\begin{gathered}
w_{k}\in\nabla f(z_{k})+\pt_{(\lambda^{-1}\varepsilon_k)}h(z_{k})+\nabla g(z_{k})p_{k}, \\[2mm]
\|w_{k}\| \leq \frac{1}{\lam}\left(1+\sigma\right)\|r_{k}\|, \quad  \varepsilon_{k} \leq \frac{\sigma^{2} }{2}\|r_{k}\|^{2};
\end{gathered}\label{eq:weak_refine}
\end{equation}
    \item[(b)] the residual pair $(w_k,\varepsilon_k)$ satisfies \begin{equation}\label{ineq:rk-deltak-wk}
\varepsilon_{k}\leq \frac{\sigma^2 D_h^2}{2(1-\sigma)^2}, \quad\|w_{k}\|\leq \left(\frac{1+\sigma}{1-\sigma}\right) \frac{D_{h}}{\lam},
   \end{equation}
where  $\sigma$ and $D_h$ are  as in   \eqref{eq:acg_input_aux_defs} and \eqref{def:DhBf1}, respectively.
\end{itemize}
\end{lemma}
\begin{proof}
(a) The proof of this statement is presented in  Appendix~\ref{app:refine}. 

(b) The first inequality in \eqref{ineq:rk-deltak-wk} follows by combining  \eqref{ineq:rkbound}, the inequality in \eqref{eq:prox_incl}, and the definition of $\sigma_k$  in \eqref{eq:acg_input_aux_defs}. The last inequality in \eqref{ineq:rk-deltak-wk}  follows from \eqref{ineq:rkbound} and  the first inequality in \eqref{eq:weak_refine}.
\end{proof}

The following technical result, whose proof can be found in Lemma 3.10 of \cite{RenWilmelo2020iteration}, plays an important role in the proof of Lemma~\ref{lem:pre-pkbound} below.

\begin{lemma}\label{lem:bound_xiN}
	Let $h$ be a function as in (A1).
	Then, for every $ u,z\in {\mathcal H}$, $\delta \ge 0$,  and $\xi \in \partial_{\delta} h(z)$, we have
	    \[
	\|\xi\|{\rm dist}(u,\partial {\mathcal H}) \le \left[{\rm dist}(u,\partial {\mathcal H})+\|z-u\|\right]K_h + \inner{\xi}{z-u}+\delta,
	    \]
where $\partial {\cal H}$ denotes the boundary of ${\cal H}$.
\end{lemma}

The idea behind the proof of Lemma 3.10 of \cite{RenWilmelo2020iteration}
is based on the following two observations:
i) any $h$ as in (A1) satisfies
the
condition
that $\partial_\varepsilon h(z) \subset {\mathcal N}^\varepsilon_{\mathcal H}(z)+ \bar B(0,K_h)$
(see Lemma A.2(ii) of \cite{RenWilmelo2020iteration});
and, ii) any closed convex function
satisfying the latter condition
satisfies the conclusion of Lemma~\ref{lem:bound_xiN}.
It is worth mentioning that the proof of the second observation
uses a technical inequality that appears in
the proof of Lemma~3 of 
\cite{PPmetNonconvex2019}.

The following technical result, whose proof is based
on the two previous lemmas, is used in Lemma~\ref{lem:pre-pkbound} to derive a recursive
formula below relating $p_{k-1}$ and $p_k$.

\begin{lemma}\label{Lem:xik}
Consider the sequence $\{(z_k,p_k)\}$ generated by NL-IAPIAL and let $\bar{z}$, $\kappa_0$, and $\bar{d}$ be  as in (A4),  \eqref{def:kappa00}, and \eqref{def:DhBf1}, respectively. Then, the following inequality holds 
\begin{equation}\label{ineq:aux9001}
\inner{\nabla g(z_k)p_k}
{z_k-\bar z} \leq  D_h\kappa_0- \bar{d}\| \nabla g(z_k) p_k\|,\qquad \forall k\geq 1.
\end{equation}
\end{lemma}	
\begin{proof}
Let $\{(z_k,p_k,w_k)\}$ be generated by NL-IAPIAL and note that, in view of the inclusion in \eqref{eq:weak_refine}, we have  
$w_{k}-\nabla f(z_{k})-\nabla g(z_{k})p_{k} \in \partial_{(\lambda^{-1}\varepsilon_k)} h(z_k)$ for every $k\geq 1$. 
Hence, it follows from
the definition of $\bar d$, and
Lemma~\ref{lem:bound_xiN} with $\xi=w_{k}-\nabla f(z_{k})-\nabla g(z_{k})p_{k}$, $z=z_k$,  $u=\bar z$ and $\delta=\lambda^{-1}\varepsilon_k$, that 
 \begin{align*}
     \bar{d}\|w_{k}-\nabla f(z_{k})-\nabla g(z_{k})p_{k}\| &\leq \left(\bar d +\|z_k-\bar{z}\|\right)K_h+\inner{w_{k}-\nabla f(z_{k})-\nabla g(z_{k})p_{k}}{z_k-\bar z}+ \frac{\varepsilon_k}{\lambda}\\[2mm]
&\leq  (\bar d+D_h) K_h  - \inner{\nabla g(z_{k})p_{k}}{z_k-\bar z}
+  \|w_k - \nabla f(z_k)\|D_h+\frac{\varepsilon_k}{\lambda},
 \end{align*}
 where the last inequality is due to Cauchy-Schwarz inequality and the fact that $\|z_k-\bar{z}\|\leq D_h$ (in view of $\bar z, \, z_k \in {\cal H}$ and the definition of $D_h$ in \eqref{def:DhBf1}). Now,  using  the reverse triangle inequality for norms and rearranging the resulting inequality, we have 
 \begin{align*}
     \inner{\nabla g(z_{k})p_{k}}{z_k-\bar z} +
  \bar d \|\nabla g(z_{k})p_{k}\| 
&\leq  (\bar{d} + D_h) K_h 
+  \|w_k - \nabla f(z_k)\|\, (\bar d +D_h)+ \frac{\varepsilon_k}{\lambda}  \\[2mm]
&\leq 2D_h K_h +  2\left(\frac{(1+\sigma)D_h}{\lambda(1-\sigma)}  +   B_f^{(1)}\right)D_h +        \frac{\sigma^2D_h^2}{2\lambda(1-\sigma)^2}
 \end{align*}
where the last inequality is due to the definition of $B_f^{(1)}$ in  \eqref{def:DhBf1}, the inequalities in \eqref{ineq:rk-deltak-wk}, and  the fact  $\bar{d}\leq D_h$.
Hence,   \eqref{ineq:aux9001} follows in view of the definition of $\kappa_0$ in \eqref{def:kappa00}.
\end{proof}

We are now ready to show that the sequence  $\{p_k\}$ is bounded.

\begin{lemma}\label{lem:pre-pkbound} Consider the sequence  $\{(p_k,\beta_k)\}$ generated by  NL-IAPIAL and let $\kappa_0$, $\tau_g$, and $\bar{d}$ be as in \eqref{def:kappa00}, (A4) and  \eqref{def:DhBf1}, respectively. Then,  for every $k\geq 1$, we have 
    \begin{equation}\label{ineq:pre-boundpk}
        \min\{1,\bar d\}\tau_g\|p_k\| +\frac{\|p_k\|^2}{\beta_{k}}\le D_h\kappa_0+\frac{1}{\beta_{k}} \inner{p_k} {p_{k-1}}.
    \end{equation}
\end{lemma}
	
\begin{proof}
First note that  the first two identities in \eqref{eq:sk_moreau} imply that
\begin{equation*}
\inner{p_k}{g(z_k)} = \frac1{\beta_k} \inner{p_k}{s_k+p_k-p_{k-1}}=  \frac{\|p_k\|^2}{\beta_k} - \frac1{\beta_k} \inner{p_k} {p_{k-1}}.
\end{equation*}
Using this identity, \eqref{ineq:aux9001}, the fact that $p_k\in \cK^*$, and
relation \eqref{ineq:gradineq-Kconvexity} with $(z,z',p)=(z_k,\bar{z},p_k)$, 
we conclude that
\begin{align*}
D_h\kappa_0- \bar{d}\| \nabla g(z_k) p_k\|&\overset{\eqref{ineq:aux9001}}{\geq} \inner{\nabla g(z_k)p_k}
{z_k-\bar z} = \inner{p_k} {g'(z_k)(z_k-\bar z)}\\
&\overset{\eqref{ineq:gradineq-Kconvexity}}{\geq} \inner{p_k} {g(z_k)} - \inner{p_k}{g(\bar z)}
= \frac{\|p_k\|^2}{\beta_k} - \frac1{\beta_k} \inner{p_k} {p_{k-1}} +|\inner{p_k}{g(\bar z)}|,
\end{align*}
or equivalently,  
$$
 \bar d \| \nabla g(z_k) p_k\| 
+|\inner{p_k}{g(\bar z)}|  +\frac{\|p_k\|^2}{\beta_k}\le D_h\kappa_0  + \frac1{\beta_k} \inner{p_k} {p_{k-1}}.
$$
 Inequality \eqref{ineq:pre-boundpk}  now
 follows from \eqref{eq:gen_slater} and the latter inequality.
\end{proof}	

Based on the recursive formula \eqref{ineq:pre-boundpk}, we are now ready to give the proof  of Proposition~\ref{th:pkbounded}.

\begin{proof}[Proof of Proposition~\ref{th:pkbounded}.]
The proof is done  by induction.
Inequality \eqref{ineq:pkbounded} trivially  holds for $k=0$.  Assume that \eqref{ineq:pkbounded} holds with $k=i-1$ for some $i\geq 1$. This
assumption together with \eqref{ineq:pre-boundpk},   the Cauchy-Schwarz inequality, and the definitions of $\theta_h$ and $\kappa_p$ in \eqref{def:DhBf1} and \eqref{ineq:pkbounded}, respectively,  imply that
		\begin{align*}
		\left(\min\{1,\bar d\}\tau_g +\frac{\|p_{i}\|}{\beta_{i}}\right)\|p_{i}\| &\leq D_h\kappa_0+\frac{\|p_{i}\|\cdot\|p_{i-1}\|}{ \beta_{i}}\leq D_h\kappa_0+\frac{\|p_{i}\|\kappa_p}{\beta_{i}} \\
		&=\min\{1,\bar d\}\tau_g\frac{\theta_h\kappa_0}{\tau_g}+\frac{\|p_{i}\|\kappa_p}{\beta_{i}} \leq\left(\min\{1,\bar d\}\tau_g+\frac{\|p_{i}\|}{\beta_{i}}\right)\kappa_p,
		\end{align*}
which implies that   $\|p_{i}\|\leq \kappa_p$.
Then,  \eqref{ineq:pkbounded} also holds with $ k=i$ and hence, by induction,  we conclude that \eqref{ineq:pkbounded} holds for the whole sequence $\{p_k\}$. 
\end{proof}

\section{Numerical Experiments} \label{sec:numerical}

This section presents numerical experiments that highlight the performance
of two variants of NL-IAPIAL, named IPL and IPL(A), against six other benchmark methods for solving
NCO problems with linear or nonlinear convex constraints. It contains
five subsections. The first four present the numerical results on different classes of constrained NCO problems, while the last one
contains a summary and some comments. For replication purposes, the 
MATLAB code for generating the results of this section is available online\footnote{See the examples in \texttt{./tests/papers/nl-IAPIAL} from the GitHub repository
\href{https://github.com/wwkong/nc_opt/}{https://github.com/wwkong/nc\_opt/}.}.

Before proceeding, we first precisely describe the implementations of NL-IAPIAL. The IPL and IPL(A) variants considered differ from the description in Section~\ref{sec:aug_lagr} in two important ways. First, they both modify the parameter $\tilde{\sigma}$
that is given to the ACG algorithm in its step~1. More specifically,
instead of choosing $\tilde{\sigma}=\sigma_k$ at the $k$-th
iteration, the implementation chooses $\tilde{\sigma}=\min\{\nu/(\widetilde{\cal M}_k)^{1/2},\sigma\}$
for $\nu\gg0$. 
Second, in view of the first modification, they both replace condition \eqref{eq:stopCOndition-deltak}
with the modified condition
\[
\Delta_{k}\leq\frac{\lam(1-\sigma^{2})\hat{\rho}^{2}}{4(1+2\nu)^{2}},
\]
where $\nu$ is as previously described. In addition to these modifications, IPL(A) replaces the ACG algorithm with an ACG variant that adapts the ACG stepsize for every ACG prox subproblem. In particular, it uses the line search subroutine outlined in Appendix~\ref{sec:acg}, and it applies a warm-start strategy\footnote{For the first prox subproblem, $\widetilde{M}$ is initialized to $\lam {\widetilde{\cal M}}_k / 2 + 1$. For $k\geq 1$, if $L_j$ is the last (estimated) curvature constant generated by the adaptive ACG for the $k^{\rm th}$ prox-subproblem, then $\widetilde M$ for the $(k+1)^{\rm th}$ subproblem is initialized to $\lam J_{k+1}/2 + 1$, where $J_{k+1}:=(L_j-1)/\lam$.} for choosing the parameter $\widetilde M$ given to ACG for each prox-subproblem. Regarding
$(\sigma,\nu)$ and the other hyperparameters, both variants choose 
\[
\beta_{1}=\max\left\{ 1,\frac{L_{f}}{\left[B_{g}^{(1)}\right]^{2}}\right\} ,\quad\lam=\frac{1}{2m_f},\quad\sigma=\sqrt{0.3},\quad\nu=\sqrt{\sigma\left(\lam L_{f}+1\right)},\quad p_{0}=0.
\]
While we do not show how the above changes affect the convergence
of IPL and IPL(A), we do note that their convergence can be analyzed using
the techniques of this paper and those in \cite{RenWilmelo2020iteration}.

We also describe the six benchmark algorithms of this section namely, two variants of the QP-AIPP method of \cite{WJRproxmet1} (nicknamed QP and QP(A)), the iALM of \cite{ImprovedShrinkingALM20}, two variants of the S-prox-ALM (nicknamed SPA1 and SPA2) of \cite{ErrorBoundJzhang-ZQLuo2020, ADMMJzhang-ZQLuo2020}, and the HiAPeM of \cite{HybridPenaltyAugLag19} (nicknamed HPM). 
QP is the method in \cite[Algorithm 4.1.1]{KongThesis2021} while QP(A) is a modification of QP that uses the same adaptive ACG variant and parameter warm-start strategy used by IPL(A). iALM was implemented by the authors to be exactly as stated in \cite[Algorithm 3]{ImprovedShrinkingALM20} with the parameters $\sigma$, $\beta_0$, $w_0$, $\textbf{y}^0$, and $\gamma_k$ chosen as
\[
\sigma=2,\quad\beta_{0}=\max\left\{ 1,\frac{L_{f}}{\|{\cal{A}}\|^{2}}\right\} ,\quad w_{0}=1,\quad\boldsymbol{y}^{0}=0,\quad\gamma_{k}=\frac{\left(\log2\right)\|c(x^{1})\|}{(k+1)\left[\log(k+2)\right]^{2}} \quad \forall k \ge 1,
\]
as suggested
in \cite[Theorem 2]{ImprovedShrinkingALM20}. Moreover, the starting point for  each
APG\footnote{APG is the name of the ACG subroutine used by iALM.} call is the prox center for the current prox subproblem. SPA1--SPA2 were also  implemented by the authors to be exactly as stated in \cite[Algorithm 2]{ErrorBoundJzhang-ZQLuo2020} with the parameters $\alpha_1$, $p$, $c$, $\beta$, $y_0$, and $z_0$ chosen as
\[
\alpha_1=\frac{\Gamma}{4},\quad p=2(L_{f}+\Gamma\|A\|^{2}),\quad c=\frac{1}{2(L_{f}+\Gamma\|A\|^{2})},\quad\beta=0.5,\quad y_{0}=0,\quad z_{0}=x_{0},
\]
where $\Gamma=1$ in SPA1 and $\Gamma=10$ in SPA2.
Finally, the code for HiAPeM was provided by the authors of \cite{HybridPenaltyAugLag19} with the parameters $\sigma$, $\beta_0$, $\gamma$, $\gamma_1$, $\gamma_2$, $N_0$, and $N_1$ chosen as
\[
\sigma = 3, \quad \beta_0 = 10^{-2}, \quad \gamma = 1.1, \quad \gamma_1 = 1.5, \quad \gamma_2 = 1, \quad N_0 = 100, \quad N_1 = 2.
\]

We next describe numerical and mathematical details that are common to all the experiments. First,
throughout this section, we denote $I$ to be the identity matrix,
$\mathbb{S}^{n}$ to be the set of symmetric $n$-by-$n$ matrices, and
$\mathbb{S}_{+}^{n}$ to be the set of positive semidefinite matrices
in $\mathbb{S}^{n}$. Second, given a tolerance pair $(\hat{\rho},\hat{\eta})\in\r_{++}^{2}$,
a pointed convex cone ${\cal K}$, and $z_{0}\in\dom h$, all the
methods attempt to find a pair $(\hat{z},\hat{p})$
satisfying
\begin{gather}
\begin{gathered}
\frac{{\rm dist }(0,\nabla f(\hat{z})+\pt h(\hat{z})+\nabla g(\hat{z})\hat{p})}{1+\|\nabla f(z_{0})\|} \leq \hat{\rho},\quad \frac{{\rm dist}(g(\hat{z}),N_{{\cal K}^*}(\hat{p}))}{1+{\rm dist}(g(z_{0}),-{\cal K})}\leq \hat{\eta}.
\end{gathered}
\label{eq:comp_crit}
\end{gather}
Third, as all the methods tested utilize an ACG variant to solve a sequence of convex proximal subproblems, the number of iterations reported in the experiments are the total number
of ACG iterations  needed to obtain a quadruple satisfying \eqref{eq:comp_crit}
(including those which fail
to satisfy parameter line searches within the adaptive ACG
variants used in IPL(A), QP(A), and HiAPeM).
Fourth, the bold numbers in each of the tables
of this section indicate the method that performed the most efficiently
for a given metric, e.g., runtime or iteration count. Finally, 
all algorithms described at the beginning of this section are implemented
in MATLAB 2021a and are run on Linux 64-bit machines, each containing
Xeon E5520 processors and at least 8 GB of memory.

We now end with some comments about the choice of algorithms in the experiments presented in
the subsections below.
First,
QP and QP(A) methods are not included in the experiments of Subsections~\ref{subsec:nonconvex_qc_qsdp} and \ref{subsec:nonconvex_qc_qp}  because
their current
implementations are only available for
linearly-constrained problems
(even though they can be extended to
nonlinearly-constrained problems).
Second, HiAPeM is only included
in the experiments of
Subsection~\ref{subsec:nonconvex_qc_qp} because the code
provided to the authors is specifically designed to solve the problem class
considered in that subsection.
Third, S-prox-ALM is only included in the experiments of Subsection~\ref{subsec:nonconvex_qp} because its convergence is only guaranteed when the composite function $h$ is the indicator function of a polyhedron. Finally, we do not include QP and IPL in Subsection~\ref{subsec:nonconvex_qp} because the results of
Subsections~\ref{subsec:nonconvex_qsdp}, \ref{subsec:nonconvex_qc_qsdp}, and \ref{subsec:nonconvex_qc_qp}
show that their adaptive variants are substantially more efficient.

\subsection{\label{subsec:nonconvex_qsdp}Nonconvex QSDP}

Given a pair of dimensions $(\ell,n)\in\mathbb{N}^{2}$, a scalar
pair $(\alpha_1,\alpha_2)\in\r_{++}^{2}$, linear operators ${\cal {\cal A}}:\mathbb{S}_{+}^{n}\mapsto\r^{\ell}$,
${\cal {\cal B}}:\mathbb{S}_{+}^{n}\mapsto\r^{n}$, and ${\cal {\cal C}}:\mathbb{S}_{+}^{n}\mapsto\r^{\ell}$
defined pointwise by 
\[
\left[{\cal A}(Z)\right]_{i}=\left\langle A_{i},Z\right\rangle ,\quad\left[{\cal B}(Z)\right]_{j}=\left\langle B_{j},Z\right\rangle ,\quad\left[{\cal C}(Z)\right]_{i}=\left\langle Q_{i},Z\right\rangle ,
\]
for matrices $\{A_{i}\}_{i=1}^{\ell},\{B_{j}\}_{j=1}^{n},\{Q_{i}\}_{i=1}^{\ell}\subseteq\r^{n\times n}$,
positive diagonal matrix $D\in\r^{n\times n}$, and a vector pair
$(b,d)\in\r^{\ell}\times\r^{\ell}$, we consider
the following nonconvex quadratic semidefinite programming (QSDP)
problem:
\begin{align*}
\min_{z\in \mathbb{S}_+^{n}}\  & -\frac{\alpha_1}{2}\|D{\cal B}(z)\|^{2}+\frac{\alpha_2}{2}\|{\cal C}(z)-d\|^{2}\\
\text{s.t.}\  & {\cal A}(z)=b, \quad  0\preceq z \preceq r I.
\end{align*}
 In particular, the problem instances tested are given in Table~\ref{tab:qsdp} for algorithms QP, QP(A), IPL, IPL(A), and iALM. For additional clarity, we describe below how the instances were generated.

\begin{table}[!tbh]
\begin{centering}
\begin{tabular}{>{\centering}p{0.5cm}>{\centering}p{0.5cm}>{\centering}p{0.4cm}>{\centering}p{0.4cm}|>{\centering}p{0.8cm}>{\centering}p{0.8cm}>{\centering}p{0.8cm}>{\centering}p{0.8cm}>{\centering}p{0.8cm}|>{\centering}p{0.8cm}>{\centering}p{0.8cm}>{\centering}p{0.8cm}>{\centering}p{0.8cm}>{\centering}p{0.8cm}}
\multicolumn{4}{c|}{\textbf{\scriptsize{}Parameters}} & \multicolumn{5}{c|}{\textbf{\scriptsize{}Iteration Count}} & \multicolumn{5}{c}{\textbf{\scriptsize{}Runtime}}\tabularnewline
\hline 
{\tiny{}$n$} & {\tiny{}$r$} & {\tiny{}$m$} & {\tiny{}$L_{f}$} & {\tiny{}iALM} & {\tiny{}QP} & {\tiny{}QP(A)} & {\tiny{}IPL} & {\tiny{}IPL(A)} & {\tiny{}iALM} & {\tiny{}QP} & {\tiny{}QP(A)} & {\tiny{}IPL} & {\tiny{}IPL(A)}\tabularnewline
\hline 
\hline 
{\tiny{}50} & {\tiny{}1.0} & {\tiny{}1} & {\tiny{}10} & {\tiny{}-} & {\tiny{}23296} & {\tiny{}1633} & {\tiny{}18618} & \textbf{\tiny{}1257} & {\tiny{}-} & {\tiny{}201.7} & {\tiny{}17.2} & {\tiny{}172.9} & \textbf{\tiny{}15.1}\tabularnewline
{\tiny{}50} & {\tiny{}1.0} & {\tiny{}1} & {\tiny{}20} & {\tiny{}-} & {\tiny{}15402} & {\tiny{}1210} & {\tiny{}10610} & \textbf{\tiny{}782} & {\tiny{}-} & {\tiny{}132.8} & {\tiny{}12.5} & {\tiny{}98.6} & \textbf{\tiny{}9.3}\tabularnewline
{\tiny{}50} & {\tiny{}1.0} & {\tiny{}1} & {\tiny{}40} & {\tiny{}-} & {\tiny{}12611} & {\tiny{}1076} & {\tiny{}7614} & \textbf{\tiny{}884} & {\tiny{}-} & {\tiny{}108.7} & {\tiny{}11.0} & {\tiny{}70.8} & \textbf{\tiny{}10.5}\tabularnewline
\hline 
{\tiny{}50} & {\tiny{}1.0} & {\tiny{}5} & {\tiny{}40} & {\tiny{}-} & {\tiny{}16499} & {\tiny{}1239} & {\tiny{}10578} & \textbf{\tiny{}753} & {\tiny{}-} & {\tiny{}144.8} & {\tiny{}13.5} & {\tiny{}100.6} & \textbf{\tiny{}9.8}\tabularnewline
{\tiny{}50} & {\tiny{}1.0} & {\tiny{}10} & {\tiny{}40} & {\tiny{}-} & {\tiny{}17868} & {\tiny{}1582} & {\tiny{}15238} & \textbf{\tiny{}1207} & {\tiny{}-} & {\tiny{}157.5} & {\tiny{}17.4} & {\tiny{}147.4} & \textbf{\tiny{}16.1}\tabularnewline
{\tiny{}50} & {\tiny{}1.0} & {\tiny{}20} & {\tiny{}40} & {\tiny{}-} & {\tiny{}74732} & {\tiny{}4425} & {\tiny{}53599} & \textbf{\tiny{}1633} & {\tiny{}-} & {\tiny{}665.1} & {\tiny{}51.6} & {\tiny{}506.2} & \textbf{\tiny{}22.6}\tabularnewline
\hline 
{\tiny{}50} & {\tiny{}5.0} & {\tiny{}1} & {\tiny{}20} & {\tiny{}-} & {\tiny{}40716} & {\tiny{}2648} & {\tiny{}35138} & \textbf{\tiny{}2335} & {\tiny{}-} & {\tiny{}353.3} & {\tiny{}28.3} & {\tiny{}326.9} & \textbf{\tiny{}28.2}\tabularnewline
{\tiny{}50} & {\tiny{}10.0} & {\tiny{}1} & {\tiny{}20} & {\tiny{}-} & {\tiny{}110657} & {\tiny{}6130} & {\tiny{}99621} & \textbf{\tiny{}5998} & {\tiny{}-} & {\tiny{}964.1} & \textbf{\tiny{}66.9} & {\tiny{}928.7} & {\tiny{}72.8}\tabularnewline
{\tiny{}50} & {\tiny{}20.0} & {\tiny{}1} & {\tiny{}20} & {\tiny{}-} & {\tiny{}129175} & {\tiny{}7112} & {\tiny{}116263} & \textbf{\tiny{}6936} & {\tiny{}-} & {\tiny{}1125.8} & \textbf{\tiny{}77.7} & {\tiny{}1088.4} & {\tiny{}86.5}\tabularnewline
\hline 
\hline 
{\tiny{}75} & {\tiny{}1.0} & {\tiny{}1} & {\tiny{}10} & {\tiny{}-} & {\tiny{}41201} & {\tiny{}1948} & {\tiny{}35565} & \textbf{\tiny{}1626} & {\tiny{}-} & {\tiny{}363.6} & {\tiny{}21.0} & {\tiny{}336.2} & \textbf{\tiny{}19.5}\tabularnewline
{\tiny{}75} & {\tiny{}1.0} & {\tiny{}1} & {\tiny{}20} & {\tiny{}-} & {\tiny{}32647} & {\tiny{}1576} & {\tiny{}27857} & \textbf{\tiny{}1289} & {\tiny{}-} & {\tiny{}289.1} & {\tiny{}16.8} & {\tiny{}264.0} & \textbf{\tiny{}15.4}\tabularnewline
{\tiny{}75} & {\tiny{}1.0} & {\tiny{}1} & {\tiny{}40} & {\tiny{}-} & {\tiny{}24932} & {\tiny{}1289} & {\tiny{}19939} & \textbf{\tiny{}984} & {\tiny{}-} & {\tiny{}220.7} & \textbf{\tiny{}13.7} & {\tiny{}202.4} & {\tiny{}18.3}\tabularnewline
\hline 
{\tiny{}75} & {\tiny{}1.0} & {\tiny{}5} & {\tiny{}40} & {\tiny{}-} & {\tiny{}31641} & {\tiny{}1462} & {\tiny{}23537} & \textbf{\tiny{}1025} & {\tiny{}-} & {\tiny{}375.5} & \textbf{\tiny{}17.5} & {\tiny{}317.1} & {\tiny{}17.9}\tabularnewline
{\tiny{}75} & {\tiny{}1.0} & {\tiny{}10} & {\tiny{}40} & {\tiny{}-} & {\tiny{}31874} & {\tiny{}1557} & {\tiny{}25519} & \textbf{\tiny{}1011} & {\tiny{}-} & {\tiny{}367.1} & {\tiny{}27.8} & {\tiny{}344.3} & \textbf{\tiny{}18.4}\tabularnewline
{\tiny{}75} & {\tiny{}1.0} & {\tiny{}20} & {\tiny{}40} & {\tiny{}-} & {\tiny{}38605} & {\tiny{}1945} & {\tiny{}23725} & \textbf{\tiny{}1077} & {\tiny{}-} & {\tiny{}481.9} & {\tiny{}27.3} & {\tiny{}312.9} & \textbf{\tiny{}21.8}\tabularnewline
\hline 
{\tiny{}75} & {\tiny{}5.0} & {\tiny{}1} & {\tiny{}20} & {\tiny{}-} & {\tiny{}92271} & {\tiny{}3830} & {\tiny{}87426} & \textbf{\tiny{}3648} & {\tiny{}-} & {\tiny{}1137.5} & {\tiny{}57.2} & {\tiny{}1088.7} & \textbf{\tiny{}42.0}\tabularnewline
{\tiny{}75} & {\tiny{}10.0} & {\tiny{}1} & {\tiny{}20} & {\tiny{}-} & {\tiny{}104348} & {\tiny{}4245} & {\tiny{}98207} & \textbf{\tiny{}4060} & {\tiny{}-} & {\tiny{}886.5} & \textbf{\tiny{}44.3} & {\tiny{}926.3} & {\tiny{}48.2}\tabularnewline
{\tiny{}75} & {\tiny{}20.0} & {\tiny{}1} & {\tiny{}20} & {\tiny{}-} & {\tiny{}152856} & {\tiny{}5961} & {\tiny{}143057} & \textbf{\tiny{}5807} & {\tiny{}-} & {\tiny{}1312.4} & \textbf{\tiny{}66.2} & {\tiny{}1380.6} & {\tiny{}71.3}\tabularnewline
\hline 
\hline 
{\tiny{}100} & {\tiny{}1.0} & {\tiny{}1} & {\tiny{}10} & {\tiny{}-} & {\tiny{}103570} & {\tiny{}3251} & {\tiny{}95110} & \textbf{\tiny{}2928} & {\tiny{}-} & {\tiny{}1641.3} & {\tiny{}62.2} & {\tiny{}1590.0} & \textbf{\tiny{}61.6}\tabularnewline
{\tiny{}100} & {\tiny{}1.0} & {\tiny{}1} & {\tiny{}20} & {\tiny{}-} & {\tiny{}74587} & {\tiny{}2466} & {\tiny{}66010} & \textbf{\tiny{}2262} & {\tiny{}-} & {\tiny{}1180.4} & \textbf{\tiny{}46.9} & {\tiny{}1102.5} & {\tiny{}47.2}\tabularnewline
{\tiny{}100} & {\tiny{}1.0} & {\tiny{}1} & {\tiny{}40} & {\tiny{}-} & {\tiny{}59253} & {\tiny{}2040} & {\tiny{}50282} & \textbf{\tiny{}1689} & {\tiny{}-} & {\tiny{}934.5} & {\tiny{}38.6} & {\tiny{}837.6} & \textbf{\tiny{}35.1}\tabularnewline
\hline 
{\tiny{}100} & {\tiny{}1.0} & {\tiny{}5} & {\tiny{}40} & {\tiny{}-} & {\tiny{}55305} & {\tiny{}1646} & {\tiny{}46890} & \textbf{\tiny{}1499} & {\tiny{}-} & {\tiny{}880.3} & \textbf{\tiny{}32.4} & {\tiny{}790.3} & {\tiny{}32.9}\tabularnewline
{\tiny{}100} & {\tiny{}1.0} & {\tiny{}10} & {\tiny{}40} & {\tiny{}-} & {\tiny{}82005} & {\tiny{}3133} & {\tiny{}61144} & \textbf{\tiny{}2698} & {\tiny{}-} & {\tiny{}1311.5} & {\tiny{}63.9} & {\tiny{}1034.8} & \textbf{\tiny{}62.2}\tabularnewline
{\tiny{}100} & {\tiny{}1.0} & {\tiny{}20} & {\tiny{}40} & {\tiny{}-} & {\tiny{}70045} & {\tiny{}2266} & {\tiny{}50591} & \textbf{\tiny{}1499} & {\tiny{}-} & {\tiny{}1127.7} & {\tiny{}46.7} & {\tiny{}866.5} & \textbf{\tiny{}36.3}\tabularnewline
\hline 
{\tiny{}100} & {\tiny{}5.0} & {\tiny{}1} & {\tiny{}20} & {\tiny{}-} & {\tiny{}129478} & {\tiny{}3998} & {\tiny{}119623} & \textbf{\tiny{}3649} & {\tiny{}-} & {\tiny{}2059.9} & {\tiny{}77.6} & {\tiny{}2008.2} & \textbf{\tiny{}76.8}\tabularnewline
{\tiny{}100} & {\tiny{}10.0} & {\tiny{}1} & {\tiny{}20} & {\tiny{}-} & {\tiny{}174666} & {\tiny{}5178} & {\tiny{}163769} & \textbf{\tiny{}4844} & {\tiny{}-} & {\tiny{}2774.6} & \textbf{\tiny{}99.5} & {\tiny{}2750.9} & {\tiny{}101.7}\tabularnewline
{\tiny{}100} & {\tiny{}20.0} & {\tiny{}1} & {\tiny{}20} & {\tiny{}-} & {\tiny{}238866} & {\tiny{}6887} & {\tiny{}225963} & \textbf{\tiny{}6563} & {\tiny{}-} & {\tiny{}3798.7} & \textbf{\tiny{}133.3} & {\tiny{}3789.0} & {\tiny{}139.3}\tabularnewline
\end{tabular}
\par\end{centering}
\caption{Iteration counts and runtimes (in seconds) for the Nonconvex QSDP
Problem in Subsection~\ref{subsec:nonconvex_qsdp}. Cells marked
with \textquotedblleft --\textquotedblright{} are those that did not
obtain a solution within the given time limit.\label{tab:qsdp}}
\end{table}

First, we chose $\ell=10$, varied $n$ across different problem instances, set
$\hat{\rho}=10^{-2}$ and $\hat{\eta}=10^{-4}$, and ensured that only 5\%
of the entries of $A_{i},B_{j},$ and $Q_{i}$ were set to be nonzero.
Second, the entries of $A_{i}$, $B_{j}$, $Q_{i}$, and $d$ (resp.
$D$) were generated by sampling from the uniform distribution ${\cal U}[0,1]$
(resp. ${\cal U}\{1,...,1000\}$). Third, the vector $b$ was set
to $b={\cal A}({\rm diag}(u))$ where $u$ is a random vector in ${\cal U}[0,r]^{n\times n}$. Fourth, the initial starting point $z_{0}$ was set to be the zero matrix.
Finally, each problem instance considered
was based on a specific triple $(r,m_f,L_{f})$, for which the scalar pair
$(\alpha_1,\alpha_2)$ is selected so that $L_{f}=\lam_{\max}(\nabla^{2}f)$
and $-m_f=\lam_{\min}(\nabla^{2}f)$, and we set a time limit of 6000 seconds.

\subsection{\label{subsec:nonconvex_qc_qsdp}Nonconvex QC-QSDP}

Given a dimension pair $(\ell,n)\in\mathbb{N}^{2}$, scalar $r>0$, matrices $P,Q,R\in\r^{n\times n}$,
and the quantities $(\alpha_1,\alpha_2)$, ${\cal {\cal B}}$, ${\cal {\cal C}}$, $D$, and $d$ as in Subsection~\ref{subsec:nonconvex_qsdp},
we consider the nonconvex quadratically constrained 
QSDP (QC-QSDP) problem:
\begin{align*}
\begin{alignedat}{2}\min_{Z}\  & -\frac{\alpha_1}{2}\|D{\cal B}(Z)\|^{2}+\frac{\alpha_2}{2}\|{\cal C}(Z)-d\|^{2}\\
\text{s.t.}\  & \frac{1}{2}(PZ)^{*}PZ+\frac{1}{2}Q^{*}QZ+\frac{1}{2}ZQ^{*}Q\preceq R^{*}R,\\
 & \quad 0\preceq Z \preceq r I.
\end{alignedat}
\end{align*}
In particular, the problem instances tested are given in Table~\ref{tab:qcqsdp} for algorithms iALM, IPL, and IPL(A). For additional clarity, we describe below how the instances were generated. 

\begin{table}[!tbh]
\begin{centering}
\begin{tabular}{>{\centering}p{0.5cm}>{\centering}p{0.5cm}>{\centering}p{0.5cm}>{\centering}p{0.5cm}>{\centering}p{0.5cm}|>{\centering}p{0.9cm}>{\centering}p{0.9cm}>{\centering}p{0.9cm}|>{\centering}p{0.9cm}>{\centering}p{0.9cm}>{\centering}p{0.9cm}}
\multicolumn{5}{c|}{\textbf{\scriptsize{}Parameters}} & \multicolumn{3}{c|}{\textbf{\scriptsize{}Iteration Count}} & \multicolumn{3}{c}{\textbf{\scriptsize{}Runtime}}\tabularnewline
\hline 
{\tiny{}$n$} & {\tiny{}$r$} & {\tiny{}$m$} & {\tiny{}$L_{f}$} & {\tiny{}$L_{g}$} & {\tiny{}iALM} & {\tiny{}IPL} & {\tiny{}IPL(A)} & {\tiny{}iALM} & {\tiny{}IPL} & {\tiny{}IPL(A)}\tabularnewline
\hline 
\hline 
{\tiny{}50} & {\tiny{}1.0} & {\tiny{}$10^{0}$} & {\tiny{}$10^{3}$} & {\tiny{}6.2} & {\tiny{}-} & {\tiny{}11058} & \textbf{\tiny{}6760} & {\tiny{}-} & {\tiny{}108.5} & \textbf{\tiny{}80.1}\tabularnewline
{\tiny{}50} & {\tiny{}1.0} & {\tiny{}$10^{0}$} & {\tiny{}$10^{4}$} & {\tiny{}10.9} & {\tiny{}-} & {\tiny{}244} & \textbf{\tiny{}213} & {\tiny{}-} & {\tiny{}2.4} & \textbf{\tiny{}2.4}\tabularnewline
{\tiny{}50} & {\tiny{}1.0} & {\tiny{}$10^{0}$} & {\tiny{}$10^{5}$} & {\tiny{}17.1} & {\tiny{}1862} & {\tiny{}778} & \textbf{\tiny{}580} & {\tiny{}18.2} & {\tiny{}7.5} & \textbf{\tiny{}6.7}\tabularnewline
\hline 
{\tiny{}50} & {\tiny{}1.0} & {\tiny{}$10^{1}$} & {\tiny{}$10^{5}$} & {\tiny{}10.9} & {\tiny{}-} & {\tiny{}244} & \textbf{\tiny{}213} & {\tiny{}-} & {\tiny{}2.3} & \textbf{\tiny{}2.4}\tabularnewline
{\tiny{}50} & {\tiny{}1.0} & {\tiny{}$10^{2}$} & {\tiny{}$10^{5}$} & {\tiny{}6.2} & {\tiny{}-} & {\tiny{}11058} & \textbf{\tiny{}6760} & {\tiny{}-} & {\tiny{}107.5} & \textbf{\tiny{}79.7}\tabularnewline
{\tiny{}50} & {\tiny{}1.0} & {\tiny{}$10^{3}$} & {\tiny{}$10^{5}$} & {\tiny{}2.7} & {\tiny{}-} & {\tiny{}13062} & \textbf{\tiny{}7381} & {\tiny{}-} & {\tiny{}134.4} & \textbf{\tiny{}89.5}\tabularnewline
\hline 
{\tiny{}50} & {\tiny{}5.0} & {\tiny{}$10^{0}$} & {\tiny{}$10^{5}$} & {\tiny{}3.4} & {\tiny{}724} & {\tiny{}778} & \textbf{\tiny{}580} & {\tiny{}7.2} & {\tiny{}7.5} & \textbf{\tiny{}6.7}\tabularnewline
{\tiny{}50} & {\tiny{}10.0} & {\tiny{}$10^{0}$} & {\tiny{}$10^{5}$} & {\tiny{}1.7} & {\tiny{}726} & {\tiny{}778} & \textbf{\tiny{}580} & {\tiny{}7.1} & {\tiny{}7.4} & \textbf{\tiny{}6.7}\tabularnewline
{\tiny{}50} & {\tiny{}20.0} & {\tiny{}$10^{0}$} & {\tiny{}$10^{5}$} & {\tiny{}0.9} & {\tiny{}720} & {\tiny{}778} & \textbf{\tiny{}580} & {\tiny{}7.1} & {\tiny{}7.5} & \textbf{\tiny{}6.7}\tabularnewline
\hline 
\hline 
{\tiny{}75} & {\tiny{}1.0} & {\tiny{}$10^{0}$} & {\tiny{}$10^{3}$} & {\tiny{}8.9} & {\tiny{}-} & {\tiny{}22766} & \textbf{\tiny{}12386} & {\tiny{}-} & {\tiny{}418.4} & \textbf{\tiny{}280.3}\tabularnewline
{\tiny{}75} & {\tiny{}1.0} & {\tiny{}$10^{0}$} & {\tiny{}$10^{4}$} & {\tiny{}15.8} & {\tiny{}-} & {\tiny{}244} & \textbf{\tiny{}212} & {\tiny{}-} & {\tiny{}4.4} & \textbf{\tiny{}4.5}\tabularnewline
{\tiny{}75} & {\tiny{}1.0} & {\tiny{}$10^{0}$} & {\tiny{}$10^{5}$} & {\tiny{}24.7} & {\tiny{}3409} & {\tiny{}777} & \textbf{\tiny{}579} & {\tiny{}61.5} & {\tiny{}14.1} & \textbf{\tiny{}12.8}\tabularnewline
\hline 
{\tiny{}75} & {\tiny{}1.0} & {\tiny{}$10^{1}$} & {\tiny{}$10^{5}$} & {\tiny{}15.8} & {\tiny{}-} & {\tiny{}244} & \textbf{\tiny{}212} & {\tiny{}-} & {\tiny{}4.4} & \textbf{\tiny{}4.6}\tabularnewline
{\tiny{}75} & {\tiny{}1.0} & {\tiny{}$10^{2}$} & {\tiny{}$10^{5}$} & {\tiny{}8.9} & {\tiny{}-} & {\tiny{}20257} & \textbf{\tiny{}12317} & {\tiny{}-} & {\tiny{}377.3} & \textbf{\tiny{}281.3}\tabularnewline
{\tiny{}75} & {\tiny{}1.0} & {\tiny{}$10^{3}$} & {\tiny{}$10^{5}$} & {\tiny{}4.0} & {\tiny{}-} & {\tiny{}135657} & \textbf{\tiny{}19950} & {\tiny{}-} & {\tiny{}2515.9} & \textbf{\tiny{}571.6}\tabularnewline
\hline 
{\tiny{}75} & {\tiny{}5.0} & {\tiny{}$10^{0}$} & {\tiny{}$10^{5}$} & {\tiny{}4.9} & {\tiny{}5879} & {\tiny{}777} & \textbf{\tiny{}579} & {\tiny{}140.4} & {\tiny{}14.2} & \textbf{\tiny{}13.0}\tabularnewline
{\tiny{}75} & {\tiny{}10.0} & {\tiny{}$10^{0}$} & {\tiny{}$10^{5}$} & {\tiny{}2.5} & {\tiny{}1115} & {\tiny{}777} & \textbf{\tiny{}579} & {\tiny{}20.2} & {\tiny{}14.2} & \textbf{\tiny{}13.0}\tabularnewline
{\tiny{}75} & {\tiny{}20.0} & {\tiny{}$10^{0}$} & {\tiny{}$10^{5}$} & {\tiny{}1.2} & {\tiny{}10832} & {\tiny{}777} & \textbf{\tiny{}579} & {\tiny{}194.9} & {\tiny{}14.2} & \textbf{\tiny{}13.0}\tabularnewline
\hline 
\hline 
{\tiny{}100} & {\tiny{}1.0} & {\tiny{}$10^{0}$} & {\tiny{}$10^{3}$} & {\tiny{}11.9} & {\tiny{}-} & {\tiny{}40755} & \textbf{\tiny{}16292} & {\tiny{}-} & {\tiny{}1230.0} & \textbf{\tiny{}612.6}\tabularnewline
{\tiny{}100} & {\tiny{}1.0} & {\tiny{}$10^{0}$} & {\tiny{}$10^{4}$} & {\tiny{}21.2} & {\tiny{}-} & {\tiny{}252} & \textbf{\tiny{}213} & {\tiny{}-} & {\tiny{}7.5} & \textbf{\tiny{}7.7}\tabularnewline
{\tiny{}100} & {\tiny{}1.0} & {\tiny{}$10^{0}$} & {\tiny{}$10^{5}$} & {\tiny{}33.2} & {\tiny{}4710} & {\tiny{}778} & \textbf{\tiny{}580} & {\tiny{}128.2} & {\tiny{}23.1} & \textbf{\tiny{}21.5}\tabularnewline
\hline 
{\tiny{}100} & {\tiny{}1.0} & {\tiny{}$10^{1}$} & {\tiny{}$10^{5}$} & {\tiny{}21.2} & {\tiny{}-} & {\tiny{}244} & \textbf{\tiny{}213} & {\tiny{}-} & {\tiny{}7.3} & \textbf{\tiny{}7.7}\tabularnewline
{\tiny{}100} & {\tiny{}1.0} & {\tiny{}$10^{2}$} & {\tiny{}$10^{5}$} & {\tiny{}11.9} & {\tiny{}-} & {\tiny{}158085} & \textbf{\tiny{}22101} & {\tiny{}-} & {\tiny{}4714.2} & \textbf{\tiny{}831.4}\tabularnewline
{\tiny{}100} & {\tiny{}1.0} & {\tiny{}$10^{3}$} & {\tiny{}$10^{5}$} & {\tiny{}5.3} & {\tiny{}-} & {\tiny{}-} & \textbf{\tiny{}61179} & {\tiny{}-} & {\tiny{}-} & \textbf{\tiny{}2306.2}\tabularnewline
\hline 
{\tiny{}100} & {\tiny{}5.0} & {\tiny{}$10^{0}$} & {\tiny{}$10^{5}$} & {\tiny{}6.6} & {\tiny{}3575} & {\tiny{}778} & \textbf{\tiny{}580} & {\tiny{}97.7} & {\tiny{}23.1} & \textbf{\tiny{}21.5}\tabularnewline
{\tiny{}100} & {\tiny{}10.0} & {\tiny{}$10^{0}$} & {\tiny{}$10^{5}$} & {\tiny{}3.3} & {\tiny{}2406} & {\tiny{}778} & \textbf{\tiny{}580} & {\tiny{}65.8} & {\tiny{}23.3} & \textbf{\tiny{}21.5}\tabularnewline
{\tiny{}100} & {\tiny{}20.0} & {\tiny{}$10^{0}$} & {\tiny{}$10^{5}$} & {\tiny{}1.7} & {\tiny{}1706} & {\tiny{}778} & \textbf{\tiny{}580} & {\tiny{}46.5} & {\tiny{}23.1} & \textbf{\tiny{}21.4}\tabularnewline
\end{tabular}
\par\end{centering}
\caption{Iteration counts and runtimes (in seconds) for Nonconvex QC-QSDP Problems
in Subsection~\ref{subsec:nonconvex_qc_qsdp}. Cells marked with
\textquotedblleft --\textquotedblright{} are those that did not obtain
a solution within the given time limit.\label{tab:qcqsdp}}
\end{table}

First, we chose $\ell=10$, varied $n$ across different problem instances, and chose $\hat{\rho}=\hat{\eta}=10^{-3}$. Second, the
quantities ${\cal B}$, ${\cal C}$, $D$, and $d$ were generated
in the same way as in Subsection~\ref{subsec:nonconvex_qsdp},
 the matrix $R$ was set to $I$, and the entries
of matrices $P$ and $Q$ were sampled from the uniform distributions
$\log(L_f/m_f)\cdot {\cal U}[0,1/\sqrt{100nr}]$ and ${\cal U}[0,1/n]$, respectively. Third,
the initial starting  point $z_{0}$ was set to
be the zero matrix. 
Finally, like in Subsection~\ref{subsec:nonconvex_qsdp},
each problem instance considered was based on a specific triple $(r,m_f,L_{f})$,
for which the scalar pair $(\alpha_1,\alpha_2)$ is selected so that $L_{f}=\lam_{\max}(\nabla^{2}f)$
and $-m_f=\lam_{\min}(\nabla^{2}f)$, and a time limit of 6000 seconds.

\subsection{\label{subsec:nonconvex_qc_qp}Nonconvex QC-QP}

Given a dimension pair $(\ell,n)\in\mathbb{N}^{2}$, matrices $\{Q_j\}_{j=0}^\ell$, vectors $\{c_j\}_{j=0}^\ell$, scalars $\{d_j\}_{j=0}^\ell$, and scalar $r>0$,
we consider the nonconvex quadratically constrained 
quadratic programming (QC-QP) problem:
\begin{align*}
\begin{alignedat}{2}\min_{z}\  & \frac{1}{2} z^T Q_0 z + c_0^T z + d_0\\
\text{s.t.}\  & \frac{1}{2} z^T Q_j z + c_j^T z + d_j \leq 0,& \quad j\in\{1,...,\ell\}, \\
 & -r\leq z_i \leq r, & \quad i\in\{1,...,n\},
\end{alignedat}
\end{align*}
where $Q_j \succeq 0$ for $j=1,...,\ell$, $Q_0$ is indefinite, and the constraint set has nonempty interior. In particular, the problem tested are given in Table~\ref{tab:qcqp} for algorithms iALM, IPL, IPL(A), and HPM. For additional clarity, we describe below how the instances were generated and the organization of the tables.

\begin{table}[!tbh]
\begin{centering}
\begin{tabular}{>{\centering}p{0.5cm}>{\centering}p{0.5cm}>{\centering}p{0.4cm}>{\centering}p{0.4cm}>{\centering}p{0.5cm}|>{\centering}p{1cm}>{\centering}p{1cm}>{\centering}p{1cm}>{\centering}p{1cm}}
\multicolumn{5}{c|}{\textbf{\scriptsize{}Parameters}} & \multicolumn{4}{c}{\textbf{\scriptsize{}Iteration Count}}\tabularnewline
\hline 
{\tiny{}$n$} & {\tiny{}$r$} & {\tiny{}$m$} & {\tiny{}$L_{f}$} & {\tiny{}$L_{g}$} & {\tiny{}iALM} & {\tiny{}IPL} & {\tiny{}IPL(A)} & {\tiny{}HPM}\tabularnewline
\hline 
\hline 
{\tiny{}250} & {\tiny{}1.0} & {\tiny{}$10^{0}$} & {\tiny{}$10^{3}$} & {\tiny{}7.3} & {\tiny{}-} & {\tiny{}2690} & \textbf{\tiny{}273} & {\tiny{}2679}\tabularnewline
{\tiny{}250} & {\tiny{}1.0} & {\tiny{}$10^{0}$} & {\tiny{}$10^{4}$} & {\tiny{}9.7} & {\tiny{}-} & {\tiny{}2973} & \textbf{\tiny{}644} & {\tiny{}27934}\tabularnewline
{\tiny{}250} & {\tiny{}1.0} & {\tiny{}$10^{0}$} & {\tiny{}$10^{5}$} & {\tiny{}12.1} & {\tiny{}-} & {\tiny{}3521} & \textbf{\tiny{}1788} & {\tiny{}59381}\tabularnewline
\hline 
{\tiny{}250} & {\tiny{}1.0} & {\tiny{}$10^{1}$} & {\tiny{}$10^{5}$} & {\tiny{}9.7} & {\tiny{}-} & {\tiny{}2690} & \textbf{\tiny{}1717} & {\tiny{}60335}\tabularnewline
{\tiny{}250} & {\tiny{}1.0} & {\tiny{}$10^{2}$} & {\tiny{}$10^{5}$} & {\tiny{}7.3} & {\tiny{}-} & {\tiny{}947} & \textbf{\tiny{}676} & {\tiny{}8206}\tabularnewline
{\tiny{}250} & {\tiny{}1.0} & {\tiny{}$10^{3}$} & {\tiny{}$10^{5}$} & {\tiny{}4.8} & {\tiny{}-} & {\tiny{}487} & \textbf{\tiny{}390} & {\tiny{}8262}\tabularnewline
\hline 
{\tiny{}250} & {\tiny{}5.0} & {\tiny{}$10^{0}$} & {\tiny{}$10^{5}$} & {\tiny{}12.1} & {\tiny{}-} & {\tiny{}13766} & \textbf{\tiny{}863} & {\tiny{}14963}\tabularnewline
{\tiny{}250} & {\tiny{}10.0} & {\tiny{}$10^{0}$} & {\tiny{}$10^{5}$} & {\tiny{}12.1} & {\tiny{}-} & {\tiny{}27590} & \textbf{\tiny{}1632} & {\tiny{}11390}\tabularnewline
{\tiny{}250} & {\tiny{}20.0} & {\tiny{}$10^{0}$} & {\tiny{}$10^{5}$} & {\tiny{}12.1} & {\tiny{}-} & {\tiny{}28430} & \textbf{\tiny{}2694} & {\tiny{}10545}\tabularnewline
\hline 
\hline 
{\tiny{}500} & {\tiny{}1.0} & {\tiny{}$10^{0}$} & {\tiny{}$10^{3}$} & {\tiny{}7.3} & {\tiny{}-} & {\tiny{}3834} & \textbf{\tiny{}332} & {\tiny{}2383}\tabularnewline
{\tiny{}500} & {\tiny{}1.0} & {\tiny{}$10^{0}$} & {\tiny{}$10^{4}$} & {\tiny{}9.7} & {\tiny{}-} & {\tiny{}3287} & \textbf{\tiny{}659} & {\tiny{}26618}\tabularnewline
{\tiny{}500} & {\tiny{}1.0} & {\tiny{}$10^{0}$} & {\tiny{}$10^{5}$} & {\tiny{}12.1} & {\tiny{}-} & {\tiny{}4316} & \textbf{\tiny{}2554} & {\tiny{}49287}\tabularnewline
\hline 
{\tiny{}500} & {\tiny{}1.0} & {\tiny{}$10^{1}$} & {\tiny{}$10^{5}$} & {\tiny{}9.7} & {\tiny{}-} & {\tiny{}3605} & \textbf{\tiny{}1912} & {\tiny{}61336}\tabularnewline
{\tiny{}500} & {\tiny{}1.0} & {\tiny{}$10^{2}$} & {\tiny{}$10^{5}$} & {\tiny{}7.3} & {\tiny{}-} & {\tiny{}1498} & \textbf{\tiny{}908} & {\tiny{}9221}\tabularnewline
{\tiny{}500} & {\tiny{}1.0} & {\tiny{}$10^{3}$} & {\tiny{}$10^{5}$} & {\tiny{}4.8} & {\tiny{}-} & {\tiny{}1000} & \textbf{\tiny{}750} & {\tiny{}8659}\tabularnewline
\hline 
{\tiny{}500} & {\tiny{}5.0} & {\tiny{}$10^{0}$} & {\tiny{}$10^{5}$} & {\tiny{}12.1} & {\tiny{}-} & {\tiny{}14452} & \textbf{\tiny{}1075} & {\tiny{}13387}\tabularnewline
{\tiny{}500} & {\tiny{}10.0} & {\tiny{}$10^{0}$} & {\tiny{}$10^{5}$} & {\tiny{}12.1} & {\tiny{}-} & {\tiny{}29301} & \textbf{\tiny{}1877} & {\tiny{}10549}\tabularnewline
{\tiny{}500} & {\tiny{}20.0} & {\tiny{}$10^{0}$} & {\tiny{}$10^{5}$} & {\tiny{}12.1} & {\tiny{}-} & {\tiny{}91119} & \textbf{\tiny{}4720} & {\tiny{}7311}\tabularnewline
\hline 
\hline 
{\tiny{}1000} & {\tiny{}1.0} & {\tiny{}$10^{0}$} & {\tiny{}$10^{3}$} & {\tiny{}7.3} & {\tiny{}-} & {\tiny{}8862} & \textbf{\tiny{}679} & {\tiny{}16812}\tabularnewline
{\tiny{}1000} & {\tiny{}1.0} & {\tiny{}$10^{0}$} & {\tiny{}$10^{4}$} & {\tiny{}9.7} & {\tiny{}-} & {\tiny{}4678} & \textbf{\tiny{}726} & {\tiny{}22044}\tabularnewline
{\tiny{}1000} & {\tiny{}1.0} & {\tiny{}$10^{0}$} & {\tiny{}$10^{5}$} & {\tiny{}12.1} & {\tiny{}-} & {\tiny{}5969} & \textbf{\tiny{}1825} & {\tiny{}42739}\tabularnewline
\hline 
{\tiny{}1000} & {\tiny{}1.0} & {\tiny{}$10^{1}$} & {\tiny{}$10^{5}$} & {\tiny{}9.7} & {\tiny{}-} & {\tiny{}5108} & \textbf{\tiny{}2026} & {\tiny{}58180}\tabularnewline
{\tiny{}1000} & {\tiny{}1.0} & {\tiny{}$10^{2}$} & {\tiny{}$10^{5}$} & {\tiny{}7.3} & {\tiny{}-} & {\tiny{}1018} & \textbf{\tiny{}594} & {\tiny{}142579}\tabularnewline
{\tiny{}1000} & {\tiny{}1.0} & {\tiny{}$10^{3}$} & {\tiny{}$10^{5}$} & {\tiny{}4.8} & {\tiny{}-} & {\tiny{}1187} & \textbf{\tiny{}847} & {\tiny{}36673}\tabularnewline
\hline 
{\tiny{}1000} & {\tiny{}5.0} & {\tiny{}$10^{0}$} & {\tiny{}$10^{5}$} & {\tiny{}12.1} & {\tiny{}-} & {\tiny{}13553} & \textbf{\tiny{}1491} & {\tiny{}17706}\tabularnewline
{\tiny{}1000} & {\tiny{}10.0} & {\tiny{}$10^{0}$} & {\tiny{}$10^{5}$} & {\tiny{}12.1} & {\tiny{}-} & {\tiny{}26983} & \textbf{\tiny{}2621} & {\tiny{}11514}\tabularnewline
{\tiny{}1000} & {\tiny{}20.0} & {\tiny{}$10^{0}$} & {\tiny{}$10^{5}$} & {\tiny{}12.1} & {\tiny{}-} & {\tiny{}53820} & \textbf{\tiny{}5658} & {\tiny{}13451}\tabularnewline
\end{tabular}
\par\end{centering}
\caption{Iteration counts for the Nonconvex QC-QP Problem in Subsection~\ref{subsec:nonconvex_qc_qp}.
Cells marked with \textquotedblleft --\textquotedblright{} are those
that did not obtain a solution within the given time.\label{tab:qcqp}}
\end{table}

First, we chose $\ell=10$, varied $n$ across different problem instances, and set $\hat{\rho}=\hat{\eta}=10^{-5}$. Second, the entries of $d_0$ and $c_j$ for $j=0,..,\ell$ were generated from the ${\cal U}[0,1]$ distribution. 
On the other hand, the entries of $d_j$ were generated from the $-20-10\cdot{\cal U}[0,10]$ distribution, the eigenvectors of $Q_j$ were taken from the QR decomposition of a random matrix from the ${\cal U}[0,1]^{n\times n}$ distribution, 
the eigenvalues of $Q_0$ are taken from the ${\cal U}[-m_f,L_f]$ distribution for a given $(m_f,L_f)\in \mathbb \r^2$, and the eigenvalues of $Q_j$ for $j=1,..,n$ are taken from the $\log(L_f/m_f) \cdot {\cal U}[0,1/3]$ distribution.
 Third, 
the initial starting  point $z_{0}$ was taken from the ${\cal U}[-r,r]^{n\times n}$ distribution. 
Finally,
each problem instance considered was based on a specific triple $(r,m_f,L_{f})$, that specifies the eigenvalues for $Q_0$ and the domain of $h$, a time limit of 3000 seconds, and an iteration limit of 1000000.

Also, for the sake of fairness, we compare HPM against iALM, IPL, and IPL(A) in terms of ACG iteration counts only. 
This is because: (i) all the tested methods perform ACG iterations that essentially require the same amount of effort; and (ii) there is substantially more computational overhead found in the more general implementations of iALM, IPL, and IPL(A) compared to the more specialized implementation of HPM
\footnote{More specifically, the implementation of HPM given by authors of \cite{HybridPenaltyAugLag19} takes the problem data $\{Q_j\}_{j=0}^{\ell}$, $\{c_j\}_{j=0}^{\ell}$, $\{d_j\}_{j=0}^{\ell}$, and $r$ as input and directly applies the HiAPeM algorithm instance for QC-QP problems. In contrast, the implementations of iALM, IPL, and IPL(A) take function oracles for $f$, $\nabla f$, $h$, $g$, $\nabla g$, and \[
{\rm prox}_{\lam h}(\cdot) = \argmin_{u\in \dom h}\{\lam h(u)+ \frac{1}{2}\|u-z\|^2\}, \quad \Pi_{\cal K}(\cdot), \quad \Pi_{\cal K^*}(\cdot),
\]
as input and manipulate these oracles to run their algorithm instances. As executing floating-point operations is substantially less costly than manipulating (symbolic) function oracles, 
the HPM implementation is drastically more efficient on an iteration-to-iteration basis (roughly 8-10x more) compared to the iALM, IPL, and IPL(A) implementations, at the cost of a less general-purpose API.}.

\subsection{\label{subsec:nonconvex_qp}Nonconvex QP}

Given a pair of dimensions $(\ell,n)\in\mathbb{N}^{2}$, a scalar
pair $(\omega_1, \omega_2)\in\R_{++}^{2}$, matrices $Q,C\in\R^{\ell\times n}$
and $B\in\r^{n\times n}$, positive diagonal matrix $D\in\R^{n\times n}$,
and a vector pair $(b,d)\in\R^{\ell}\times\R^{\ell}$, we consider
the problem 
\begin{align*}
\begin{aligned}
\min_{z}\  & f(z)-\frac{\omega_{1}}{2}\|DBz\|^{2}+\frac{\omega_{2}}{2}\|{\cal C}z-d\|^{2}\\
\text{s.t.}\  & Qz=b, \\
& -r\leq z_i \leq r, & \quad i\in\{1,...,n\}.
\end{aligned}
\end{align*}
In particular, the problem instances tested are given in Table~\ref{tab:qp} for algorithms IPL(A), QP(A), SPA1, and SPA2. For additional clarity, we describe below some differences between NL-IAPIAL and S-prox-ALM, as well as how the instances were generated. 

\begin{table}[!tbh]
\begin{centering}
\begin{tabular}{>{\centering}p{0.5cm}>{\centering}p{0.5cm}>{\centering}p{0.4cm}>{\centering}p{0.4cm}|>{\centering}p{0.85cm}>{\centering}p{0.85cm}>{\centering}p{0.85cm}>{\centering}p{0.85cm}>{\centering}p{0.85cm}|>{\centering}p{1cm}>{\centering}p{1cm}>{\centering}p{1cm}>{\centering}p{1cm}>{\centering}p{1cm}}
\multicolumn{4}{c|}{\textbf{\scriptsize{}Parameters}} & \multicolumn{5}{c|}{\textbf{\scriptsize{}Iteration Count}} & \multicolumn{5}{c}{\textbf{\scriptsize{}Residual $\hat{r}$/Runtime}}\tabularnewline
\hline 
{\tiny{}$n$} & {\tiny{}$r$} & {\tiny{}$m$} & {\tiny{}$L_{f}$} & {\tiny{}iALM} & {\tiny{}QP(A)} & {\tiny{}IPL(A)} & {\tiny{}SPA1} & {\tiny{}SPA2} & {\tiny{}iALM} & {\tiny{}QP(A)} & {\tiny{}IPL(A)} & {\tiny{}SPA1} & {\tiny{}SPA2}\tabularnewline
\hline 
\hline 
{\tiny{}250} & {\tiny{}1.0} & {\tiny{}$10^{0}$} & {\tiny{}$10^{3}$} & {\tiny{}111250} & {\tiny{}53625} & \textbf{\tiny{}23000} & {\tiny{}-} & {\tiny{}-} & {\tiny{}-/894} & {\tiny{}-/403} & \textbf{\tiny{}-/177} & {\tiny{}3E-04/-} & {\tiny{}2E-03/-}\tabularnewline
{\tiny{}250} & {\tiny{}1.0} & {\tiny{}$10^{0}$} & {\tiny{}$10^{4}$} & {\tiny{}103710} & {\tiny{}60997} & \textbf{\tiny{}50195} & {\tiny{}-} & {\tiny{}-} & {\tiny{}-/1009} & {\tiny{}-/541} & \textbf{\tiny{}-/452} & {\tiny{}3E-04/-} & {\tiny{}3E-04/-}\tabularnewline
{\tiny{}250} & {\tiny{}1.0} & {\tiny{}$10^{0}$} & {\tiny{}$10^{5}$} & {\tiny{}58049} & {\tiny{}38963} & \textbf{\tiny{}30024} & {\tiny{}-} & {\tiny{}-} & {\tiny{}-/406} & {\tiny{}-/255} & \textbf{\tiny{}-/199} & {\tiny{}2E-05/-} & {\tiny{}2E-05/-}\tabularnewline
\hline 
{\tiny{}250} & {\tiny{}1.0} & {\tiny{}$10^{1}$} & {\tiny{}$10^{5}$} & {\tiny{}103800} & {\tiny{}60851} & \textbf{\tiny{}50195} & {\tiny{}-} & {\tiny{}-} & {\tiny{}-/550} & {\tiny{}-/344} & \textbf{\tiny{}-/284} & {\tiny{}2E-04/-} & {\tiny{}2E-04/-}\tabularnewline
{\tiny{}250} & {\tiny{}1.0} & {\tiny{}$10^{2}$} & {\tiny{}$10^{5}$} & {\tiny{}130970} & {\tiny{}49208} & \textbf{\tiny{}20775} & {\tiny{}-} & {\tiny{}-} & {\tiny{}-/695} & {\tiny{}-/277} & \textbf{\tiny{}-/119} & {\tiny{}4E-04/-} & {\tiny{}3E-04/-}\tabularnewline
{\tiny{}250} & {\tiny{}1.0} & {\tiny{}$10^{3}$} & {\tiny{}$10^{5}$} & {\tiny{}427430} & {\tiny{}279680} & \textbf{\tiny{}16146} & {\tiny{}269460} & {\tiny{}256820} & {\tiny{}-/2257} & {\tiny{}-/1609} & \textbf{\tiny{}-/96} & {\tiny{}-/1860} & {\tiny{}-/1771}\tabularnewline
\hline 
{\tiny{}250} & {\tiny{}5.0} & {\tiny{}$10^{0}$} & {\tiny{}$10^{5}$} & {\tiny{}52603} & {\tiny{}40483} & \textbf{\tiny{}33431} & {\tiny{}-} & {\tiny{}-} & {\tiny{}-/277} & {\tiny{}-/228} & \textbf{\tiny{}-/187} & {\tiny{}2E-05/-} & {\tiny{}2E-05/-}\tabularnewline
{\tiny{}250} & {\tiny{}10.0} & {\tiny{}$10^{0}$} & {\tiny{}$10^{5}$} & {\tiny{}67225} & {\tiny{}41561} & \textbf{\tiny{}33706} & {\tiny{}-} & {\tiny{}-} & {\tiny{}-/355} & {\tiny{}-/233} & \textbf{\tiny{}-/190} & {\tiny{}2E-05/-} & {\tiny{}2E-05/-}\tabularnewline
{\tiny{}250} & {\tiny{}20.0} & {\tiny{}$10^{0}$} & {\tiny{}$10^{5}$} & {\tiny{}57393} & {\tiny{}41786} & \textbf{\tiny{}34756} & {\tiny{}-} & {\tiny{}-} & {\tiny{}-/302} & {\tiny{}-/234} & \textbf{\tiny{}-/195} & {\tiny{}2E-05/-} & {\tiny{}2E-05/-}\tabularnewline
\hline 
\hline 
{\tiny{}500} & {\tiny{}1.0} & {\tiny{}$10^{0}$} & {\tiny{}$10^{3}$} & {\tiny{}-} & {\tiny{}-} & \textbf{\tiny{}35529} & {\tiny{}-} & {\tiny{}-} & {\tiny{}8E-04/-} & {\tiny{}6E-02/-} & \textbf{\tiny{}-/677} & {\tiny{}5E-03/-} & {\tiny{}5E-03/-}\tabularnewline
{\tiny{}500} & {\tiny{}1.0} & {\tiny{}$10^{0}$} & {\tiny{}$10^{4}$} & {\tiny{}-} & {\tiny{}67928} & \textbf{\tiny{}48991} & {\tiny{}-} & {\tiny{}-} & {\tiny{}5E-03/-} & {\tiny{}-/1103} & \textbf{\tiny{}-/807} & {\tiny{}6E-04/-} & {\tiny{}5E-04/-}\tabularnewline
{\tiny{}500} & {\tiny{}1.0} & {\tiny{}$10^{0}$} & {\tiny{}$10^{5}$} & {\tiny{}69861} & {\tiny{}49650} & \textbf{\tiny{}35549} & {\tiny{}-} & {\tiny{}-} & {\tiny{}-/1491} & {\tiny{}-/789} & \textbf{\tiny{}-/568} & {\tiny{}4E-04/-} & {\tiny{}4E-05/-}\tabularnewline
\hline 
{\tiny{}500} & {\tiny{}1.0} & {\tiny{}$10^{1}$} & {\tiny{}$10^{5}$} & {\tiny{}-} & {\tiny{}67875} & \textbf{\tiny{}48991} & {\tiny{}-} & {\tiny{}-} & {\tiny{}7E-03/-} & {\tiny{}-/1089} & \textbf{\tiny{}-/801} & {\tiny{}2E-03/-} & {\tiny{}6E-04/-}\tabularnewline
{\tiny{}500} & {\tiny{}1.0} & {\tiny{}$10^{2}$} & {\tiny{}$10^{5}$} & {\tiny{}-} & {\tiny{}123980} & \textbf{\tiny{}24988} & {\tiny{}-} & {\tiny{}-} & {\tiny{}7E-02/-} & {\tiny{}-/2009} & \textbf{\tiny{}-/425} & {\tiny{}1E-03/-} & {\tiny{}1E-03/-}\tabularnewline
{\tiny{}500} & {\tiny{}1.0} & {\tiny{}$10^{3}$} & {\tiny{}$10^{5}$} & {\tiny{}-} & {\tiny{}-} & \textbf{\tiny{}67534} & {\tiny{}-} & {\tiny{}-} & {\tiny{}1E+00/-} & {\tiny{}6E-01/-} & \textbf{\tiny{}-/1185} & {\tiny{}1E-03/-} & {\tiny{}5E-04/-}\tabularnewline
\hline 
{\tiny{}500} & {\tiny{}5.0} & {\tiny{}$10^{0}$} & {\tiny{}$10^{5}$} & {\tiny{}68644} & {\tiny{}50567} & \textbf{\tiny{}35274} & {\tiny{}-} & {\tiny{}-} & {\tiny{}-/1441} & {\tiny{}-/791} & \textbf{\tiny{}-/556} & {\tiny{}5E-04/-} & {\tiny{}3E-05/-}\tabularnewline
{\tiny{}500} & {\tiny{}10.0} & {\tiny{}$10^{0}$} & {\tiny{}$10^{5}$} & {\tiny{}73137} & {\tiny{}50497} & \textbf{\tiny{}35396} & {\tiny{}-} & {\tiny{}-} & {\tiny{}-/1566} & {\tiny{}-/794} & \textbf{\tiny{}-/559} & {\tiny{}3E-04/-} & {\tiny{}3E-05/-}\tabularnewline
{\tiny{}500} & {\tiny{}20.0} & {\tiny{}$10^{0}$} & {\tiny{}$10^{5}$} & {\tiny{}79126} & {\tiny{}50586} & \textbf{\tiny{}35242} & {\tiny{}-} & {\tiny{}-} & {\tiny{}-/1599} & {\tiny{}-/760} & \textbf{\tiny{}-/534} & {\tiny{}2E-04/-} & {\tiny{}3E-05/-}\tabularnewline
\hline 
\hline 
{\tiny{}1000} & {\tiny{}1.0} & {\tiny{}$10^{0}$} & {\tiny{}$10^{3}$} & {\tiny{}-} & {\tiny{}-} & \textbf{\tiny{}30340} & {\tiny{}-} & {\tiny{}-} & {\tiny{}6E-03/-} & {\tiny{}3E-02/-} & \textbf{\tiny{}-/2868} & {\tiny{}2E-03/-} & {\tiny{}6E-03/-}\tabularnewline
{\tiny{}1000} & {\tiny{}1.0} & {\tiny{}$10^{0}$} & {\tiny{}$10^{4}$} & {\tiny{}-} & {\tiny{}27184} & \textbf{\tiny{}16540} & {\tiny{}-} & {\tiny{}-} & {\tiny{}4E-03/-} & {\tiny{}-/2250} & \textbf{\tiny{}-/1380} & {\tiny{}1E-04/-} & {\tiny{}1E-04/-}\tabularnewline
{\tiny{}1000} & {\tiny{}1.0} & {\tiny{}$10^{0}$} & {\tiny{}$10^{5}$} & {\tiny{}-} & {\tiny{}35192} & \textbf{\tiny{}27672} & {\tiny{}-} & {\tiny{}-} & {\tiny{}4E-04/-} & {\tiny{}-/2952} & \textbf{\tiny{}-/2515} & {\tiny{}3E-02/-} & {\tiny{}2E-05/-}\tabularnewline
\hline 
{\tiny{}1000} & {\tiny{}1.0} & {\tiny{}$10^{1}$} & {\tiny{}$10^{5}$} & {\tiny{}-} & {\tiny{}27217} & \textbf{\tiny{}16540} & {\tiny{}-} & {\tiny{}-} & {\tiny{}4E-03/-} & {\tiny{}-/2298} & \textbf{\tiny{}-/1411} & {\tiny{}3E-02/-} & {\tiny{}1E-04/-}\tabularnewline
{\tiny{}1000} & {\tiny{}1.0} & {\tiny{}$10^{2}$} & {\tiny{}$10^{5}$} & {\tiny{}-} & {\tiny{}-} & \textbf{\tiny{}16129} & {\tiny{}-} & {\tiny{}-} & {\tiny{}4E-02/-} & {\tiny{}3E-02/-} & \textbf{\tiny{}-/1461} & {\tiny{}2E-02/-} & {\tiny{}3E-03/-}\tabularnewline
{\tiny{}1000} & {\tiny{}1.0} & {\tiny{}$10^{3}$} & {\tiny{}$10^{5}$} & {\tiny{}-} & {\tiny{}-} & \textbf{\tiny{}11325} & {\tiny{}-} & {\tiny{}-} & {\tiny{}3E-01/-} & {\tiny{}2E-01/-} & \textbf{\tiny{}-/1155} & {\tiny{}7E-03/-} & {\tiny{}3E-03/-}\tabularnewline
\hline 
{\tiny{}1000} & {\tiny{}5.0} & {\tiny{}$10^{0}$} & {\tiny{}$10^{5}$} & {\tiny{}-} & {\tiny{}35564} & \textbf{\tiny{}27810} & {\tiny{}-} & {\tiny{}-} & {\tiny{}4E-04/-} & {\tiny{}-/2986} & \textbf{\tiny{}-/2340} & {\tiny{}3E-02/-} & {\tiny{}2E-05/-}\tabularnewline
{\tiny{}1000} & {\tiny{}10.0} & {\tiny{}$10^{0}$} & {\tiny{}$10^{5}$} & {\tiny{}-} & {\tiny{}35515} & \textbf{\tiny{}27973} & {\tiny{}-} & {\tiny{}-} & {\tiny{}4E-04/-} & {\tiny{}-/2983} & \textbf{\tiny{}-/2354} & {\tiny{}3E-02/-} & {\tiny{}2E-05/-}\tabularnewline
{\tiny{}1000} & {\tiny{}20.0} & {\tiny{}$10^{0}$} & {\tiny{}$10^{5}$} & {\tiny{}-} & {\tiny{}-} & \textbf{\tiny{}28033} & {\tiny{}-} & {\tiny{}-} & {\tiny{}4E-04/-} & {\tiny{}7E-06/-} & \textbf{\tiny{}-/2358} & {\tiny{}3E-02/-} & {\tiny{}2E-05/-}\tabularnewline
\end{tabular}
\par\end{centering}
\caption{Iteration counts, runtimes, and residuals (see \eqref{eq:combined_resid})
for the Nonconvex QP Problem in Subsection~\ref{subsec:nonconvex_qp}.
Entries marked with \textquotedblleft --\textquotedblright{} are those
that either: (i) obtained a solution with a residual below the prescribed
tolerance; or (ii) did not obtain a solution within the given time
limit.\label{tab:qp}}
\end{table}

We now describe the experiment parameters for the problem instances
considered. First, we chose $\ell=25$, varied $n$ across different problem instances, set $\hat{\rho}=\hat{\eta}=10^{-5}$, and ensured
all generated matrices were fully dense. Second, the entries
of $Q$, $B$, $C$, and $d$ (resp. $D$) were generated by sampling
from the uniform distribution ${\cal U}[0,1]$ (resp. ${\cal U}\{1,...,1000\}$), and the vector $b$ was set to $b=Q(u)$ where $u$
is a random vector in ${\cal U}[-r,r]^n$. Third, the initial starting point $z_{0}$
was a set to be a random vector in ${\cal U}[-r,r]^n$. 
Finally, all experiments were run with a time limit of 3000 seconds,
and the tables of this subsection also report the minimum of the aggregate residuals 
\begin{equation}
\hat{r}:=\max\left\{ \frac{{\rm dist }(0,\nabla f(\hat{z})+\pt h(\hat{z})+\nabla g(\hat{z})\hat{p})}{1+\|\nabla f(z_{0})\|}, \frac{{\rm dist}(g(\hat{z}),N_{{\cal K}^*}(\hat{p}))}{1+{\rm dist}(g(z_{0}),-{\cal K})}\right\}. \label{eq:combined_resid}
\end{equation}
It is worth mentioning that we only report the above  residuals in our numerical experiments because it is (computationally) difficult to choose the right parameters in the S-prox-ALM that guarantee convergence (see Section~\ref{sec:conclusion} for more details).

\subsection{Comments about the numerical results} \label{subsec:summary}

Overall, the most efficient methods for the above experiments were the NL-IAPIAL variants (IPL and IPL(A)). IPL(A) performed particularly well on the linearly-constrained instances where the ratio $L_f/m$ was relatively small. Between the two NL-IAPIAL variants, IPL(A) is substantially more efficient. In the QC-QP experiments, we also noticed that the results of IPL variants did not fluctuate as much as the ones of HiAPeM across different problem instances. 

We conjecture that IPL and IPL(A) perform significantly better than
HiAPeM and iALM on some instances because they apply their multiplier updates more often. 

\section{Concluding Remarks}
\label{sec:conclusion}

We first discuss how the n-PAL methods and PAL methods described in the \textit{Overview of AL methods} part Section~\ref{sec:intro} above compare to one another. First, the subproblems generated by the n-PAL methods can be nonconvex whereas the ones generated by the PAL methods are always strongly convex.
Second, some n-PAL algorithms compute the approximate stationary point $z_k$ of  ${\cal L}_{\beta_k}(\cdot;p_{k-1})$  by using prox-type
methods that generate a sequence of
convex subproblems similar to those of the PAL methods.
Hence, the subproblems generated by the n-PAL methods 
are generally
much harder to solve than those generated by the PAL methods.

We now give a detailed comparison of NL-IAPIAL with the HiAPeM of \cite{HybridPenaltyAugLag19}. Both methods employ an ACG-type subroutine to inexactly solve a generated sequence of strongly convex proximal subproblems. 
Using nearly the same assumptions as in this paper and denoting $\varepsilon = \min\{\hat \rho, \hat \eta\}$, \cite{HybridPenaltyAugLag19} establishes an improved ${\cal O}(\varepsilon^{-2.5}\log\varepsilon^{-1})$ ACG iteration complexity of HiAPeM starting from any point in $\dom h$ for problems where ${\cal K}=\{0\}\times \r_+^n$. 
However, as noted in the \textit{Related works} part of Section~\ref{sec:intro}, HiAPeM is neither a PAL method (like NL-IAPIAL), nor an n-PAL method (like the iALM of \cite{ImprovedShrinkingALM20}), but rather an inexact PPM applied to nonconvex problem \eqref{eq:main_prb} (see, for example, \cite{rockafellar1976augmented} for the analysis of inexact PPMs for solving \eqref{eq:main_prb} in the convex setting). Loosely speaking, for some suitable prox stepsize $\lam > 0$, its $k$-th prox iteration computes an approximate stationary point $z_k$ of the strongly convex subproblem $\min_{z} \{\lam \phi(z) + \|z - z_{k-1}\|^2 / 2 : g(z) \preceq_{\cal K} 0\}$ by using either an accelerated penalty method or an accelerated AL method.
It is worth mentioning that in the case where $f$ is convex, solving the $k$-th subproblems corresponds to inexactly solving
\[
\partial_z {\cal L}_0(z;p) + \frac1\lam_k (z-z_{k-1}) \ni 0,  \quad - \partial_p {\cal L}_0(z;p) \ni 0,
\]
for $(z,p)=(z_k, p_k)$ (cf. \eqref{eq:pal_update} and \eqref{eq:n-pal_update}). 

We next compare NL-IAPIAL with the S-prox-ALM of \cite{ADMMJzhang-ZQLuo2020},
which is neither a PAL nor n-PAL method,
but is based on the augmented Lagrangian function and performs
multiplier updates similar to
the ones in PAL or n-PAL methods.
First, it is shown in \cite{ADMMJzhang-ZQLuo2020} that
S-prox-ALM has an
${\cal O}(\varepsilon^{-2})$ iteration complexity under the assumption that
$g$ is affine and the strong assumption
that the function $h$ in \eqref{eq:main_prb}  is the indicator function of a polyhedron. Second,
S-prox-ALM generates a sequence of proximal subproblems as in \eqref{eq:approx_primal_update}, 
but applies a single composite gradient step to inexactly solve a variant\footnote{Instead of inexactly minimizing the function $\lam {\cal L}(\cdot;p_{k-1}) + \|\cdot-z_{k-1}\|^2/2$, the S-prox-ALM exactly minimizes the linear approximation of the function $\lam {\cal L}(\cdot;p_{k-1}) + \|z-\tilde{z}_{k-1}\|/2$ for a point $\tilde{z}_{k-1}$ different from ${z_{k-1}}$. Hence, S-prox-ALM is neither a PAL method nor an n-PAL method.} of \eqref{eq:approx_primal_update} instead of an ACG-type subroutine. 
Finally, while the NL-IAPIAL method only requires choosing its parameters based
 on the scalars $m_{f}$, $L_{f}$, $L_g$, and $M_g$ to guarantee convergence, the S-prox-ALM requires
choosing its parameters based on the supremum of a set of Hoffman
constants (see the proof of \cite[Lemma 3.10]{ADMMJzhang-ZQLuo2020} and \cite[Lemma 4.8]{ADMMJzhang-ZQLuo2020})
that is generally difficult to compute and compare with the other constants of NL-IAPIAL.

Finally, it is worth mentioning that NL-IAPIAL is a slightly modified version of the proximal method of multipliers (PMM) studied by Rockafellar in \cite{MR0418919}. More specifically, the $k$-th iteration of the PMM consists of (3)--(4) with ${\cal K} = \r_+^\ell$ and $\lam_k=\beta_k$ for every $k$ and, hence, can be
viewed  as
inexactly solving (8) with $\lam_k=\beta_k$ and $\chi_k = 1$ so that both 
inclusions on it have the same prox stepsize.
Under the assumption that (1) is a convex optimization problem, Rockafellar then uses
classical results for inexact 
proximal point methods to analyze the
convergence of the PMM.
However, the approach outlined above
does not generalize to the nonconvex setting
in several aspects, namely:
(i) while the PMM converges when $\beta_k$ is constant, convergence of NL-IAPIAL requires $\beta_k$ to grow significantly; (ii) in contrast to the PMM, NL-IAPIAL chooses
$\lam_k$  to be a sufficiently small constant to convexify the subproblem in (3); and (iii) the analysis of NL-IAPIAL does not rely on proximal point theory for
maximal monotone operators since
the operator $(z,p)\mapsto[\partial_z {\cal L}_0(z;p), -\partial_p {\cal L}_0(z;p)]$ is not
monotone in the setting of
NL-IAPIAL.

\begin{appendices}

\section{Review of an ACG Algorithm} \label{sec:acg}

This section reviews an ACG algorithm 
invoked by	NL-IAPIAL  for solving  the sequence of
subproblems \eqref{eq:approx_primal_update} which
arise during its implementation.
It also describes a  bound on the number of ACG iterations performed in order to obtain a certain type of inexact solution of each subproblem.

    Consider the composite optimization problem
	
	\begin{equation}\label{eq:main_prob_acg}
	\min \ \{\psi(x):=\psi_s(x)+\psi_n(x) :  x \in \Re^n\}, 
	\end{equation}
	where the following conditions are assumed to hold:
	where the following conditions are assumed to hold:
	\begin{itemize}
		\item [{\bf(B1)}]
		$\psi_n:\Re^n\rightarrow (-\infty,+\infty]$ is a proper closed  convex  function;
		\item [{\bf (B2)}]$\psi_s$ is a convex differentiable function on $\dom \psi_n$
		and there exists  $({\widetilde \mu},{\widetilde M})\in \r_{+}^2$ satisfying $\widetilde M > \widetilde \mu$ and
		\begin{equation}\label{ineqs in Assump B2}
		    {\widetilde \mu}\|u-x\|^2/2 \leq \psi_s(u)-\ell_{\psi_s}(u; x) 
		\leq  {\widetilde M}\|u-x\|^2/2
		\end{equation}
		for every $x, u \in \dom \psi_n$, where $\ell_{\psi_s}(\cdot \,;\cdot)$ is defined in \eqref{eq:defell}.
	\end{itemize}

	The ACG algorithm, given $(y_0, \tilde \sigma)\in \dom \psi_n\times \r_{++}$, inexactly solves \eqref{eq:main_prob_acg} by computing a triple $(y,u,\eta)\in \dom \psi_n\times \Re^n\times \Re_{+}$ satisfying
    \begin{equation}\label{eq:inexact_acg_prb}
        u\in  \partial_{\eta}(\psi_s+\psi_n)(y) \quad \|u\|^{2}+2\eta\le \tilde\sigma^2\|y_{0}-y+u\|^{2}.
    \end{equation}
With this in mind, we now state the ACG variant considered in this paper.
	
\noindent \rule[0.5ex]{1\columnwidth}{1pt}

\noindent \textbf{ACG}

\noindent \rule[0.5ex]{1\columnwidth}{1pt}
\begin{itemize}
\item[(0)] Let a pair of functions $(\psi_{s},\psi_{n})$ satisfying \textbf{(B1)
}and \textbf{(B2)} for some $({\widetilde \mu},\widetilde M)\in\r_{+}^{2}$, a scalar
$\tilde{\sigma}>0$, and an initial point $y_{0}\in\dom\psi_{n}$
be given; set $x_0=y_0$, $A_{0}=0$, $\tau_0=1$,  and $j=0$;
\item[(1)]  $\zeta = 1/({\widetilde M} - {\widetilde \mu})$ and compute the quantities 
\begin{align}
 a_{j+1} & =\frac{\zeta\tau_{j}+\sqrt{(\zeta\tau_{j})^{2}+4\tau_{j}A_{j}}}{2}, \quad A_{j+1}=A_{j}+a_{j+1}, \quad \tilde{x}_{j+1} =\frac{A_{j} y_j + a_{j+1} x_j}{A_{j+1}}\nonumber\\
\tau_{j+1} & =\tau_{j} +{\widetilde \mu} a_{j+1}, \quad y_{j+1} =\argmin_{y\in\rn}\left\{ \ell_{\psi_s}(y;\tilde{x}_{j+1}) + \psi_n(y) +\frac{{\widetilde M}}{2}\|y-\tilde{x}_{j+1}\|^{2}\right\}, \label{eq:prox-acg}\\
x_{j+1} & = \frac{1}{\tau_{j+1}} \left[\frac{a_{j+1}}{\zeta}(y_{j+1} - \tilde{x}_{j+1}) + {\widetilde \mu} a_{j+1} y_{j+1} + \tau_j x_j \right]\nonumber;
\end{align}
\item[(2)] compute the quantities
\begin{align*}
 u_{j+1} &= {\widetilde \mu}(y_{j+1} - x_{j+1}) + \frac{x_{0}-x_{j+1}}{A_{j+1}}, \\
 \eta_{j+1} &= \frac{1}{2 A_{j+1}}\left(\|x_0 - y_{j+1}\|^2 - \tau_{j+1}\|x_{j+1}-y_{j+1}\|^2\right);
\end{align*}
\item[(3)] if the inequality 
\[
\|u_{j+1}\|^{2}+2\eta_{j+1}\leq\tilde{\sigma}^{2}\|y_{0}-y_{j+1}+u_{j+1}\|^{2}
\]
holds, then stop and output $(y,u,\eta):=(y_{j+1},u_{j+1},\eta_{j+1})$;
otherwise, set $j=j+1$ and go to (1).
\end{itemize}
\rule[0.5ex]{1\columnwidth}{1pt}

 	Some remarks about ACG follow. First, the most common way of describing an iteration of ACG is as in step~1. 
	Second, the
	auxiliary iterates pair $\{(u_j,\eta_j)\}$ computed  in step~2 is used to develop a stopping criterion for ACG  when it is called as a subroutine for solving the subproblems generated in step~1 of NL-IAPIAL
	in Subsection~\ref{sec:NLIAPIAL}.
Third, it can be shown (see for example   \cite{florea2018accelerated, fistaReport2021}) that ACG (without steps 2 and 3) with ${\widetilde \mu}=0$ corresponds to the well-known FISTA algorithm. 
	Fourth, the sequence $\{A_j\}$ has the following increasing property:
	$$ A_{j}\geq\frac{1}{{\widetilde M-\widetilde \mu}}\max\left\{\frac{j^{2}}{4},\left(1+\sqrt{\frac{{\widetilde \mu}}{4(\widetilde M-\widetilde \mu)}}\right)^{2(j-1)}\right\}, \qquad \forall j\geq 1. 
	$$
Finally, notice that each iteration of an ACG-type method consists of an ${\cal O}(1)$ number of $\psi_s$ function, $\psi_s$ gradient, and $\psi_n$ prox evaluations.

It is worth mentioning that adaptive variants\footnote{The closest variant to ACG in this paper can be found in \cite[Section 5.2]{KongThesis2021}.} of ACG have been studied, for example, in \cite{beck2009fast, KongThesis2021, pmlr-v32-lin14, nesterov2012gradient, parikh2014proximal}. One kind of adaptiveness
used in these variants, which is also used
inside some methods benchmarked in Section~4,
involves replacing ${\widetilde M}$ in
the computation of $y_{j+1}$ in step~1 by
an estimate $M_{j+1}$ computed as follows:
$M_{j+1}$ is initially set to be $M_{j}$ and, if necessary, is increased (either additively, multiplicatively, or both) and
step 1 is repeated a few times
(if needed)
until the inequality $\psi_s(y_{j+1}) - \ell_{\psi_s}(y_{j+1};\tilde x_{j+1})\leq M_{j+1} \|y_{j+1} - \tilde{x}_{j+1}\|^2 / 2 $ is satisfied. Observe that
every time
step 1 is repeated within the $j$-th iteration of ACG,
$\zeta$ changes (and hence
so do $a_{j+1}$, $A_{j+1}$, $\tilde x_{j+1}$, $\tau_{j+1}$, 
and $y_{j+1}$) since $M_{j+1}=\tilde M$ changes
adaptively.

	The next result, whose proof can be found in \cite[Lemma~2.13]{fistaReport2021}, summarizes  the main properties of the above ACG.

	\begin{proposition}\label{prop:nest_complex}
		Let $\{(y_j, u_j,\eta_j)\}_{j\geq 1}$ be the sequence generated by ACG applied to \eqref{eq:main_prob_acg},
		where $(\psi_s,\psi_n)$ is a given pair of data functions satisfying {\bf (B1)} and {\bf (B2)}. 
		Then, the following statements hold: 
		\begin{itemize}
			\item [(a)] for every $j\geq 1$, we have $\eta_j\geq0$ and  
			$u_j\in  \partial_{\eta_j}(\psi_s+\psi_n)(y_j)$;
			\item[(b)] for any $\tilde \sigma>0$, the ACG method outputs a triple $(y,u,\eta)\in \dom \psi_n\times\Re^n\times\Re_+$ 
			  satisfying  
			\begin{equation}\label{mainprob:nesterov}
			u\in  \partial_{\eta}(\psi_s+\psi_n)(y) \quad \|u\|^{2}+2\eta\le \tilde\sigma^2\|y_0-y+u\|^{2}
			\end{equation}
			in at most 
			\begin{equation}
\left\lceil 1+\left(\frac{1}{2}+\sqrt{\frac{\widetilde M-\widetilde\mu}{\widetilde\mu}}\right)\log_{1}^{+}\widetilde{\cal A} \right\rceil
\end{equation}
iterations, where
$$
\widetilde{\cal A}:=(2\widetilde\mu+3)(1+\widetilde\sigma)^{2}(\widetilde M-\widetilde\mu)\widetilde\sigma^{-2}.
$$
	\end{itemize}
\end{proposition}

\section{Convex Analysis} \label{app:cvx}

The first result presents some well-known 
properties about the projection and distance functions over a closed
convex set.
\begin{lem} \label{lem:dist_props}
Let $\cK \subseteq\rn$ be a nonempty  closed convex cone and $S$ be a nonempty closed convex set. Then the following properties hold:
\begin{itemize}
    \item[(a)] for every $u,z\in\rn$, we have $\|\Pi_{S}(u)-\Pi_{S}(u)\|\leq\|u-z\|$;
    \item[(b)] the function $d(\cdot):={\rm dist}^{2}(\cdot,S)/2$ is differentiable,
    and its gradient is given by
    \begin{equation}
    \nabla d(u)=u-\Pi_{S}(u)\in N_{S}(\Pi_{S}(u))\quad\forall u\in\rn; \label{eq:proj_incl}
    \end{equation}
    \item[(c)] it holds that $u\in N_{\cK^*}(p)$ if and only if $\inner{u}{p} = 0$, $u\in -\cK$, and $p \in \cK^*$.
\end{itemize}
\end{lem}
\begin{proof}
See \cite[Theorem 5.4]{beck2017first} for (a), \cite[Example 6.61]{beck2017first} and \cite[Theorem 6.39(ii)]{beck2017first} for (b), and \cite[Example 11.4]{VariaAna} for (c).
\end{proof}

The next result presents a well-known fact (see, for example, \cite[Sub-subsection 2.13.2]{Dattorro05convexoptimization}) about closed convex cones.

\begin{lem} \label{lem:cone_generator}
    For any closed convex cone $\cK$, we have that $x\in \inte\cK$ if and only if 
    \begin{align}
    \inner{x}{p} >0 \quad \forall p\in \cK^* \quad \text{such that} \quad \|p\|=1.
    \end{align}
\end{lem}

The below technical result presents a fact about approximate subdifferentials, and its proof can be found, for example, in \cite[Lemma A.3]{RJWIPAAL2020}.

\begin{lemma} \label{lem:auxNewNest2}
		Let a proper function $\tilde \phi: \Re^n \to (-\infty,\infty]$,
		scalar $\tilde \sigma \in (0,1)$  and
		$(x_0,x) \in \Re^n \times     \dom \tilde \phi$ be given, and assume that there exists
		$(v,\varepsilon)$ such that 
		\begin{gather}
		v\in \partial_{\varepsilon} \left(\tilde\phi+\frac{1}{2}\|\cdot-x_0\|^2\right) (x), 
		\quad  \|v\|^2 + 2 \varepsilon \leq \tilde\sigma^2 \|v+x_{0}-x\|^{2}. \label{Auxeq:prox_incl}  
		\end{gather}
		Then, for every $x\in \Re^n$ and $s>0$, we have
		\[
		\tilde \phi(x)+\frac{1}{2} \left[ 1 - \tilde\sigma^2 ( 1 + s^{-1}) \right]\|v+x_0 - x\|^{2}\le \tilde \phi(z) +\frac{s+1}{2} \|z-x_0\|^2.
		\]
	\end{lemma}

\section{Proof of Lemma~\ref{prop:refinement} and Lemma~\ref{prop:refinement1}(a)} \label{app:refine}

The first result, whose proof is given in \cite[Appendix A]{WJRproxmet1}, describes some properties of 
a composite gradient step.

\begin{lemma}\label{lem:approxsolreps}
Assume that $\tilde h\in \cConv \rn$,
$\tilde g$ is  a differentiable function on ${\dom \tilde h}$,
and $(z,\varepsilon) \in {\dom \tilde h} \times \Re_+$ is such that
\begin{equation}\label{Inc:repsdif}
0 \in \partial_{\varepsilon}( \tilde g + \tilde h) (z).
\end{equation}
Assume also that
there exists $\tilde L>0$ such that  \begin{equation}\label{uppercurvature_hatg}
\tilde g(u)- \ell_{\tilde g}(u; z) \leq\frac{\tilde L}{2}\|u-z\|^2\qquad \forall u \in {\mathcal \dom \tilde h},
\end{equation}
and define
\begin{equation}\label{eq:def_zhat}
{\tilde z} := \argmin_u \left\{ \ell_{\tilde g}(u;z)  + \tilde h(u) + \frac{\tilde L}{2} \|u-z\|^2  \right \},  \qquad \widetilde w:= \tilde L(z-\tilde z).
\end{equation}
Then,  the quadruple $(z,\tilde z, \widetilde w, \varepsilon)$ satisfies 
\begin{equation}\label{eq:inclusion_w}
 \widetilde w \in \nabla \tilde g(z) + \partial \tilde h({\tilde z}), \qquad \widetilde w \in \nabla \tilde g(z) + \partial_{{\varepsilon}} \tilde h(z), \qquad \|\widetilde w\| \leq \sqrt{2 \tilde L\varepsilon}. 
\end{equation}
\end{lemma}

The next result specializes the above results to our setting and gives two technical identities.

\begin{lem} \label{lem:spec_refine}
    Let  $\widetilde{\cal{L}}_{\beta}$ be as in \eqref{eq:L_tilde_def},
    let $\beta_k$, $(z_k, v_k, \varepsilon_k)$, $\hat{z}_{k}$, and $(z_{k-1}, p_{k-1})$ be
    as in the $k$-th iteration of NL-IAPIAL, and define  
\begin{equation}
    \tilde g:=\lam \widetilde{{\cal L}}_{\beta_k}(\cdot; p_{k-1})-\left\langle v_{k},\cdot\right\rangle + \frac{1}{2}\|\cdot-z_{k-1}\|^{2}, \quad \tilde h:=\lam h, \quad \tilde{w}_k := \widetilde{\cal M}_k(z_k-\hat z_k) \label{eq:psi_aux_defs}
\end{equation}
Then, it holds that 
\begin{equation}\label{incl-ineq-auxappendixprop}    
\tilde{w}_k\in\nabla \tilde g(z_k)+\pt \tilde h(\hat{z}_k), 
\quad \tilde{w}_k \in\nabla \tilde g(z_k)+\pt_{\varepsilon_k} \tilde h(z_k), 
\quad \|\tilde{w}_k\|\leq \sqrt{2\varepsilon_k\widetilde{\cal M}_k}.
\end{equation}
where $\widetilde{\cal M}_k$ is as in \eqref{eq:acg_input_aux_defs}. Moreover, it holds that
\begin{align}
\frac{1}{\lambda}\left(r_k+\nabla \tilde g(z_k)\right) & = \nabla_z \widetilde{{\cal L}}_{\beta_k}(z_k; p_{k-1}) 
= \nabla f(z_k) + \nabla g(z_k) \Pi_{{\cal K}^*}(p_{k-1} + \beta_{k} g(z_k)) \quad \forall u\in\rn. \label{eq:grad_psi_s}
\end{align}
\end{lem}

\begin{proof}
It follows from the definition of $\varepsilon$-subdifferential in  \eqref{def:epsSubdiff} and the fact that the triple $(z_{k},v_{k},\varepsilon_{k})$ satisfies the inclusion in \eqref{eq:prox_incl}  that  \eqref{Inc:repsdif} holds with $(\tilde g, \tilde h)$ and $(z,\varepsilon)=(z_k,\varepsilon_k)$. 
In view of assumptions (A1)--(A3), 
Lemma~\ref{lem:dist_smoothness}, and the definition of $\widetilde{\cal M}_k$ in \eqref{eq:acg_input_aux_defs},  the functions pair $(\tilde g,\tilde h)$ defined above satisfies the assumptions of Lemma~\ref{lem:approxsolreps} with $\widetilde L=\widetilde{\cal M}_k$. Note also that  the element  $\tilde z$ computed according to \eqref{eq:def_zhat} corresponds to $\hat z_k$ computed  in \eqref{eq:z_k_def}, in view of the definition of $r_k$   given in \eqref{eq:dual_update2}.  Hence, it follows from Lemma~\ref{lem:approxsolreps} that \eqref{incl-ineq-auxappendixprop} holds. The last statement of the lemma follows from the definition of $r_k$ in \eqref{eq:dual_update2}  and Lemma~\ref{lem:dist_smoothness}(b).
\end{proof}

We are now ready to prove Lemma~\ref{prop:refinement1}(a).

\begin{proof}[Proof of Lemma~\ref{prop:refinement1}(a)]
\noindent Let $\tilde h$ be as in \eqref{eq:dual_update2}. In view of \eqref{def:epsSubdiff},  the definitions of $p_k$ and  $w_k$ in \eqref{eq:dual_update2} and  \eqref{eq:refine_aux_defs}, respectively, and Lemma~\ref{lem:spec_refine}, we have
\begin{align}
w_k  = \frac{1}{\lam} \left(r_k + \widetilde{\cal M}_k(z_k - \hat z_k)\right) &\in \frac{1}{\lam} \left(r_k + \nabla \tilde g(z_k) + \pt_{\varepsilon_k} \tilde h(z_k) \right) \nonumber \\
& = \nabla f(z_k) + \nabla g(z_k) \Pi_{{\cal K}^*}(p_{k-1} + \beta_{k} g(z_k)) + \pt_{(\lam^{-1}\varepsilon_k)} h(z_k) \nonumber \\
& = \nabla f(z_k) + \nabla g(z_k) p_k + \pt_{(\lam^{-1}\varepsilon_k)}  h(z_k), \nonumber
\end{align}
which proves  the inclusion in \eqref{eq:weak_refine}.
We now show that the inequalities in \eqref{eq:weak_refine} hold. The bound on $\varepsilon_k$ in \eqref{eq:weak_refine} follows immediately from the inequality in \eqref{eq:prox_incl} and the definition of  $r_k$  given in \eqref{eq:refine_aux_defs}. Now, it follows from the inequality in \eqref{eq:prox_incl},  the definition of $r_k$ and $w_k$ in \eqref{eq:dual_update2} and \eqref{eq:refine_aux_defs}, respectively, the triangle inequality for norms, and Lemma~\ref{lem:spec_refine}  that
\begin{align}
\lam \|w_k\| &= \|r_k + \widetilde{\cal M}_k(z_k - \hat z_k)\| \leq \|r_k\|+ \widetilde{\cal M}_k\|z_{k}-\hat{z}_{k}\| \nonumber \\
& \leq \|r_k\| + \sqrt{2\varepsilon_{k}\widetilde{\cal M}_k} 
\leq \left(1+\sigma_k\sqrt{\widetilde{\cal M}_k}\right)\|r_{k}\|, \label{eq:delta_z_hat_bd}
\end{align}
which immediately implies the desired bound on $\|w_k\|$ in view of the definition of $\sigma_k$ in \eqref{eq:acg_input_aux_defs}.
\end{proof}

We now close with the proof of Lemma~\ref{prop:refinement}. 

\begin{proof}[Proof of Lemma~\ref{prop:refinement}]

We first show that the inclusion in \eqref{eq:Relations(a)-PropRef} holds. Using the first identity in \eqref{eq:grad_psi_s}, 
Lemma~\ref{lem:spec_refine}, Lemma~\ref{lem:dist_smoothness}(b), 
and the definitions of $w_k$ and $(\hat w_k, \hat p_k)$ in \eqref{eq:refine_aux_defs} and \eqref{eq:refined_points}, respectively, we have 
\begin{align}
\hat w_k & = \frac{1}{\lam} \left[r_k + \widetilde{\cal M}_k(z_k - \hat z_k)\right] + \left[\nabla_z\widetilde{{\cal L}}_{\beta_k} (\hat{z}_{k}; p_{k-1})-\nabla_z\widetilde{{\cal L}}_{\beta_k} (z_{k}; p_{k-1})\right] \nonumber \\[3mm] 
& \in \frac{1}{\lam} \left[r_k + \nabla \tilde g(z_k) + \pt \tilde h(\hat z_k) \right]+ \left[\nabla_z\widetilde{{\cal L}}_{\beta_k} (\hat{z}_{k}; p_{k-1})-\nabla_z\widetilde{{\cal L}}_{\beta_k} (z_{k}; p_{k-1})\right] \nonumber \\[3mm]
& = \nabla_z\widetilde{{\cal L}}_{\beta_k} (\hat{z}_{k}; p_{k-1}) + \pt h(\hat z_k) = \nabla f(\hat z_k) + \nabla g(\hat z_k) \Pi_{{\cal K}^*}(p_{k-1} + \beta_{k} g(\hat z_k)) + \pt h(\hat z_k) \nonumber \\[3mm]
& = \nabla f(\hat z_k) + \nabla g(\hat z_k) \hat p_k + \pt h(\hat z_k), \nonumber
\end{align}
which is the desired inclusion in \eqref{eq:Relations(a)-PropRef}. We now show that the bound on $\|\hat w_k\|$ in \eqref{eq:strong_refine} holds. Using its definition in \eqref{eq:refined_points},   Lemma~\ref{lem:dist_smoothness}(c) and the definition of $\widetilde{\cal M}_k$ in \eqref{eq:acg_input_aux_defs}, 
the inequality in \eqref{eq:prox_incl},
the definition of $r_k$ given in \eqref{eq:refine_aux_defs}, 
Lemma~\ref{lem:spec_refine}, the triangle inequality for norms, 
and \eqref{eq:delta_z_hat_bd},  we have 
\begin{align}
\lam \|\hat w_k\| & \leq \lam \|w_k \| + \lam \|\nabla_z\widetilde{{\cal L}}_{\beta_k} (\hat{z}_{k}; p_{k-1})-\nabla_z\widetilde{{\cal L}}_{\beta_k} (z_{k}; p_{k-1}) \| \nonumber \\
& \leq \left(1+\sigma_k\sqrt{\widetilde{\cal M}_k}\right)\|r_{k}\| + \widetilde{\cal M}_k \|\hat z_k - z_k\| \leq \left(1+ 2\sigma_k\sqrt{\widetilde{\cal M}_k}\right)\|r_{k}\|, \nonumber
\end{align}
which immediately implies the desired bound on $\|\hat{w}_k\|$ in view of  the definition of $\sigma_k$ in \eqref{eq:acg_input_aux_defs}.

To show the bound on $\hat{q}_k$, we first use the definitions of $B_g^{(1)}$,   $p_{k}$, and $\hat{p}_{k}$ given in \eqref{eq:bd_Psi_val},  \eqref{eq:dual_update2}, and \eqref{eq:refined_points}, respectively, the last two inequalities in  \eqref{eq:delta_z_hat_bd}, the Mean Value Inequality, and Lemma~\ref{lem:dist_props}(a) to obtain
\begin{align*}
\frac{1}{\beta_k}\|\hat{p}_{k}-p_{k}\| & =\frac{1}{\beta_k}\left\Vert \Pi_{{\cal K}^{*}}\left(p_{k-1}+\beta_kg(\hat{z}_{k})\right)-\Pi_{{\cal K}^{*}}\left(p_{k-1}+\beta_kg(z_{k})\right)\right\Vert \leq\frac{1}{\beta_k}\|\beta_kg(\hat{z}_{k})-\beta_kg(z_{k})\| \\
 & \leq \sup_{t \in [0,1]} \|\nabla g(t \hat{z}_k + [1-t]z_k)\| \cdot \|\hat{z}_{k}-z_{k}\| \leq B_g^{(1)}\|\hat{z}_{k}-z_{k}\| \leq \frac{B_g^{(1)}\sigma_k}{\sqrt{\widetilde{\cal M}_k}}\|r_{k}\|.
\end{align*}
Hence, using the triangle inequality for norms and the definition of $\hat{q}_k$ given in \eqref{eq:refined_points}, we have
\[
\|\hat{q}_k\| = \frac{1}{\beta_k}\|\hat{p}_k - p_{k-1}\| \leq  \frac{1}{\beta_k}\|\hat{p}_k - p_k\| + \frac{1}{\beta_k}\|p_{k} - p_{k-1}\| \leq \frac{B_g^{(1)}\sigma_k}{\sqrt{\widetilde{\cal M}_k}}\|r_{k}\| + \frac{1}{\beta_k}\|p_{k} - p_{k-1}\|,
\]
which proves the bound on $\hat q_k$ in view of the definition of $\sigma_k$ in \eqref{eq:acg_input_aux_defs}.

To finish the proof of Lemma~\ref{prop:refinement}, it remains to show that the last three relations in \eqref{eq:Relations(a)-PropRef} hold. The last relation in \eqref{eq:Relations(a)-PropRef} follows immediately from the definition of $\hat p_k$ in \eqref{eq:refined_points}. Now, using Lemma~\ref{lem:dist_props}(b) with $S={\cal K}^*$ and $u=p_{k-1}+\beta_kg(\hat{z}_{k})$ as well as the definitions of $\hat q_k$ and $\hat p_k$ in \eqref{eq:refined_points}, we have that
\begin{equation}
g(\hat z_k) + \hat q_k = \frac{1}{\beta_k}\left[p_{k-1}+\beta_kg(\hat{z}_{k})-\hat{p}_{k}\right] \in N_{\cK^*}(\hat p_k). 
\end{equation}
Hence, the remaining relations in \eqref{eq:Relations(a)-PropRef} follow from   the above relation and Lemma~\ref{lem:dist_props}(c) with $u=g(\hat{z}_{k})+\hat q_{k}$ and $p=\hat{p}_k$.
\end{proof}

\section{Proof of Proposition~\ref{prop:weak_slater}}
\label{app:slater}

\begin{proof} [(a) $\implies$ (b)] This is immediate.

[(b) $\implies$ (c)] Suppose (b) holds. If $\bar{z}$ satisfies
(c) then we are done, so suppose that $g_{\iota}(\bar{z})\not\prec_{{\cal J}}0$
and $g_{e}(\bar{z})=0$. Our goal is to find $d\in\rn$ such that (c) holds
with $\bar{z}=\bar{z}+d$, which in view of Lemma~\ref{lem:cone_generator}
with $x=-g_{\iota}(\bar{z}+d)$ and the fact that $g_{e}$ is affine,
is equivalent to
\begin{equation}
g_{e}'(\bar{z})d=0,\quad\inf_{\|p_{\iota}\|=1,p_{\iota}\in{\cal J}^{*}}\left\langle -g_{\iota}(\bar{z}+d),p_{\iota}\right\rangle >0.\label{eq:goal}
\end{equation}
We now bound the left-hand-side of the inequality in \eqref{eq:goal}.
Using the assumption that $\nabla g_{\iota}(\cdot)$ is $L_{g_{\iota}}$-Lipschitz,
we have
\begin{align}
\inf_{\|p_{\iota}\|=1,p_{\iota}\in{\cal J}^{*}}-\left\langle g_{\iota}\left(\bar{z}+d\right),p_{\iota}\right\rangle  & =\inf_{\|p_{\iota}\|=1,p_{\iota}\in{\cal J}^{*}}-\left\langle g_{\iota}(\bar{z})+g_{\iota}'(\bar{z})d+\left[g_{\iota}\left(\bar{z}+d\right)-g_{\iota}(\bar{z})-g_{\iota}'(\bar{z})d\right],p_{\iota}\right\rangle \nonumber \\
 & \geq\inf_{\|p_{\iota}\|=1,p_{\iota}\in{\cal J}^{*}}\left\langle -g_{\iota}(\bar{z})-g_{\iota}'(\bar{z})d,p_{\iota}\right\rangle -\|g_{\iota}\left(\bar{z}+d\right)-g_{\iota}(\bar{z})-g_{\iota}'(\bar{z})d\|\nonumber \\
 & \geq\inf_{\|p_{\iota}\|=1,p_{\iota}\in{\cal J}^{*}}\left\langle -g_{\iota}(\bar{z})-g_{\iota}'(\bar{z})d,p_{\iota}\right\rangle -\frac{L_{g_{\iota}}\|d\|^{2}}{2},\label{eq:slater_main_ineq}
\end{align}
for any $d\in\rn$, so it suffices to find $d\in\rn$ so that the
last expression in \eqref{eq:slater_main_ineq} is positive. To find
an appropriate direction, we let $0\neq q_{\iota}\in\intr{\cal J}$
and consider the primal-dual conic optimization problems
\begin{equation}
\left(\underbrace{\begin{aligned}\min_{p}\  & -\left\langle p_{\iota},g_{\iota}(\bar{z})\right\rangle \\
\text{s.t.}\  & \nabla g_{\iota}(\bar{z})p_{\iota}+\nabla g_{e}(\bar{z})p_{e}=0\\
 & \left\langle q_{\iota},p_{\iota}\right\rangle =1\\
 & p_{\iota}\in{\cal K}^{*},p_{e}\in\r^{t}
\end{aligned}
}_{(P)}\right)\equiv\left(\underbrace{\begin{aligned}\max_{d,\mu}\  & \mu\\
\text{s.t.}\  & -g_{\iota}(\bar{z})-g{}_{\iota}'(\bar{z})d\succeq_{{\cal J}}\mu q_{\iota}\\
 & g_{e}'(\bar{z})d=0\\
 & d\in\rn,\mu\in\r
\end{aligned}
}_{(D)}\right).\label{eq:primal_dual_pair}
\end{equation}
Denoting $p_{\iota}^{*}$ and $(d^{*},\mu^{*})$ to be optimal solutions
of $(P)$ and $(D)$, respectively, we show that $\mu^{*}$ is positive
and then argue that $d^{*}$ is an appropriate direction. Using the
fact that $(D)$ has a Slater point (and hence strong duality holds
for \eqref{eq:primal_dual_pair}), our assumption that $-g(\bar{z})\in{\cal K}$
(and hence $-\langle p^{*},g(\bar{z})\rangle\geq0$), and \eqref{eq:sp0},
it follows that 
\begin{equation}
\mu^{*}=-\left\langle p_{\iota}^{*},g_{\iota}(\bar{z})\right\rangle =\max\left\{ \left|\left\langle \left[\begin{array}{c}
p_{\iota}^{*}\\
0
\end{array}\right],g(\bar{z})\right\rangle \right|,\left\Vert \nabla g(\bar{z})\left[\begin{array}{c}
p_{\iota}^{*}\\
0
\end{array}\right]\right\Vert \right\} \geq\tilde{\tau}_{g}\|p_{\iota}^{*}\|>0,\label{eq:pos_mu}
\end{equation}
where the last inequality follows from the second constraint in $(P)$,
the fact that $q_{\iota}\in\intr{\cal J}$, and Lemma~\ref{lem:cone_generator}
with $(p,x)=(p_{\iota}^{*},q_{\iota})$. Since $g_{e}'(\bar{z})d^{*}=0$
from the second constraint of $(D)$, it only remains to show that
the last expression in \eqref{eq:slater_main_ineq} is positive for
some positive multiple of $d^{*}$, i.e., $d=\lam d^{*}$ for some
$\lam>0$. Using the fact that $d^{*}$ is feasible to $(D)$ and
our assumption that $g_{\iota}(\bar{z})\preceq_{{\cal J}}0$ (and
hence $-\langle p_{\iota},g(\bar{z})\rangle\geq0$ for every $p_{\iota}\in{\cal J}^{*}$),
we first have that for $\lam<1$ and $d=\lam d^{*}$,
\begin{align}
 & \inf_{\|p_{\iota}\|=1,p_{\iota}\in{\cal J}^{*}}\left\langle -g_{\iota}(\bar{z})-g_{\iota}'(\bar{z})d,p_{\iota}\right\rangle -\frac{L_{g_{\iota}}\|d\|^{2}}{2}\nonumber \\
 & =\lam\left[\inf_{\|p_{\iota}\|=1,p_{\iota}\in{\cal J}^{*}}\left\langle -\frac{1}{\lam}g_{\iota}(\bar{z})-g_{\iota}'(\bar{z})d^{*},p_{\iota}\right\rangle -\frac{\lam L_{g_{\iota}}\|d^{*}\|^{2}}{2}\right]\nonumber \\
 & \geq\lam\left[\inf_{\|p_{\iota}\|=1,p_{\iota}\in{\cal J}^{*}}\left\langle -g_{\iota}(\bar{z})-g_{\iota}'(\bar{z})d^{*},p_{\iota}\right\rangle -\frac{\lam L_{g_{\iota}}\|d^{*}\|^{2}}{2}\right]\nonumber \\
 & \geq\lam\left[\mu^{*}\inf_{\|p_{\iota}\|=1,p_{\iota}\in{\cal J}^{*}}\left\langle q_{\iota},p_{\iota}\right\rangle -\frac{\lam L_{g_{\iota}}\|d^{*}\|^{2}}{2}\right]\nonumber \\
 & =\lam\left[\mu^{*}\nu-\frac{\lam L_{g_{\iota}}\|d^{*}\|^{2}}{2}\right],\label{eq:d_eps_bd}
\end{align}
where $\nu:=\inf_{\|p_{\iota}\|=1,p_{\iota}\in{\cal J}^{*}}\langle q_{\iota},p_{\iota}\rangle$.
Using \eqref{eq:pos_mu} and Lemma~\ref{lem:cone_generator} with
$(p,x)=(p_{\iota},q_{\iota})$, it holds that $\mu^{*}\nu>0$ and,
hence, there exists $\lam>0$ sufficiently small so that the last
expression in \eqref{eq:d_eps_bd} is positive. As a consequence,
it follows from \eqref{eq:slater_main_ineq} that \eqref{eq:goal}
holds, or equivalently, (c) holds with $\bar{z}=\bar{z}+\lam d^{*}$.

{[}(c) $\implies$ (a){]} Suppose (c) holds. Since $g_{e}$ is affine
and onto, its gradient matrix $G_{e}:=\nabla g_{e}$ is independent
of $z$ and has full column rank. Hence, there exists $\tau_{g_{e}}>0$
such that 
\begin{equation}
\|G_{e}p_{e}\|\ge\tau_{g_{e}}\|p_{e}\|_{1}\quad\forall p_{e}\in\r^{t}.\label{eq:wslater_bd1}
\end{equation}
On the other hand, the assumption that $g_{\iota}(\bar{z})\prec_{{\cal J}}0$,
and Lemma~\ref{lem:cone_generator} with ${\cal K}={\cal J}$ and
$x=-g_{\iota}(\bar{z})$, imply that there exists $\tau_{g_{\iota}}>0$
such that 
\[
-\left\langle p_{\iota},g_{\iota}(\bar{z})\right\rangle \ge\tau_{g_{\iota}}\|p_{\iota}\|\quad\forall p_{\iota}\in{\cal J}^{*}.
\]
Using the previous inequality and the fact that $\|\nabla g_{\iota}(z)\|$
is bounded on ${\cal H}$, we conclude that there exists $\gamma>0$
such that 
\begin{equation}
-\|\nabla g_{\iota}(z)p_{\iota}\|-2\gamma\inner{p_{\iota}}{g_{\iota}(\bar{z})}\ge[2\gamma\tau_{g_{\iota}}-\|\nabla g_{\iota}(z)\|]\cdot\|p_{\iota}\|\geq\tau_{g_{\iota}}\|p_{\iota}\|_{1}\quad\forall z\in{\cal H}.\label{eq:wslater_bd2}
\end{equation}
Relations \eqref{eq:wslater_bd1}, \eqref{eq:wslater_bd2}, and the
reverse triangle inequality, then imply that for every $z\in{\cal H}$,
\begin{align*}
 & \|\nabla g(z)p\|-2\gamma\left\langle p,g(\bar{z})\right\rangle =\|\nabla g_{\iota}(z)p_{\iota}+G_{e}p_{e}\|-2\gamma\left\langle p_{\iota},g_{\iota}(\bar{z})\right\rangle \\
 & \geq\|G_{e}p_{e}\|-\|\nabla g_{\iota}(z)p_{\iota}\|-2\gamma\left\langle p_{\iota},g_{\iota}(\bar{z})\right\rangle \geq\tau_{g_{e}}\|p_{e}\|_{1}+\tau_{g_{\iota}}\|p_{\iota}\|_{1}\geq\hat{\tau}\|p\|_{1}\geq\hat{\tau}\|p\|,
\end{align*}
where $\hat{\tau}:=\min\{\tau_{g_{e}},\tau_{g_{\iota}}\}$. It is
now straightforward to see that the above inequality yields inequality
\eqref{eq:gen_slater} with $\tau_{g}=\hat{\tau}/(1+2\gamma)$. Statement
(a) now follows from \eqref{eq:gen_slater} and the previous conclusion.
\end{proof}

\end{appendices}

\ACKNOWLEDGMENT{
The first author has been supported by the US Department of Energy (DOE) and UT-Battelle, LLC, under contract DE-AC05-00OR22725 and also supported by the Exascale Computing Project (17-SC-20-SC), a collaborative effort of the U.S. Department of Energy Office of Science and the National Nuclear Security Administration.
The second author was partially supported by ONR Grant N00014-18-1-2077 and AFOSR Grant FA9550-22-1-0088.
}

\typeout{}
\bibliographystyle{informs2014}
\bibliography{Proxacc_ref}

\def\cprime{$'$}
\begin{thebibliography}{43}
\providecommand{\natexlab}[1]{#1}
\providecommand{\url}[1]{\texttt{#1}}
\providecommand{\urlprefix}{URL }

\bibitem[{Aybat \protect\BIBand{} Iyengar(2011)}]{Aybatpenalty}
Aybat N, Iyengar G (2011) A first-order smoothed penalty method for compressed
  sensing. \emph{SIAM J. Optim.} 21(1):287--313.

\bibitem[{Aybat \protect\BIBand{} Iyengar(2012)}]{AybatAugLag}
Aybat N, Iyengar G (2012) A first-order augmented {{L}agrangian} method for
  compressed sensing. \emph{SIAM J. Optim.} 22(2):429--459,
  \urlprefix\url{http://dx.doi.org/10.1137/100786721}.

\bibitem[{Beck(2017)}]{beck2017first}
Beck A (2017) \emph{First-order methods in optimization} (SIAM).

\bibitem[{Beck \protect\BIBand{} Teboulle(2009)}]{beck2009fast}
Beck A, Teboulle M (2009) A fast iterative shrinkage-thresholding algorithm for
  linear inverse problems. \emph{SIAM J. Imaging Sci.} 2(1):183--202.

\bibitem[{Bertsekas(2016)}]{bertsekas2016nonlinear}
Bertsekas D (2016) \emph{Nonlinear programming} (Athena Scientific), "3"
  edition.

\bibitem[{Boob et~al.(2019)Boob, Deng, \protect\BIBand{}
  Lan}]{Lan-ConstrainedStocasticProxMetNonconvex2019}
Boob D, Deng Q, Lan G (2019) Stochastic first-order methods for convex and
  nonconvex functional constrained optimization. \emph{Available on
  arXiv:1908.02734} .

\bibitem[{Carmon et~al.(2018)Carmon, Duchi, Hinder, \protect\BIBand{}
  Sidford}]{Aaronetal2017}
Carmon Y, Duchi JC, Hinder O, Sidford A (2018) Accelerated methods for
  nonconvex optimization. \emph{SIAM J. Optim.} 28(2):1751--1772,
  \urlprefix\url{http://dx.doi.org/10.1137/17M1114296}.

\bibitem[{Dattorro \protect\BIBand{}
  Dattorro(2005)}]{Dattorro05convexoptimization}
Dattorro M, Dattorro J (2005) \emph{Convex Optimization \& Euclidean Distance
  Geometry} (Meeboo Publishing).

\bibitem[{Fletcher(2013)}]{fletcher2013practical}
Fletcher R (2013) \emph{Practical methods of optimization} (John Wiley \&
  Sons).

\bibitem[{Florea \protect\BIBand{} Vorobyov(2018)}]{florea2018accelerated}
Florea MI, Vorobyov SA (2018) An accelerated composite gradient method for
  large-scale composite objective problems. \emph{IEEE Transactions on Signal
  Processing} 67(2):444--459.

\bibitem[{Ghadimi \protect\BIBand{} Lan(2016)}]{nonconv_lan16}
Ghadimi S, Lan G (2016) Accelerated gradient methods for nonconvex nonlinear
  and stochastic programming. \emph{Math. Program.} 156:59--99, ISSN 1436-4646.

\bibitem[{Hajinezhad \protect\BIBand{} Hong(2019)}]{HongPertAugLag}
Hajinezhad D, Hong M (2019) Perturbed proximal primal–dual algorithm for
  nonconvex nonsmooth optimization. \emph{Math. Program.} 176:207--245.

\bibitem[{Hiriart-Urruty \protect\BIBand{} Lemarechal(1993)}]{Hiriart1}
Hiriart-Urruty J, Lemarechal C (1993) \emph{Convex Analysis and Minimization
  Algorithms I} (Berlin: Springer).

\bibitem[{Hong(2016)}]{ProxAugLag_Ming}
Hong M (2016) Decomposing linearly constrained nonconvex problems by a proximal
  primal dual approach: algorithms, convergence, and applications.
  \emph{Available on arXiv:1604.00543} .

\bibitem[{Jiang et~al.(2019)Jiang, Lin, Ma, \protect\BIBand{}
  Zhang}]{SZhang-Pen-admm}
Jiang B, Lin T, Ma S, Zhang S (2019) Structured nonconvex and nonsmooth
  optimization algorithms and iteration complexity analysis. \emph{Comput.
  Optim. Appl.} 72(3):115–157.

\bibitem[{Kong(2021)}]{KongThesis2021}
Kong W (2021) Accelerated inexact first-order methods for solving nonconvex
  composite optimization problems. \emph{arXiv:2104.09685} .

\bibitem[{Kong et~al.(2019{\natexlab{a}})Kong, Melo, \protect\BIBand{}
  Monteiro}]{WJRproxmet1}
Kong W, Melo JG, Monteiro RDC (2019{\natexlab{a}}) Complexity of a quadratic
  penalty accelerated inexact proximal point method for solving linearly
  constrained nonconvex composite programs. \emph{SIAM J. Optim.}
  29(4):2566--2593.

\bibitem[{Kong et~al.(2019{\natexlab{b}})Kong, Melo, \protect\BIBand{}
  Monteiro}]{WJRComputQPAIPP}
Kong W, Melo JG, Monteiro RDC (2019{\natexlab{b}}) An efficient adaptive
  accelerated inexact proximal point method for solving linearly constrained
  nonconvex composite problems. \emph{Comput. Optim. Appl.} 76(2):305--346.

\bibitem[{Kong et~al.(2020)Kong, Melo, \protect\BIBand{}
  Monteiro}]{RenWilmelo2020iteration}
Kong W, Melo JG, Monteiro RDC (2020) Iteration-complexity of an inner
  accelerated inexact proximal augmented {L}agrangian method based on the
  classical {L}agrangian function. \emph{arXiv preprint arXiv:2008.00562} .

\bibitem[{Kong et~al.(2021)Kong, Melo, \protect\BIBand{}
  Monteiro}]{fistaReport2021}
Kong W, Melo JG, Monteiro RDC (2021) {FISTA and Extensions - Review and New
  Insights}. \emph{Optimization Online} .

\bibitem[{Kong \protect\BIBand{} Monteiro(2021)}]{MinMax-RenWilliam}
Kong W, Monteiro RDC (2021) An accelerated inexact proximal point method for
  solving nonconvex-concave min-max problems. \emph{SIAM Journal on
  Optimization} 31(4):2558--2585.

\bibitem[{Lan \protect\BIBand{} Monteiro(2013)}]{LanRen2013PenMet}
Lan G, Monteiro RDC (2013) Iteration-complexity of first-order penalty methods
  for convex programming. \emph{Math. Program.} 138(1):115--139.

\bibitem[{Lan \protect\BIBand{} Monteiro(2016)}]{LanMonteiroAugLag}
Lan G, Monteiro RDC (2016) Iteration-complexity of first-order augmented
  {L}agrangian methods for convex programming. \emph{Math. Program.}
  155(1):511--547.

\bibitem[{Li et~al.(2020)Li, Chen, Liu, Lu, \protect\BIBand{}
  Xu}]{ImprovedShrinkingALM20}
Li Z, Chen PY, Liu S, Lu S, Xu Y (2020) Rate-improved inexact augmented
  {L}agrangian method for constrained nonconvex optimization. \emph{Available
  on arXiv:2007.01284} .

\bibitem[{Li \protect\BIBand{} Xu(2020)}]{HybridPenaltyAugLag19}
Li Z, Xu Y (2020) Augmented {L}agrangian based first-order methods for convex
  and nonconvex programs: nonergodic convergence and iteration complexity.
  \emph{arXiv e-prints} arXiv--2003.

\bibitem[{Lin et~al.(2019)Lin, Ma, \protect\BIBand{} Xu}]{PPmetNonconvex2019}
Lin Q, Ma R, Xu Y (2019) Inexact proximal-point penalty methods for non-convex
  optimization with non-convex constraints. \emph{Available on
  Arxiv:1908.11518} .

\bibitem[{Lin \protect\BIBand{} Xiao(2014)}]{pmlr-v32-lin14}
Lin Q, Xiao L (2014) An adaptive accelerated proximal gradient method and its
  homotopy continuation for sparse optimization. \emph{Proc. 31st Int. Conf.
  Mach. Learn.} 32:73--81.

\bibitem[{Liu et~al.(2019)Liu, Liu, \protect\BIBand{} Ma}]{ShiqiaMaAugLag16}
Liu Y, Liu X, Ma S (2019) On the nonergodic convergence rate of an inexact
  augmented {L}agrangian framework for composite convex programming.
  \emph{Math. Oper. Res.} 44(2):632--650.

\bibitem[{{Lu} \protect\BIBand{} {Zhou}(2018)}]{zhaosongAugLag18}
{Lu} Z, {Zhou} Z (2018) {Iteration-complexity of first-order augmented
  {L}agrangian methods for convex conic programming}. \emph{Available on
  arXiv:1803.09941} .

\bibitem[{Melo et~al.(2020)Melo, Monteiro, \protect\BIBand{}
  Wang}]{RJWIPAAL2020}
Melo JG, Monteiro RDC, Wang H (2020) Iteration-complexity of an inexact
  proximal accelerated augmented {L}agrangian method for solving linearly
  constrained smooth nonconvex composite optimization problems. \emph{Available
  on arXiv:2006.08048} .

\bibitem[{Necoara et~al.(2017)Necoara, Patrascu, \protect\BIBand{}
  Glineur}]{IterComplConicprog}
Necoara I, Patrascu A, Glineur F (2017) Complexity of first-order inexact
  {L}agrangian and penalty methods for conic convex programming. \emph{Optim.
  Methods Softw.} 1--31.

\bibitem[{Nesterov(2012)}]{nesterov2012gradient}
Nesterov Y (2012) Gradient methods for minimizing composite functions.
  \emph{Math. Program.} 1--37.

\bibitem[{Nocedal \protect\BIBand{} Wright(2006)}]{nocedal2006numerical}
Nocedal J, Wright S (2006) \emph{Numerical optimization} (Springer Science \&
  Business Media).

\bibitem[{Parikh \protect\BIBand{} Boyd(2014)}]{parikh2014proximal}
Parikh N, Boyd S (2014) Proximal algorithms. \emph{Foundations and Trends in
  optimization} 1(3):127--239.

\bibitem[{Patrascu et~al.(2017)Patrascu, Necoara, \protect\BIBand{}
  Tran-Dinh}]{Patrascu2017}
Patrascu A, Necoara I, Tran-Dinh Q (2017) Adaptive inexact fast augmented
  {L}agrangian methods for constrained convex optimization. \emph{Optim. Lett.}
  11(3):609--626.

\bibitem[{Rockafellar(1976{\natexlab{a}})}]{rockafellar1976augmented}
Rockafellar R (1976{\natexlab{a}}) Augmented {L}agrangians and applications of
  the proximal point algorithm in convex programming. \emph{Mathematics of
  operations research} 1(2):97--116.

\bibitem[{Rockafellar(1976{\natexlab{b}})}]{MR0418919}
Rockafellar RT (1976{\natexlab{b}}) Augmented {L}agrangians and applications of
  the proximal point algorithm in convex programming. \emph{Math. Oper. Res.}
  1(2):97--116, ISSN 0364-765X.

\bibitem[{Rockafellar \protect\BIBand{} Wets(1998)}]{VariaAna}
Rockafellar RT, Wets RJB (1998) \emph{Variational analysis} (Berlin: Springer),
  ISBN 3-540-62772-3, \urlprefix\url{http://opac.inria.fr/record=b1093869}.

\bibitem[{Sahin et~al.(2019)Sahin, Eftekhari, Alacaoglu, Latorre,
  \protect\BIBand{} Cevher}]{inexactAugLag19}
Sahin M, Eftekhari A, Alacaoglu A, Latorre F, Cevher V (2019) An inexact
  augmented {L}agrangian framework for nonconvex optimization with nonlinear
  constraints. \emph{Available on arXiv:1906.11357} .

\bibitem[{Xie \protect\BIBand{} Wright(2019)}]{xie2019complexity}
Xie Y, Wright S (2019) Complexity of proximal augmented {L}agrangian for
  nonconvex optimization with nonlinear equality constraints. \emph{arXiv
  preprint arXiv:1908.00131} .

\bibitem[{Xu(2019)}]{YangyangAugLag17}
Xu Y (2019) Iteration complexity of inexact augmented {L}agrangian methods for
  constrained convex programming. \emph{Math. Program.} ISSN 1436-4646,
  \urlprefix\url{http://dx.doi.org/10.1007/s10107-019-01425-9}.

\bibitem[{Zhang \protect\BIBand{}
  Luo(2020{\natexlab{a}})}]{ErrorBoundJzhang-ZQLuo2020}
Zhang J, Luo ZQ (2020{\natexlab{a}}) A global dual error bound and its
  application to the analysis of linearly constrained nonconvex optimization.
  \emph{Available on arXiv:2006.16440} .

\bibitem[{Zhang \protect\BIBand{}
  Luo(2020{\natexlab{b}})}]{ADMMJzhang-ZQLuo2020}
Zhang J, Luo ZQ (2020{\natexlab{b}}) A proximal alternating direction method of
  multiplier for linearly constrained nonconvex optimization. \emph{Available
  on arXiv:2006.16440} .

\end{thebibliography}

\end{document}